\documentclass[12pt,reqno]{amsart}
\usepackage{a4wide,enumerate,color}
\usepackage{amsmath}
\usepackage{bm} 
\usepackage{scrextend} 
\allowdisplaybreaks
\usepackage{hyperref}

\let\pa\partial
\let\na\nabla
\let\eps\varepsilon
\newcommand{\vk}[1]{\bm{v}^{(#1)}}

\newcommand{\D}{\mathbb{D}}
\newcommand{\N}{{\mathbb N}}
\newcommand{\R}{{\mathbb R}}
\newcommand{\diver}{\operatorname{div}}

\newcommand{\vcg}[1]{{\pmb #1}}
\newcommand{\tn}[1]{{\mathbb #1}}

\newcommand{\rhotot}{\rho}
\newcommand{\St}{\mathbb{S}}
\newcommand{\vel}{\bm{v}}

\newcommand{\J}{\bm{J}}
\newcommand{\B}{{\mathcal B}}
\newcommand{\Pro}{\mathrm{P}}

\newtheorem{theorem}{Theorem}[section]
\newtheorem{lemma}[theorem]{Lemma}

\newtheorem{remark}[theorem]{Remark}

\newtheorem{definition}[theorem]{Definition}

\numberwithin{equation}{section}

\begin{document}

\title[Compressible fluid model for mixtures]{
Existence analysis of a stationary \\
compressible fluid model for heat-conducting \\
and chemically reacting mixtures}

\author[M. Buli\v{c}ek]{Miroslav Buli\v{c}ek}
\address{M.B.: Charles University, Faculty of Mathematics and Physics, Sokolovsk\'a 83,
186 75 Praha 8, Czech Republic}
\email{mbul8060@karlin.mff.cuni.cz}

\author[A. J\"ungel]{Ansgar J\"ungel}
\address{A.J.: Institute for Analysis and Scientific Computing, Vienna University of
	Technology, Wiedner Hauptstra\ss e 8--10, 1040 Wien, Austria}
\email{juengel@tuwien.ac.at}

\author[M. Pokorn\'y]{Milan Pokorn\'y}
\address{M.P.: Charles University, Faculty of Mathematics and Physics, Sokolovsk\'a 83,
186 75 Praha 8, Czech Republic}
\email{pokorny@karlin.mff.cuni.cz}

\author[N. Zamponi]{Nicola Zamponi}
\address{N.Z.: University of Mannheim,
School of Business Informatics and Mathematics,
B6, 28, 68159 Mannheim, Germany}
\email{nzamponi@mail.uni-mannheim.de}

\date{\today}

\thanks{The authors acknowledge partial support from the bilateral Czech-Austrian
project of the Austrian Agency for International Cooperation in Education
and Research (\"OAD), grant CZ 15/2018, and the project no.\ 8J18AT018 financed by the Ministry of Education, Youth and Sports of Czech Republic.
The second author was partially supported by
the Austrian Science Fund (FWF), grants P30000, P33010, W1245, and F65. The work
of the third and the fourth author was partially supported by the grant of the
Czech Science Foundation no.\ 19-04243S. The fourth author acknowledges support from the Alexander von Humboldt Foundation}

\begin{abstract}
The existence of large-data weak solutions to a steady compressible
Navier--Stokes--Fourier system for chemically reacting fluid mixtures is proved.
General free energies are considered satisfying some structural assumptions,
with a pressure containing a $\gamma$-power law.
The model is thermodynamically consistent and contains the Maxwell--Stefan
cross-diffusion equations in the Fick--Onsager form
as a special case. Compared to previous works, a very general model class is
analyzed, including cross-diffusion effects, temperature gradients,
compressible fluids, and different molar masses.
A priori estimates are derived from the entropy balance and the total
energy balance. The compactness for the total mass density follows from
an estimate for the 
pressure in $L^p$ with $p>1$, the effective viscous
flux identity, and uniform bounds related to Feireisl's oscillations defect measure.
These bounds rely heavily on the convexity of the free energy and the strong convergence
of the relative chemical potentials.
\end{abstract}



\maketitle

\tableofcontents

\section{Introduction}

This paper is concerned with the mathematical analysis of a stationary
model for fluid mixtures, coupling Maxwell--Stefan-type diffusion with the
Navier--Stokes--Fourier model of compressible heat-conducting fluids.
A key feature of the present paper is that the fluid mixture is
modeled in a thermodynamically consistent way.
Compared to \cite{DDGG20}, we include
temperature effects, and compared to \cite{PiPo17,PiPo18}, our constitutive equations
are different and we allow for temperature gradients inside the diffusive fluxes, which yields a cross-diffusion coupling between the equations for the partial mass densities and the equation for the  energy. The most important issue which allows for better results than in previous papers for steady compressible models of chemically reacting mixtures is the convexity of the Helmholtz free energy, similarly as in \cite{DDGG20}, where an evolutionary model was studied, however, under the assumption that the temperature is constant. On the other hand, we neglect chemical reactions on the boundary which is an important issue in the aforementioned papers.

\subsection{Balance equations}

The stationary balance equations for the
partial mass densities $\rho_i$, the momentum $\rho\vel$ and the
total energy $\rho E$ are
\begin{align}
 \diver(\rho_i\vel + \J_i) &= r_i ,
\quad i=1,\ldots,N, \label{1.massbal} \\
 \diver(\rhotot\vel\otimes\vel - \St) + \na p &= \rhotot \bm{b}, \label{1.mombal} \\
 \diver\left(\rhotot E\vel + \bm{Q}- \St \vel + p \vel\right) &= \rhotot \bm{b}\cdot \vel,
\label{1.totEbal}
\end{align}
where the physical meaning of the variables is explained in Table \ref{table}.

\begin{center}
	\begin{table}[thb]
		\begin{tabular}{|l|l|}
			\hline
			Variable & Physical meaning \\ \hline
			$\rho_i$ & partial mass density of the $i$th species \\
			$\rhotot = \sum_{i=1}^N\rho_i$ & total mass density \\
			$m_i$ & molar mass of the $i$th species \\
			$n_i$ & number density, $n_i=\rho_i/m_i$ \\
			$\overline{n}$ & total number density, $\overline{n}=\sum_{i=1}^N n_i$ \\
			$\J_i$ & partial flux of the $i$th species \\
			$\vel$ & barycentric velocity \\
			$r_i$ & reaction term for the $i$th species \\
			$p$ & pressure \\
			$\St$ & viscous stress tensor \\
			$\bm{b}$ & momentum force term \\
			$\bm{\mathcal{J}}$ & entropy flux\\
			$\Xi$ & entropy production\\
			$\bm{Q}$ & internal heat flux \\
			$\mu_i$ & chemical potential of the $i$th species \\
			$\theta$ & thermodynamic temperature \\
			$s$ & specific entropy \\
			$e$ & specific internal energy \\
			$E = e + |\vel|^2/2$ & specific total energy\\
			$\psi$ & specific Helmholtz free energy \\
			\hline
		\end{tabular}
		\vskip2mm
		\caption{Physical meaning of the variables.}\label{table}
	\end{table}
\end{center}

Equations \eqref{1.massbal}--\eqref{1.totEbal}
are solved in a bounded domain $\Omega\subset\R^3$ and
are supplemented with the following boundary conditions on $\partial \Omega$:
\begin{align}
  & \J_i\cdot\bm{\nu} = 0, \quad i=1,\ldots,N, \label{bc.J} \\
  & \vel\cdot\bm{\nu} = 0, \quad
	(\mathbb{I}-\bm{\nu}\otimes\bm{\nu})(\mathbb{S}\bm{\nu} + \alpha_1\vel) = \bm{0},
	\label{bc.vel} \\
	& \bm{Q}\cdot\bm{\nu} + \alpha_2(\theta_0-\theta) = 0, \label{bc.theta}
\end{align}
where  $\alpha_1 \geq 0$ and $\alpha_2 >0$ are constants, $\theta_0 \in L^1(\partial \Omega)$ is a strictly positive function, and
$\bm{\nu}$ is the exterior unit normal vector to $\pa\Omega$. The matrix
$\mathbb{I}-\bm{\nu}\otimes\bm{\nu}$ projects onto the orthogonal complement
of $\operatorname{span}\{\bm{\nu}\}$.
Equations \eqref{bc.J} are no-flux boundary conditions,
\eqref{bc.vel} are the Navier slip boundary conditions (partial if $\alpha_1>0$, complete for $\alpha_1=0$), and
\eqref{bc.theta} means that the normal component of the
heat flux $\bm Q$ is proportional to the temperature
difference $\theta-\theta_0$, where $\theta_0$ has the physical meaning of the outer 
temperature.

We also prescribe the total mass
\begin{equation}\label{1.int}
  \frac{1}{|\Omega|}\int_\Omega\rhotot\, dx = \overline\rho,
\end{equation}
where $|\Omega|$ denotes the measure of $\Omega$, and $\overline{\rho}$ is assumed to be positive. Note that the total mass is in fact $|\Omega| \overline{\rho}$ and the quantity $\overline \rho$ has the meaning of the average density.

In some situations, especially when the integrability of the density is low, we replace the total energy balance by the entropy inequality
\begin{equation}
-\diver\bm{\mathcal J} + \Xi \leq 0.\label{1.entrbal}
\end{equation}
Note that we replaced an equality by an inequality. Hence, we may obtain too many solutions to our problem which are non-physical. However, due to mathematical reasons (and physically, it is not surprising either), as explained below,  we cannot expect equality for the balance of entropy. To avoid this problem, similarly as for the single-component steady Navier--Stokes--Fourier system (see, e.g., \cite{JNP14}), we also add the total energy balance \eqref{1.totEbal} integrated over $\Omega$,
\begin{equation} \label{1.totEint}
 \int_{\pa\Omega}\alpha_1|\vel|^2\, ds
= \int_{\pa\Omega}\alpha_2(\theta_0-\theta)\,dx
+ \int_\Omega\rhotot\bm{b}\cdot\vel\, dx.
\end{equation}
We will discuss this issue later; at this point, we just note that \eqref{1.entrbal} together with \eqref{1.totEint} possesses the property of weak-strong compatibility, i.e.,  any smooth solution to \eqref{1.massbal}, \eqref{1.mombal}, \eqref{1.entrbal}, and \eqref{1.totEint} is in fact a smooth solution to \eqref{1.massbal}--\eqref{1.totEbal}.

\subsection{Notation}

We denote by a bold letter a vector $\bm{a}\in\R^3$ with
components $(a_1,a_2,a_3)$, by a blackboard bold letter a matrix
$\mathbb{A}\in\R^{m\times n}$ with coefficients $(A_{ij})$ or
$(\mathbb{A}_{ij})$, and by $\vec{a}$ a vector $(a_1,\ldots,a_N)$ in $\R^N$. Recall that $N$ denotes the number of species in equation \eqref{1.massbal}.
The unit matrix in $\R^{m\times m}$, where $m\in\N$, is denoted by $\mathbb{I}$.
Given two matrices $\mathbb{A}$, $\mathbb{B}\in \R^{m\times n}$, we define
$\mathbb{A}:\mathbb{B} \equiv \sum_{i=1}^{m}\sum_{j=1}^{n}A_{ij}B_{ij}$.
We also set  $\R_+=(0,\infty)$ and $\R_{+,0}=[0,\infty)$.


\subsection{Constitutive equations}

We specify the entropy flux $\bm{\mathcal{J}}$, entropy production $\Xi$,
viscous stress tensor $\St$, heat flux $\bm{Q}$, diffusion fluxes $\J_i$, and reaction terms $r_i$.

The {\bf entropy flux} $\bm{\mathcal{J}}$ is the sum of the contributions of free transport, diffusion fluxes, and heat flux,
\begin{equation*}
  \bm{\mathcal{J}} = \rhotot s\vel - \sum_{i=1}^N\frac{\mu_i}{\theta}\J_i
	+ \frac{1}{\theta}\bm Q.
\end{equation*}

The {\bf entropy production} $\Xi$ keeps into account the contributions from the diffusion fluxes, the heat flux, the viscous stress and the reaction terms,
\begin{equation*}
  \Xi = -\sum_{i=1}^N\J_i\cdot\nabla\frac{\mu_i}{\theta} +
  \bm Q \cdot \nabla\frac{1}{\theta} + \frac{\St : \nabla\vel}{\theta}
  -\sum_{i=1}^N r_i\frac{\mu_i}{\theta}.
\end{equation*}
According to the second law of thermodynamics \cite{Kre81}, the entropy production
$\Xi$ must be nonnegative. This is achieved by requiring that
\begin{align}\label{entr.prod.nn}
-\sum_{i=1}^N\J_i\cdot\nabla\frac{\mu_i}{\theta} +
\bm Q \cdot \nabla\frac{1}{\theta}\geq 0,\quad
\frac{\St : \nabla\vel}{\theta}\geq 0,\quad
-\sum_{i=1}^N r_i\frac{\mu_i}{\theta}\geq 0.
\end{align}
The following definitions of $\St$, $\bm{Q}$, $\bm{J}_i$, and $r_i$ are chosen
in such a way that these requirements are satisfied.

The {\bf viscous stress tensor} is assumed to be a linear function of the
symmetric velocity gradient,
\begin{equation}\label{1.S}
  \St = 2\lambda_1(\theta)\bigg(\mathbb{D}(\vel) - \frac13\diver(\vel)\mathbb{I}\bigg)
	+ \lambda_2(\theta)\diver(\vel)\mathbb{I},
\end{equation}
where $\lambda_1(\theta)$ and $\lambda_2(\theta)$ are the temperature-dependent
shear and bulk viscosity coefficients, respectively.

The {\bf heat flux} consists of the Fourier law and the molecular diffusion term,
\begin{equation}\label{1.Q}
  \bm{Q} = -\kappa(\theta)\na\theta - \sum_{i=1}^N M_i\na\frac{\mu_i}{\theta},
\end{equation}
where $\kappa(\theta)$ is the thermal conductivity and the coefficients
$M_i$ depend on $\vec\rho$ and $\theta$. The molecular diffusion term plays
an important role in the energy identity.

The {\bf diffusion flux} is a
linear combination of the thermodynamic forces $\na(\vec\mu/\theta)$
and $\na(1/\theta)$,
\begin{equation}\label{1.J}
  \J_i = -\sum_{j=1}^N M_{ij}\na\frac{\mu_j}{\theta} - M_i\na\frac{1}{\theta},
	\quad i=1,\ldots,N,
\end{equation}
where $M_{ij}=M_{ij}(\vec\rho,\theta)$ are diffusion coefficients.
By Onsager's principle, the coefficient matrix $(M_{ij})$ is symmetric \cite{Kre81}.

To fulfill the mass conservation equation $\diver(\rhotot\vel)=0$,
the sum of the diffusion fluxes and the sum of the reaction terms should vanish, i.e.\
$\sum_{i=1}^N\J_i=\bm{0}$, $\sum_{i=1}^N r_i=0$.
Hence, the diffusion matrix has a nontrivial kernel, and we assume that
\begin{equation}\label{1.M1}
  \sum_{i=1}^N M_{ij} = \sum_{i=1}^N M_i = 0, \quad j=1,\ldots,N.
\end{equation}
To be consistent with the first inequality in \eqref{entr.prod.nn} and relations
\eqref{1.Q} and \eqref{1.J}, we assume that $\kappa(\theta) >0$ and that the mobility matrix $(M_{ij})$ is nonnegative. More precisely,
\begin{equation}\label{1.M2}
\exists C_M>0 : \quad
  \sum_{i,j=1}^N M_{ij}z_iz_j \ge C_M|\Pi\vec{z}|^2\quad\mbox{for all }\vec{z}\in\R^N,
\end{equation}
where $\Pi = \mathbb{I}-\vec{1}\otimes\vec{1}/N$ is the orthogonal projector
on $\operatorname{span}\{\vec{1}\}^\perp$.
We refer to \eqref{1.J} as the Maxwell--Stefan fluxes in Fick--Onsager form.
In Maxwell--Stefan systems, the driving forces are usually given by linear
combinations of the diffusion fluxes. This formulation can be written as \eqref{1.J}
for a specific choice of $M_{ij}$ and $M_i$ \cite{HeJu20}.
For the structure of the mobility matrix and its connection to the
Maxwell--Stefan theory, we also refer to \cite{BoDr15}.

According to the third inequality in \eqref{entr.prod.nn}, the entropy production due to reactions, $-\sum_{i=1}^N r_i \mu_i/\theta $, must be nonnegative. Therefore, we suppose for the {\bf reaction terms} that
\begin{equation}\label{1.r}
\begin{cases}
  r_i = r_i(\Pi(\vec\mu/\theta), \theta),\quad i=1,\ldots,N,\\
  \exists C_r>0,\,\zeta>0,\,\beta>0 : \quad
  -\sum_{i=1}^N r_i(\Pi(\vec{q}),\theta)q_i \ge C_r|\Pi\vec q|^2, \\
  |r_i(\Pi(\vec{q}),\theta)| \leq C_r(|\Pi\vec{q}|^{5(6-\zeta)/6}
	+ \theta^{5(3\beta-\zeta)/6}) \quad\mbox{for all }\vec{q}\in\R^N, \theta>0,
\end{cases}
\end{equation}
for some $\zeta >0$ which is possibly very small, 
$\beta>0$, and all $i =\{1,2,\dots,N\}$.
This condition is satisfied for the reaction terms used in \cite{DDGG20},
\begin{equation}\label{1.exr}
  r_i = -\sum_{j=1}^l\frac{\pa\Phi}{\pa D_j}(D^R)\gamma_i^j, \quad\mbox{where }
	D_j^R = \sum_{k=1}^N \gamma_k^j\frac{\mu_k}{\theta}, \
	i=1,\ldots,N,\ j=1,\ldots,l,
\end{equation}
where $\Phi$: $\R^l\to\R$ is a convex potential with suitable growth
and $\vec{\gamma}^j\in\R^N$ is a vectorial
coefficient associated to the $j$th reaction; see Remark \ref{rem.reac}.

The remaining variables---{\bf chemical potentials}~$\mu_i$, {\bf pressure}~$p$,
{\bf total internal energy}~$e$, and {\bf entropy}~$s$---are determined from the
Helmholtz free energy density $\rhotot\psi$ which is assumed to be a function of
$\vec\rho$ and $\theta$:
\begin{equation}\label{1.p}
\begin{aligned}
  & \mu_i = \frac{\pa(\rhotot\psi)}{\pa\rho_i}, \quad
	p = -\rhotot\psi + \sum_{i=1}^N\rho_i\mu_i, \\
  & \rhotot e = \rhotot\psi - \theta\frac{\pa(\rhotot\psi)}{\pa\theta}, \quad
	\rhotot s = -\frac{\pa(\rhotot\psi)}{\pa\theta}.
\end{aligned}
\end{equation}
The definition of the pressure is known as the Gibbs--Duhem relation.
In this paper, we allow for general free energies satisfying Hypothesis (H6) below.
A specific example is given in Remark \ref{rem.energy}.

For the mathematical analysis, we rename the {\bf free energy density} by writing
$h_\theta(\vec\rho)=(\rhotot\psi)(\vec\rho,\theta)$ and denote by
\begin{equation*} 
  h_\theta^*(\vec\mu) =  \sup_{\vec\rho\in\R_+^N}(\vec\rho\cdot\vec\mu
	- h_\theta(\vec\rho))
\end{equation*}
the Legendre transform of $h_\theta$, which in fact equals the pressure $p$.
It is well defined on
$D_\theta^* = \{\vec\mu\in\R^N:\exists\vec\rho\in\R_+^N:
\vec\mu=\na h_\theta(\vec\rho)\}$ \cite[\S26]{Roc70}.
If $h_\theta\in C^2(\R_+^N)$ depends
smoothly on $\theta$, then $h_\theta^*\in C^2(D_\theta^*)$ also depends smoothly
on $\theta$. Moreover, for any $\vec\mu\in D^*_\theta$ and $\vec\rho\in\R_+^N$
\cite[Theorem 26.5]{Roc70},
$$
  \vec\mu = \na h_\theta(\vec\rho) \quad\mbox{if and only if}\quad
	\vec\rho = \na h_\theta^*(\vec\mu).
$$
Note that $\nabla h_\theta(\vec{\rho})$ means
$\na_{\vec{\rho}}h_\theta(\vec{\rho})$,
i.e., the derivatives are taken with respect to $\rho_i$, $i=1,2,\dots, N$, 
while $\nabla h^*_\theta(\vec{\mu})$ means $\na_{\vec{\mu}}h_\theta^*(\vec{\mu})$. 
We discuss the properties of $h^*_{\theta}$ in~Lemma~\ref{lem:H6}, 
and it will be shown that it is well defined under Hypothesis (H6) below.


\subsection{Hypotheses}

Before formulating the main hypotheses, we introduce three important positive 
parameters, $\gamma,\beta,\nu>0$, that are supposed to be fixed from now and that 
are used in the whole paper only in the context of the following assumptions. 
The parameter $\gamma$ describes the growth of the pressure with respect to 
the density and we assume that
$$
  \gamma > \frac32.
$$
The parameter $\beta>0$ describes the growth of the heat conductivity with respect 
to the temperature, and the parameter $\nu>0$ is related to $\gamma$ and $\beta$ via
\begin{align}\label{def.nu}
\nu := \gamma\min\left\{\frac{2\gamma-3}{\gamma},
	\frac{3\beta-2}{3\beta+2}\right\}.
\end{align}
The constant $\nu$ is the improvement of the bound for the pressure by the 
Bogovskii estimate, see Lemma \ref{lem.dens}.

We impose the following mathematical hypotheses:

\begin{labeling}{(A44)}
\item[(H1)] {\it Domain:} $\Omega\subset\R^3$ is a bounded domain
with a $C^2$ boundary that is not axially symmetric.

\item[(H2)] {\it Data:} $\alpha_1 \geq 0$, $\alpha_2 >0$,
$\theta_0\in L^1(\partial \Omega)$,
$\mbox{ess}\inf_{\partial \Omega}\theta_0>0$, $\bm{b}\in L^\infty(\Omega;\R^3)$.

\item[(H3)] {\it Viscosity and heat conductivity:} $\lambda_1$, $\lambda_2$, $\kappa\in C^0(\R_+)$ and there exist constants $c_1$, $c_2$, $\kappa_1$, $\kappa_2$, 
such that for all $\theta>0$,
$$
\begin{aligned}
  c_1(1+\theta)&\le \lambda_1(\theta)\le c_2(1+\theta), \\
	0\leq \lambda_2(\theta) &\le c_2 (1+\theta), \\
	\quad \kappa_1(1+\theta)^\beta&\le \kappa(\theta)\le\kappa_2(1+\theta)^\beta.
	\end{aligned}
$$

\item[(H4)] {\it Diffusion coefficients:} There exists $\zeta>0$ such that for all 
$i,j=1,\ldots,N$, the coefficients $M_{ij}$, $M_i\in C^0(\R_{+,0}^N\times\R_{+})$ 
satisfy \eqref{1.M1}, \eqref{1.M2}, and
$$
  |M_{ij}(\vec\rho,\theta)|+\frac{|M_i(\vec\rho,\theta)|}{\theta}
	\le \widetilde{C}_M ( \rho^{(\gamma+\nu-\zeta)/3} + \theta^{(3\beta-\zeta)/3} + 1)
$$
for all $(\vec\rho,\theta)\in\R_+^N\times\R_+$ and some constant $\widetilde{C}_M$.

\item[(H5)] {\it Reaction terms:} $\vec{r}=(r_1,\ldots,r_N)\in C^0(\R^N\times\R_+;\R^N)$
satisfies \eqref{1.r} and $\sum_{i=1}^Nr_i=0$.

\item[(H6)] {\it Free energy density:} For any fixed $\theta>0$,
$h_\theta\in C^2(\R_+^N)$ is a strictly convex function with respect to $\vec{\rho}$. Furthermore, we specify the asymptotic behavior near zero and infinity as well as 
some kind of uniform convexity:
\begin{itemize}
\item[(H6a)] For all $\theta>0$, there holds that
\begin{align}
\begin{aligned}
0&=\lim_{\rhotot\to 0_+} h_{\theta}(\vec{\rho})=\lim_{\rhotot \to 0} \nabla h_{\theta}(\vec{\rho})\cdot \vec{\rho} \qquad \textrm{and} \qquad \infty= \lim_{\rhotot \to \infty} \frac{ h_{\theta}(\vec{\rho})}{\rhotot}.\label{LMB1}
\end{aligned}
\end{align}
In addition, we require that for all $\theta\in (0,\infty)$ and $R>0$,
\begin{align}\label{LMB2}
\limsup_{\rho_i \to 0_+} \left(\sup_{\{\rho_j\in (0,R); \, j\neq i\}}\frac{\partial h_{\theta}}{\partial \rho_i}(\rho_1,\ldots, \rho_N)\right)&=-\infty  \qquad \textrm{for all } i\in \{1,\ldots,N\},
\end{align}
and we assume that there exists a positive constant $K$ such that for all $\theta\in \mathbb{R}_+$, $\vec{\rho}\in \mathbb{R}^N_+$, and $\vec{x}\in \mathbb{R}^N$,
\begin{align}
 \sum_{i,j=1}^N\frac{\partial^2 h_{\theta}}{\partial \rho_i \partial \rho_j}(\vec{\rho})x_i x_j&>\frac{K \theta}{\rho} |\vec{x}|^2.\label{LMB21}
\end{align}
Next, we assume that for all $\kappa\in (0,1)$, there exists a constant $k>0$ 
such that for all $\theta\in (\kappa, \kappa^{-1})$ and 
$\vec{\rho}\in \mathbb{R}^N_+$ fulfilling $\rho\in (0,\kappa^{-1})$,
\begin{equation}\label{AUXc}
\left|(\nabla^2 h_{\theta}(\vec{\rho}))^{-1} \frac{\partial \nabla h_{\theta}}{\partial \theta}(\vec{\rho}) \right|\le k.
\end{equation}
Furthermore, we require that for every $\kappa\in (0,1)$, there exists $C>0$ s
uch that for all $\theta\in (\kappa,\kappa^{-1})$ and $\vec{\rhotot}\in\R_+^N$,
\begin{equation} \label{H7C}
  \rho_i \in (\kappa,\kappa^{-1}) \mbox{ implies that }
	\mu_i:= \frac{\pa h_\theta}{\pa\rho_i}(\vec{\rho}) \ge -C.
\end{equation}
\item[(H6b)] There exists $\frac{1}{2}<\alpha_0<1$, $\beta_0\geq 0$ such that
for every $\kappa\in (0,1)$, there exists $k>0$ such that
\begin{equation} \label{H6.b}
\theta \in (\kappa,\kappa^{-1}) \mbox{ implies that }
\sum_{i=1}^N\rho_i^{\alpha_0} \bigg|\frac{\pa^2 h_\theta}{\pa\theta\pa\rho_i}
(\vec\rho)\bigg|\leq k(1 + |\vec\rho|^{\beta_0}).
\end{equation}
In addition, we suppose the following growth conditions. If $\gamma >\frac 32$
and $\beta > \frac 23$, we assume that for all $\vec\rho\in\R_{+,0}^N$, $\theta>0$,
\begin{equation} \label{grow_h_2}
  \bigg|\frac{\pa h_{\theta}}{\pa\theta}(\vec{\rhotot})\bigg|
  \leq C_h(1+ \rhotot^{(5\beta-2)(\gamma+\nu-\zeta)/(6\beta)}
  + \theta^{5(3\beta-\zeta)/6-1}+ |\ln\theta|^a).
\end{equation}
If $\gamma >5/3$ and $\beta>1$, we replace \eqref{grow_h_2} by
\begin{equation} \label{grow_h_1}
\begin{aligned}
  |h_{\theta}(\vec{\rhotot})|  &\leq C_h (1+ \rhotot^{5(\gamma+\nu-\zeta)/6}
  + \theta^{5(3\beta-\zeta)/6}), \\
  \bigg|\frac{\pa h_{\theta}}{\pa\theta}(\vec{\rhotot})\bigg|
  & \leq C_h(1+ \rhotot^{5(\gamma+\nu-\zeta)/6}
  + \theta^{5(3\beta-\zeta)/6} + |\ln\theta|^a),
\end{aligned}
\end{equation}
for some constants $C_h$, $\zeta$, $\nu>0$, $a< 5$; $\nu$ and $\zeta$
are as above.
\end{itemize}

\item[(H7)] We assume that there exists $\omega\in (1,\gamma)$, $\tilde{c}_p>0$ and $\tilde{C}_p>0$ such that for all $\theta\in \mathbb{R}_+$,  
$\vec{\rho}\in \mathbb{R}^N_+$, and $\vec{x}\in \mathbb{R}^N$,
\begin{equation*}
    \tilde{c}_p|\vec{x}|^2 \left(\frac{\theta }{\rho} + \rho^{\gamma-2}\right)\le \sum_{i,j=1}^N \frac{\partial^2 h_{\theta}}{\partial \rho_i \partial \rho_j}(\vec{\rho})x_i x_j\le \tilde{C}_p |\vec{x}|^2 \left(\frac{\theta }{\rho} + \rho^{\gamma-2}+\rho^{\omega-2}\right).
\end{equation*}
\end{labeling}

We prove in Lemma \ref{lem:H6} that it follows from Hypotheses (H6)--(H7) that 
the {\it pressure} $p$, defined by \eqref{1.p}, satisfies for some 
$c_p$, $C_p>0$, depending only on $\tilde{c}_p$, $\tilde{C}_p$, $N$, and $\gamma$, 
the following inequalities:
$$
  c_p(\rhotot\theta + \rhotot^\gamma)\le p(\vec\rho,\theta)
	\le C_p(1+\rhotot\theta+\rhotot^\gamma)
$$
for all $(\vec\rho,\theta)\in\R_{+,0}^N\times\R_+$.

\begin{remark}[Discussion of the hypotheses]\rm \label{rem.hypo}
\begin{labeling}{(A44)}
\item[(H1)] If we approximate $\Omega$ by smooth domains, it is possible to extend our existence result to less regular domains as, e.g., Lipschitz ones. As this would technically complicate the paper, we skip this and only point out the paper \cite{Novo}, where such a technique is used in the case of compressible Navier--Stokes equations.
\item[(H2)] The force term $\bm{b}$ is assumed to be dependent only on $x$ just for
simplicity. It may also depend on $(\vec\rho,\theta)$; then $\bm{b}(x,\vec\rho,\theta)$
needs to be measurable in the first variable, continuous in the last $N+1$ variables, and bounded. We may assume that $\alpha_1=0$ (but not
$\alpha_2=0$); then we need that $\Omega$ is not axially symmetric
to apply the Korn inequality \cite[(4.17.19)]{NoSt04}. In case $\alpha_1>0$, we have the Korn inequality also for axially symmetric domains, but the estimate of the velocity gradient depends additionally on the tangential trace of the velocity. However, as shown below in Lemma \ref{lem.est1}, the estimate of the trace of $\vel$ depends also on the density and therefore, the problem becomes slightly more complex and leads to additional restrictions on the exponents $\beta$ and $\gamma$ (cf. \cite{JNP14}). To avoid such problems, we prefer to assume that the domain is not axially symmetric.
\item[(H3)] The linear growth of the viscosities leads to optimal a priori estimates.
It is known that for gases, the viscosity is an increasing function of the temperature and indeed, we could assume another polynomial growth, however, at the price of a 
significantly complicated proof; see, e.g., \cite{KrNePo} for a partial result in this direction in the case of single constituent fluid.
Finally, the lower bound avoids degeneracies in the coefficients.
\item[(H4)] The growth assumption on $M_{ij}$, $M_i$ is necessary to have strong relative compactness $M_{ij}(\vec\rho_\delta,\theta_\delta)$, $M_i(\vec\rho_\delta,\theta_\delta)$ in $L^3(\Omega)$ (for a suitable subsequence of $(\vec\rho_\delta,\theta_\delta)$).
\item[(H5)] We show in Remark \ref{rem.reac} that the reaction terms \eqref{1.exr}
satisfy condition \eqref{1.r}.
This condition excludes vanishing reaction terms. In fact, it is needed
to derive an $L^2(\Omega)$ bound for $\Pi\vec q$ and, together with \eqref{1.M2},
an $H^1(\Omega)$ bound for $\Pi\vec q$ (see Lemma \ref{lem.est1}).
Condition \eqref{1.r} can be replaced by
a Robin-type boundary condition for the diffusion fluxes
involving the chemical potentials, as done in \cite[Hypothesis (H8)]{BPZ17}.
\item[(H6)] This hypothesis is split onto two parts. Part (H6a) deals in some sense with almost minimal natural structural assumptions on the free energy. They are used to show that $\na_{\vec{\rhotot}} h_{\theta}(\vec{\rhotot})$ is a strictly monotone invertible mapping  from $\R^N_+$ onto $\R^N$. Moreover, we prescribe a lower bound on the convexity of the free energy with respect to the density, which describes at least the behavior of an ideal gas. We emphasize that the convexity is a quite natural assumption in order to fulfill the second law of thermodynamics. A more restrictive assumption on the convexity of the free energy is stated in Hypothesis (H7).  

The second part (H6b) is a rather technical assumption, and we believe that it could be avoided in the analysis. It is remarkable that (H6) is satisfied by the prototypical example of a free energy from Remark \ref{rem.energy} presented below. From a mathematical view point, it is more natural to impose suitable assumptions on the convex conjugate $h^*_{\theta}$, namely that it fulfills \eqref{H6A} and \eqref{H6C}, which are just consequences of (H6), as shown in Lemma~\ref{lem:H6}. However, Hypothesis (H6) is stated for the free energy $h_{\theta}$ instead of $h^*_{\theta}$ for the sake of readability of the paper. 
\item[(H7)] We stated the assumption on the lower and upper growth of the second derivatives of $h_{\theta}$ just to formulate all assumptions in terms of the free energy. Nevertheless, these growth assumptions are almost equivalent to the estimates on the pressure also mentioned in (H7) and then proved in Lemma~\ref{lem:H6}. Below, we present an example of the free energy for which
$p(\vec\rho,\theta)=\overline{n}\theta + (\gamma-1)\overline{n}^\gamma$. 
This pressure satisfies Hypothesis (H7). The condition $\gamma>3/2$ is needed to derive a bound of $\rho$ in $L^{\gamma +\nu}(\Omega)$; see Lemma \ref{lem.dens}. This allows us to control the pressure in a better space than $L^1(\Omega)$; see Lemma \ref{lem.p}. In the case of a single-constituent fluid, the term $\overline{n}^\gamma$ is usually called a``cold pressure'', as it represents the contribution of the pressure for very low temperatures. There exists a physical explanation for this term, however, this is beyond the scope of this paper.
Therefore, we consider this term rather as a mathematical tool to obtain estimates than a physically relevant term. In fact, the whole mathematical theory of weak solutions for compressible fluids is based on the presence of such a term.
\end{labeling}
\end{remark}

\begin{remark}[Example of a free energy]\label{rem.energy}\rm
An example of the free energy density fulfilling Hypothesis (H6) is given by
\begin{equation}\label{1.fe}
  \rhotot\psi = \theta\sum_{i=1}^N n_i\log n_i + \overline{n}^\gamma
	- c_W\rhotot\theta\log\theta,
\end{equation}
where $\gamma>1$ is some exponent and $c_W>0$ is the heat capacity;
see Appendix \ref{app.energy} for the proof.
The first term is related to the mixing of the components, the second one is
needed for the mathematical analysis (to obtain an estimate for the
total mass density; a certain physical justification of this term can be found in
\cite[Chapter 1]{FeNo09}),
and the third one describes the thermal energy. We have assumed
that the specific volumes of all components are the same. In the paper
\cite[Formula (20)]{DDGG20}, the first term is defined in a different way:
$$
  \theta \overline{n}\sum_{i=1}^N\frac{n_i}{\overline{n}}\log\frac{n_i}{\overline{n}}
	= \theta\sum_{i=1}^N n_i\log n_i - \theta \overline{n}\log \overline{n}.
$$
This expression does not contribute
to the pressure, while our mixing entropy term gives the pressure contribution
$n\theta$ of an ideal gas. With the free energy density \eqref{1.fe}, we obtain
\begin{align*}
  \mu_i &= \frac{\theta}{m_i}(\log n_i+1) + \frac{\gamma}{m_i}\overline{n}^{\gamma-1}
	- c_W\theta\log\theta, \\
	p &= \overline{n}\theta + (\gamma-1)\overline{n}^\gamma, \\
  \rhotot e &= \overline{n}^\gamma + c_W\rhotot\theta, \\
	\rhotot s &= -\sum_{i=1}^N n_i\log n_i + c_W\rhotot(\log\theta+1).
\end{align*}
The first term of the pressure $p$ represents the ideal gas law, while the second 
term is the so-called ``cold pressure'', which was already discussed in 
Remark \ref{rem.hypo}.
\end{remark}

\begin{remark}[Example of reaction terms]\label{rem.reac}\rm
We claim that \eqref{1.exr}
satisfies \eqref{1.r}. Here, $\Phi\in C^1(\R_+)$ is a uniformly convex
potential such that $\Phi(0)=0$ and $\na\Phi(0)=0$. The vectors
$\vec{\gamma}^1,\ldots,\vec{\gamma}^l\in\R$ satisfy $\sum_{i=1}^N\gamma_i^j=0$
for $j=1,\ldots,l$, and $\operatorname{span}\{\vec{\gamma}^1,\ldots,\vec{\gamma}^l\}
= \operatorname{span}\{\vec{1}\}^\perp$.
Indeed, set $\vec{q}=\vec\mu/\theta$. Note that due to $\sum_{i=1}^N\gamma_i^j=0$, the functions $r_i$ in fact depend only on $\Pi(\vec{q})$. A Taylor expansion around $D^R=0$ yields
$$
  -\sum_{i=1}^N r_iq_i
	= \sum_{j=1}^l\frac{\pa\Phi}{\pa D_j^R}(D^R)D^R_j
	\ge \Phi(D^R) \ge c|D^R|^2 = \frac{C}{\theta^2}
	\sum_{k=1}^l|\vec{\mu}\cdot\vec{\gamma}^k|^2,
$$
where $c>0$ is the convexity constant for $\Phi$. Since $(\vec{\gamma}^k)_{k=1}^N$
spans $\operatorname{span}\{\vec{1}\}^\perp$ and
$\Pi(\vec{q})$ lies in $\operatorname{span}\{\vec{1}\}^\perp$,
there exists another constant $C>0$ such that
$$
  \frac{1}{\theta^2}\sum_{k=1}^N|\vec{\mu}\cdot\vec{\gamma}^k|^2
	= \sum_{k=1}^N|\vec{q}\cdot\vec{\gamma}^k|^2
	\ge C|\Pi(\vec{q})|^2.
$$	
We infer that $-\sum_{i=1}^N r_i\mu_i/\theta \ge c|\Pi(\vec{q})|^2$, proving the claim.

The definition of $r_i$ differs
slightly from that one in \cite{DDGG20} because of the role played by the
temperature. In fact, the entropy inequality provides an estimate
for $\vec{\gamma}^j\cdot\vec{\mu}/\theta$, which in turn yields an $L^2(\Omega)$
bound for $\Pi(\vec\mu/\theta)$ (see Lemma \ref{lem.est1}).
We require that $\sum_{i=1}^N\gamma_i^j=0$ for all $j=1,\ldots,l$ in order to
achieve $\sum_{i=1}^N r_i=0$, needed for mass conservation.
As a consequence, $\vec{\gamma}^j\in\operatorname{span}\{\vec{1}\}^\perp$.
We assume that the linear hull of all $\vec{\gamma}^j$ is in fact
 equal to $\operatorname{span}\{\vec{1}\}^\perp$, which implies that $l\ge N-1$.
This condition is necessary since
Theorem 11.3 in \cite{DDGG20}, which gives an $L^2(\Omega)$ estimate for
$\Pi(\vec\mu/\theta)$, cannot be generalized in a straightforward way to
the non-isothermal case.
\end{remark}

For all $p\in [1,\infty]$, we introduce the following spaces:
$$
W^{1,p}_{\bm{\nu}}(\Omega;\R^3)
= \{\bm{u}\in W^{1,p}(\Omega;\R^3): \bm{u}\cdot\bm{\nu}=0\mbox{ on }\pa\Omega\},\quad
H_{\bm{\nu}}^1(\Omega;\R^3) = W^{1,2}_{\bm{\nu}}(\Omega;\R^3).
$$


\subsection{Solution concept and main result}

Before we formulate our main result, we explain two types of solutions. We consider so-called weak solutions, i.e.\ solutions that fulfill equations \eqref{1.massbal}--\eqref{1.totEbal} with boundary conditions \eqref{bc.J}--\eqref{bc.theta} in the weak sense, and so-called variational entropy solutions (we use the terminology from \cite{NoPo_JDE}), i.e.\ solutions that fulfill \eqref{1.massbal}, \eqref{1.mombal} weakly, the integrated form of the total energy balance \eqref{1.totEint}, and the entropy inequality \eqref{1.entrbal}. For a certain choice of parameters, we always obtain the latter, while the former will be satisfied only if some terms from the total energy balance are integrable, hence for a smaller set of the parameters.

Let us explain the definition of our solutions in more detail. Before doing so, we detail the weak formulation of our equations. We consider

\begin{itemize}
\item
{\it the weak formulation of the mass balance},
\begin{align}
   \sum_{i=1}^N\int_\Omega\bigg(-\rho_i\vel + \sum_{j=1}^NM_{ij}
	\na\frac{\mu_j}{\theta} + M_i\na\frac{1}{\theta}\bigg)
	\cdot\na\phi_i \, dx
	= \sum_{i=1}^N\int_\Omega r_i\phi_i \, dx, \label{w.rhoi}
\end{align}
for all $\phi_1,\ldots,\phi_N\in W^{1,\infty}(\Omega)$;
\item
{\it the weak formulation of the momentum balance},
\begin{align}
\label{w.rhov}
\int_\Omega(-\rhotot\vel\otimes\vel+\St):\na\bm{u}\,dx
+ \int_{\pa\Omega}\alpha_1\vel\cdot\bm{u}\,ds
= \int_\Omega(p\diver\bm{u} + \rhotot\bm{b}\cdot\bm{u})\,dx,
\end{align}
for all $\bm{u}\in W_{\bm{\nu}}^{1,\infty}(\Omega)$;
\item
{\it the weak formulation of the total energy balance},
\begin{align}
\label{w.rhovE}
 \int_\Omega(-\rhotot E \bm{v} - \bm{Q} +\St \bm{v} -p\bm{v})\cdot\na \varphi \, dx
+\int_{\pa\Omega}(\alpha_1|\bm{v}|^2+ \alpha_2 (\theta-\theta_0))\varphi\, ds
= \int_\Omega\rhotot\bm{b}\cdot\bm{v} \varphi\, dx,
\end{align}
for all $\varphi\in W^{1,\infty}(\Omega)$; and
\item
{\em the weak formulation of the entropy inequality},
\begin{equation}
\begin{aligned}\label{w.rhos}
\int_\Omega &\bigg(\rho s \vel + \sum_{i=1}^N \frac{\mu_i}{\theta} \bigg( \sum_{j=1}^N M_{ij}\na \frac{\mu_j}{\theta} + M_i\na\frac{1}{\theta} \bigg) - \frac{1}{\theta}\bigg( \kappa(\theta)\na\theta + \sum_{i=1}^N M_i\na \frac{\mu_i}{\theta}\bigg)\bigg) \cdot\nabla\Phi\, dx  \\
  &\phantom{xx}{}+\int_\Omega\bigg(\sum_{i,j=1}^N M_{ij}\nabla \frac{\mu_i}{\theta}\cdot\nabla \frac{\mu_j}{\theta}
+ \kappa(\theta)|\nabla\log\theta|^2
+ \frac{\St : \na\vel}{\theta} -\sum_{i=1}^N r_i \frac{\mu_i}{\theta}\bigg)\Phi\,dx \\
& \leq \alpha_2\int_{\pa\Omega}\frac{\theta-\theta_0}{\theta}\Phi\, ds,
\end{aligned}
\end{equation}
for every $\Phi\in W^{1,\infty}(\Omega)$ with $\Phi\geq 0$ a.e.~in $\Omega$.
\end{itemize}
We also use the term {\it the total energy equality}, which is nothing but equation 
\eqref{w.rhovE} with $\psi \equiv 1$, i.e.
\begin{align}
\label{t.rhovE}
\int_{\pa\Omega}(\alpha_1|\bm{v}|^2+ \alpha_2 (\theta-\theta_0))\, ds
= \int_\Omega\rhotot\bm{b}\cdot\bm{v}\,  dx.
\end{align}
Finally, recall that if we choose $\phi_i\equiv 1$ in \eqref{w.rhoi} and
add the weak formulation for all constituents, we obtain in the sum
{\it the weak formulation of the continuity equation},
\begin{align} \label{w.rho}
\int_\Omega \rhotot \bm{v}\cdot \nabla\Phi\, dx = 0,
\end{align}
for all $\Phi \in W^{1,\infty}(\Omega)$. However, we shall work with a slightly stronger definition of a solution to this equation, with {\it the renormalized solution to the continuity equation}, which is an important tool in the theory of weak solutions to the compressible Navier--Stokes equations to show compactness of the sequence of densities. Hence, we consider only renormalized solutions in what follows, i.e., solutions satisfying in addition to the weak formulation \eqref{w.rho}
for $u\in W^{1,\infty}(\Omega)$ and $b \in C^{1}(\mathbb{R})$, $b(0)=0$ with $b'$ having compact support,
\begin{equation}\label{con-ren}
\int_{\Omega} \big(b(\rhotot)  \vel \cdot \nabla u - u(\rhotot b'(\rhotot)-b(\rhotot)) \diver \vel\big)\, dx  =0.
\end{equation}

\begin{definition}[Weak and variational entropy solutions]
We call the functions
$$
\rho_1,\dots, \rho_N \in L^\gamma(\Omega), \quad \vel \in H^1_{\bm{\nu}}(\Omega), \quad \log \theta,\theta^{\beta/2} \in H^1(\Omega)
$$
such that
$$
\rhotot |\vel|^2 \vel, \, \tn{S}(\theta,\nabla \vel)\vel, \, p(\vec{\rhotot},\theta)\vel  \in L^1(\Omega;\R^3)
$$
a {\em renormalized weak solution} to problem \eqref{1.massbal}--\eqref{1.J}, \eqref{1.r}--\eqref{1.p} if there holds the weak formulations of the species equation \eqref{w.rhoi}, momentum equation \eqref{w.rhov}, total energy balance \eqref{w.rhovE}, and the total density $\rhotot := \sum_{i=1}^N \rho_i$ is a renormalized solution to \eqref{con-ren}.

We call the functions
$$
\rho_1,\dots, \rho_N \in L^\gamma(\Omega), \quad \vel \in H^1_{\bm{\nu}}(\Omega), \quad \log \theta,\theta^{\beta/2} \in H^1(\Omega)
$$
such that
$$
\rhotot |\vel|^2 \in L^1(\Omega)
$$
a {\em renormalized variational entropy solution} to problem \eqref{1.massbal}--\eqref{1.J}, \eqref{1.r}--\eqref{1.p} if there holds the weak formulations of the species equation \eqref{w.rhoi}, momentum equation \eqref{w.rhov}, entropy inequality \eqref{w.rhos}, and  the total energy equality \eqref{t.rhovE}, and the total density $\rhotot := \sum_{i=1}^N \rho_i$ is a renormalized solution to \eqref{con-ren}.
\end{definition}

The main result of this paper is the following theorem.

\begin{theorem}[Large-data existence of solutions]\label{thm.ex}
Let Hypotheses (H1)--(H7) hold. Let $\beta > 2/3$ and $\gamma > 3/2$.
Then there exists a renormalized variational entropy solution
to \eqref{1.massbal}--\eqref{1.J}, \eqref{1.r}--\eqref{1.p}. Moreover, if $\beta >1$
and $\gamma > 5/3$, then the solution is also a renormalized weak solution.
\end{theorem}

\begin{remark}\rm 
Replacing the estimates of the total density computed from the Bogovskii operator estimates by a different technique used, e.g., in \cite{NoPo_SIMA} (see also \cite{JNP14,MPZ_Hand,PiPo17,PiPo18}), we could treat the problem with arbitrary $\gamma >1$ and certain bounds on $\beta$ depending on $\gamma$. However, it would significantly
complicate and extend this quite technical and long paper. Therefore we prefer not to do it here.
\end{remark}

\begin{remark} \rm
Note that in \eqref{w.rhoi} we may use test functions
$\phi_i=0$ and $\phi_j$ being a non-zero constant for all $j\neq i$.
This leads to a sort of compatibility condition,
$$
\int_\Omega r_j \, dx =0, \qquad j=1,2,\dots, N.
$$
We do not require this condition directly, but due to our assumptions on $r_j$ and due
to the construction, we are able to fulfil these conditions;
see Remark \ref{last_remark}.
\end{remark}


\subsection{Key ideas of the proof}

We first prove the existence of solutions to an approximate
Navier--Stokes--Fourier system, then derive a priori estimates from	
the entropy balance and the total energy equality, and finally pass to the limit
of vanishing approximation parameters.
The approximate system is obtained from the mass densities and momentum balance equations \eqref{1.massbal}, \eqref{1.mombal}, as well as from the following {\em internal energy balance equation:}
\begin{equation*}
  \diver(\rhotot e\vel + \bm{Q}) - (\St-p\,\mathbb{I}):\na\vel = 0 , 
\end{equation*}
which we consider in place of the total energy balance \eqref{1.totEbal} and entropy inequality \eqref{1.entrbal}. It is obtained by computing the scalar product between $\vel$ and \eqref{1.mombal} and then subtracting the resulting equation from \eqref{1.totEbal}.

We use several levels of approximation. They are described in detail in Section
\ref{sec.scheme}, which deals with the existence of solutions. In the construction of a solution to the approximate problem, we use for the velocity the Galerkin approximation
with dimension $n\in\N$; we add lower-order regularizations with parameter $\eps>0$
and higher-order regularizations with parameters $(\chi,\delta,\xi)$
to the other equations, leading to $H^2(\Omega)$ regularity; and we regularize the
free energy with parameter $\eta>0$ to obtain higher integrability of the total
mass density.
Moreover, we construct not the temperature, but rather its logarithm, which allows
us to deduce the a.e.\ positivity of the temperature.

We cannot use the artificial viscosity regularization
in the total mass continuity equation as in \cite[Section 3.3.1]{FeNo09},
yielding a linear elliptic equation with drift, because of the cross-diffusion
terms, which significantly complicates our analysis. In fact, we need
two approximation levels, distinguishing between higher-order and lower-order
regularizations.

The existence of a solution to the approximate scheme is proved in Section
\ref{sec.scheme}. The positivity of the temperature is obtained from the
bound $\log\theta\in W^{1,4}(\Omega)\hookrightarrow L^\infty(\Omega)$.
The uniform estimates
from the entropy and total energy balance allow us to pass to the limits
$\delta\to 0$ and $n\to\infty$ (in this order). For the limits
$\chi\to 0$, $\eps\to 0$, and $\xi\to 0$ (in this order), we need
an estimate for the total mass density in $L^{\gamma+\nu}(\Omega)$ for some $\nu>0$.
It yields a uniform bound for the pressure in a space better than $L^1(\Omega)$.
This is shown by using a test function involving the Bogovskii operator in the
weak formulation of the momentum balance equation (see \cite[Section 7.3.3]{NoSt04}).

The most difficult part of the proof is the strong convergence of the approximate
total mass density $(\rhotot_\eps)$.
It is based on an effective viscous flux identity or weak compactness
identity for the pressure \cite[Section 3.7.4]{FeNo09} (Lemma \ref{lem.eff}) and
some properties related to Feireisl's oscillations defect measure. More precisely,
we shall prove that (the bar denotes the weak limit for $\delta\to 0_+$ and $\varepsilon\to 0_+$) 
\begin{equation}\label{1.pT1}
  \overline{p(\vec\rho_\eps,\theta_\eps)T_k(\rhotot_\eps)}
	- \overline{p(\vec\rho_\eps,\theta_\eps)}\ \overline{T_k(\rhotot_\eps)}
\end{equation}
vanishes, where the "overline" signifies
the weak limit and $T_k$ is some truncation operator (Lemma \ref{lem.Wk}), and that
\begin{equation}\label{1.pT2}
  \overline{(p(\vec\rho_\eps,\theta)-p(\vec\rho_\delta))(T_k(\rhotot_\eps)
	- \rhotot_\eps T_k'(\rhotot_\eps) - T_k(\rhotot_\delta)
	+ \rhotot_\delta T_k'(\rhotot_\delta))}
\end{equation}
vanishes (Lemma \ref{lem.Qk}).
The proofs of the strong limits of
\eqref{1.pT1} and \eqref{1.pT2} rely heavily on the
convexity of the free energy density and the strong convergence of the relative
chemical potentials. These limits are necessary to deal with the cross-diffusion terms,
and the proofs seem to be new.
Another ingredient is the proof that the weak limit $\rhotot$
of $(\rhotot_\eps)$ is a renormalized solution to the mass continuity equation.
Using a special test function and renormalization function, we are able to
control the oscillations defect measure and to prove the strong convergence
of $(\rhotot_\eps)$ to $\rhotot$.
Finally, the convergence of the vector of mass densities follows (again)
from the convexity of the free energy and the compactness of the relative
chemical potentials (Lemma \ref{thm.conv.total}).

We underline that the method used in this paper is essentially new in the aspect how we treat the strong convergence of the total density. The main new idea in this direction is the fact that we may avoid the use of the oscillation defect measure estimates which requires additional restrictions on the coefficients $\gamma$ and $\beta$. In earlier papers dealing with compressible fluids (both single constituent and mixtures) this tool is needed in the case when the density does not belong to $L^2(\Omega)$, since the classical DiPerna--Lions theory for renormalization is not available. In our situation, this corresponds to $\gamma \leq  5/3$ and $\beta \leq 1$. Note that this allows us to prove the result under minimal conditions and in fact also to improve the known results for the steady compressible Navier--Stokes--Fourier equations.


\subsection{State of the art and originality of the paper}

The literature on compressible Navier--Stokes and Navier--Stokes--Fourier
systems is very extensive. First results on the existence of solutions to
the stationary compressible Navier--Stokes equations (for a single species)
without  assumptions
on the size of the data goes back to P.-L.~Lions \cite{Lio98}. He assumed the pressure
relation $p(\rho)=\rho^\gamma$ with $\gamma>5/3$ for the stationary flow.
The most difficult part of the proof, the strong convergence of the sequence
of approximate densities, is based on a weak continuity property of the
effective viscous flux $p(\rho)+(2\lambda_1+\lambda_2)\diver\bm{v}$ and
the theory of renormalized solutions applied to the continuity equation.
A first improvement on the pressure exponent $\gamma$
is due to B\v{r}ezina and Novotn\'y \cite{BrNo08}, who assumed that
$\gamma>(1+\sqrt{13})/3$. Their proof is based on some ideas due to Plotnikov and Sokolowski \cite{PlSo} and on Feireisl's idea of the oscillations defect measure estimate (see
\cite{Fei01} in the evolutionary case), described  in the steady case in \cite{NoSt04} for a special class of non-volume forces. Further improvements of the lower bound for the  exponent $\gamma$ are due to Frehse, Steinhauer, and Weigant \cite{FrStWe}.
The existence of weak solutions to the steady compressible Navier--Stokes
equations for any $\gamma>1$ was shown in \cite{JiZh11}
(for space periodic boundary conditions), in \cite{JeNo13}
(for slip boundary conditions), and in \cite{PlWe} (for Dirichlet boundary conditions). In the case of evolutionary Navier--Stokes equations, the existence of a solution was proved in \cite{Lio98} for $\gamma \geq 9/5$ and the lower bound was decreased to $\gamma > 3/2$ in \cite{FeNoPe}.

The existence theory for the Navier--Stokes--Fourier system  employs the techniques of both Lions and Feireisl.
The first result for the stationary compressible Navier--Stokes--Fourier system
was proved by P.-L.~Lions \cite{Lio98} under the assumption that the density
is bounded in some $L^q$ space for sufficiently large values of $q$.
Without this assumption, when the density is a priori controlled only in
$L^1(\Omega)$, the existence of weak solutions was shown in \cite{MuPo09}
for $\gamma>3$ and in \cite{MuPo10} for $\gamma>7/3$.
These results were improved in a series of papers, see \cite{NoPo_JDE,NoPo_SIMA}
for Dirichlet boundary conditions and \cite{JNP14} for the Navier boundary conditions,
showing the existence of a variational entropy solution
(satisfying the entropy inequality and the total energy equality)
for any $\gamma>1$ and the existence of a weak solution (satisfying the
total energy balance) for any $\gamma>5/4$ (Navier boundary conditions) or $\gamma>4/3$
(Dirichlet boundary conditions). We refer to \cite{MPZ_Hand} for further information.

For results on the time-dependent compressible Navier--Stokes and
Navier--Stokes--Fou\-rier equations, we refer
to the monographs \cite{FKP16,FeNo09,NoSt04}. One difficulty
is the proof of the strong convergence of the sequence of approximate temperatures
which makes necessary the application of the div-curl lemma and the
effective viscous flux relation by using a cancellation property different from
the isothermal model.
The transient Navier--Stokes equations
with density-dependent viscosities satisfying a certain structure condition
allow for new a priori estimates thanks to the
Bresch--Desjardin entropy \cite{BrDe07}. However, these estimates are not
available for the steady problem.
The evolutionary problem for a heat conducting fluid with basically the same restriction on the adiabatic coefficient $\gamma$ was considered in
\cite{Fe_book,FeNo09} for different formulations of the energy balance (internal energy inequality and entropy inequality, respectively).

All these results concern the single-species case. The theory of fluid mixtures
requires some careful modeling to maintain thermodynamic consistency;
we refer to \cite{BoDr15,BuHa,Gio99} for the thermodynamic theory of fluid mixtures.
One of the first results was proved in \cite{Zat11}, namely
the existence of weak solutions to the stationary Navier--Stokes equations
assuming Fick's law and $\gamma>7/3$.
This result was improved in \cite{GPZ15} for Maxwell--Stefan-type fluxes
and $\gamma>5/3$. Another improvement was achieved in \cite{PiPo17,PiPo18}
for variational entropy solutions with $\gamma>1$
and for weak solutions with $\gamma>4/3$. These results are based on the assumption of same molar masses for each constituent.
Concerning the evolutionary problem, the first global in time result for arbitrarily large data is due to \cite{FePeTr} for Fick's law.
An existence result for a general thermodynamically consistent transient
Navier--Stokes model with $\gamma>3/2$ was recently presented in \cite{DDGG20}.
Electrically charged incompressible mixtures were analyzed in \cite{BPZ17}.
We also mention the work \cite{MaPr18} in which a multicomponent viscous compressible
fluid model with separate velocities of the components was studied.

Our existence result generalizes previous works on fluid mixtures. Indeed,
the mobility matrix in \cite{GPZ15,PiPo17,PiPo18}
has no contributions to the temperature
gradients, the pressure, and internal energy are not defined through the free energy,
and the molar masses are assumed to be equal.
These restrictions were removed in \cite{BPZ17}, but the authors consider
incompressible fluids. In the works \cite{DDGG20}, a very general
compressible fluid model is analyzed but no temperature effects have been
taken into account.

In this paper, we combine all the features studied in the above-mentioned works,
namely we allow for temperature gradients, a thermodynamically consistent modeling
starting from the Helmholtz free energy, compressible fluids, and
different molar masses. Moreover, we obtain a new proof for the strong
convergence of the sequence of approximate densities by exploiting the convexity
of the free energy.

\medskip
The paper is organized as follows. 
Some auxiliary results, in particular on the convex conjugate
$h_\theta^*$, are shown in Section \ref{sec.aux},
and a priori estimates for smooth solutions are
derived in Section \ref{sec.smooth}. We prove in Section \ref{sec.comp}
the compactness of the sequence of total mass densities satisfying the
Navier--Stokes--Fourier mixture model. This step highlights the key features
and novelties of the proof without obstructing it by the numerous approximating terms.
The construction of smooth approximate solutions and the deregularization
limits are presented in Sections \ref{sec.scheme} and \ref{sec_conv}, respectively.
In the Appendix, we recall further auxiliary results needed in this paper and
we show that the free energy density in Remark \ref{rem.energy}
satisfies Hypothesis (H6).


\section{Auxiliary results}\label{sec.aux}

\begin{lemma}\label{lem:H6}
Let $h_{\theta}$ satisfy Hypotheses (H6)--(H7). Then the following holds:

{\rm (i)} The mapping
$\na h_{\theta}:\R^N_+ \to \R^N$ is a bijection,
the convex conjugate $h_{\theta}^*$ of $h_\theta$ is well defined, and it holds that
$\na h_{\theta}^* = (\na h_{\theta})^{-1} : \R^N\to \R^N_+$.

{\rm (ii)} For all $R>0$, there exist $K_1>0$ and a continuous function $\omega$
fulfilling $\omega(0)=0$ such that for all $\theta_1,\theta>0$ and
$\mu \in \R^N$,
\begin{equation}\label{H6A}
\begin{aligned}
  &\mbox{if}\quad
  \theta+\theta_1+\theta_1^{-1}+\theta^{-1}+\sum_{i=1}^N
	\frac{\pa h_\theta^*}{\pa\mu_i}(\vec\mu)\le R\quad\mbox{then}\quad
  \bigg|\frac{\pa^2 h_\theta^*}{\pa\mu_i\pa\mu_j}(\vec\mu)\bigg|
	\le K_1, \\
	&	|\na h_\theta^*(\vec\mu)-\na h_{\theta_1}^*(\vec\mu)|
	+|h_\theta^*(\vec\mu)-h_{\theta_1}^*(\vec\mu)|\le K_1\omega(\theta_1-\theta).
\end{aligned}
\end{equation}

{\rm (iii)} There exists $K_2>0$ such that for all $\theta>0$ and 
$\vec{\mu}\in \R^N$,
\begin{equation}\label{H6C}
 \theta \sum_{i,j=1}^N\frac{\pa^2 h_\theta^*}{\pa\mu_i\pa\mu_j}(\vec\mu)
	\le K_2 \sum_{i=1}^N\frac{\pa h_\theta^*}{\pa\mu_i}(\vec\mu).
\end{equation}

{\em (iv)} If additionally Hypothesis (H7) holds then for all $\vec\rho\in\R_{+,0}^N$ 
and $\theta>0$,
\begin{equation}\label{prrr}
  c_p(\rhotot\theta + \rhotot^\gamma)\le p(\vec\rho,\theta)
	\le C_p(1+\rhotot\theta+\rhotot^\gamma).
\end{equation}
\end{lemma}

\begin{proof}
Let $\theta>0$ be fixed in the whole proof. 

{\bf (i)} Let us assume that (H6a) holds and recall the definition of 
$h_{\theta}^*$, 
$$
  h^*_{\theta}(\vec{\mu}):=\sup_{\vec{\rho}\in \R^N_+}
  (\vec{\mu} \cdot \vec{\rho} - h_{\theta}(\vec{\rho})).
$$
Note that the definition is meaningful for all $\vec{\mu}\in \mathbb{R}^N$. 
It follows from \eqref{LMB1} that
$$
  h^*_{\theta}(\vec{\mu})\ge \lim_{\rhotot \to 0_+}(\vec{\mu} 
	\cdot \vec{\rho} - h_{\theta}(\vec{\rho}))=0,
$$
and thus, $h^*_{\theta}$ is nonnegative. Next, if the supremum of 
$\vec\rho\mapsto \vec{\mu} \cdot \vec{\rho} - h_{\theta}(\vec{\rho})$
is attained at some point $\vec{\rho} \in \R^N_+$, then evidently (since all 
partial derivatives must vanish)
\begin{equation}\label{EV1}
  \vec{\mu}=\nabla h_{\theta}(\vec{\rho}).
\end{equation}
Thus, we need to show that the supremum is attained at some element of 
$\R^N_+$. For this, consider a maximizing sequence 
$(\vec{\rho}^n)_{n\in\N}\subset \R^N_+$ such that
$$
  0\le h^*_{\theta}(\vec{\mu})=\lim_{n\to \infty} \left(\vec{\rho}^n 
	\cdot \vec{\mu} - h_{\theta}(\vec{\rho}^n)\right).
$$
By \eqref{LMB1} again, for any $\vec{\mu}$, we can find $R>0$ such that 
if $\rhotot >R$ then $h_{\theta}(\vec{\rho})/\rhotot > |\vec{\mu}|+1$.
Consequently, if $\rhotot^n >R$ then
$$
  \vec{\rho}^n \cdot \vec{\mu} - h_{\theta}(\vec{\rho}^n)
	= \rhotot^n \bigg(\frac{\vec{\rho}^n}{\rhotot^n} \cdot \vec{\mu} 
	- \frac{h_{\theta}(\vec{\rho}^n)}{\rhotot^n} \bigg) 
	\le \rhotot^n (|\vec{\mu}| - |\vec{\mu}|-1)\le -R<0,
$$
which contradicts $h_{\theta}^* \ge 0$. Hence, there exists $n_0\in\N$ such that 
for $n\ge n_0$, we have $\rhotot^n \le R$. Thus, $(\vec{\rho}^n)$ is a bounded 
sequence, and there exists $\vec{\rho}\in \R^N_{+,0}$ such that (for a 
subsequence that we do not relabel) $\vec{\rho}^n \to \vec{\rho}$ as $n\to\infty$.

To conclude, we need to verify that the limit satisfies $\rho_i>0$ for all
$i=1,\ldots,N$. Then $\vec{\rho}\in \R^N_+$ and the proof of the first part is finished. To this end, we assume by contradiction that for some $i=1,\ldots, N$, 
we have $\rho_i^n \to 0_+$.
Since we already know that $\rhotot^n \le R$ for some $R>0$ and all $n\ge n_0$, 
we can use \eqref{LMB2} to find $n_1\in\N$ such that for all $n\ge n_1$, we have
$$
  0<\rho^{n_1}_i \le 1\quad \textrm{and}\quad 
	\frac{\pa h_{\theta}}{\pa s}(\rho_1^n,\ldots, \rho_{i-1}^n,s,\rho_{i+1}^n,
	\ldots, \rho^n_N) 
	< -\frac{|\vec{\mu}| +1}{\rho_i^{n_1}} 
$$
for all $s\in (0,\rho^{n_1}_i)$. Set 
$$
  \vec{\varrho^n} :=  (\rho_1^n,\ldots, \rho_{i-1}^n,\rho^{n_1}_i,
  \rho_{i+1}^n,\ldots, \rho^n_N).
$$
It follows from the definition of $h_{\theta}^*$ that
\begin{align*}
  h^*_{\theta}(\vec{\mu}) 
	&= \lim_{n\to \infty} \big(\vec{\rho}^n \cdot \vec{\mu} 
	- h_{\theta}(\vec{\rho}^n)\big)
	= \lim_{n\to \infty} \big(\vec{\varrho^n} \cdot \vec{\mu} 
	- h_{\theta}(\vec{\varrho^n})+(\vec{\rho}^n-\vec{\varrho^n}) \cdot \vec{\mu} 
	- h_{\theta}(\vec{\rho}^n)+h_{\theta}(\vec{\varrho^n})\big) \\
  &\le h^*_{\theta}(\vec{\mu}) + \limsup_{n\to \infty} 
	\bigg((\rho_i^n-\rho_i^{n_1})\mu_i +\int_{\rho_i^n}^{\rho_i^{n_1}}
	\frac{\pa h_{\theta}}{\pa s}(\rho_1^n,\ldots, \rho_{i-1}^n,s,\rho_{i+1}^n,\ldots, 
	\rho^n_N)ds\bigg) \\
  &\le h^*_{\theta}(\vec{\mu}) + \limsup_{n\to \infty} 
	\bigg(|\vec{\mu}|- \frac{\rho_i^{n_1}-\rho_i^n}{\rho_i^{n_1}}(|\vec{\mu}|+1)\bigg)
	= h^*_{\theta}(\vec{\mu})-1,
\end{align*}
which is a contradiction.

We infer that \eqref{EV1} is valid for some $\vec{\rho}\in \R^N_+$. In addition, 
since $h_{\theta}$ is strictly convex and $\R^N_+$ is convex, the vector 
$\vec{\rho}$ is unique. Then the definition of $h_{\theta}^*$ yields
\begin{equation}\label{EPP1}
  h^*_{\theta}(\vec{\mu}) = \vec\mu \cdot (\na h_{\theta})^{-1}(\vec{\mu}) 
	- h_{\theta}((\na h_{\theta})^{-1}(\vec{\mu})).
\end{equation}
Consequently, 
\begin{equation}
  \na h^*_{\theta} (\vec{\mu}) = (\na h_{\theta})^{-1} (\vec{\mu}), \label{EP2}
\end{equation}
and we conclude that $\na h^*_{\theta}: \R^N \to \R^N_+$ is a bijection.

{\bf (ii)} Next, we prove \eqref{H6A}. For this, we define the symmetric matrices 
$$
  \mathbb{A}^{\theta}(\vec{\rho}):=\nabla^2 h_{\theta}(\vec{\rho}), \quad 
	\mathbb{B}^{\theta}(\vec{\mu}):=\nabla^2 h^*_{\theta}(\vec{\mu}).
$$
By \eqref{EP2}, they are inverse to each other at 
$\vec\mu=\nabla h_{\theta}(\vec{\rho})$:
\begin{equation}\label{EP3}
  \mathbb{A}^{\theta}(\vec{\rho})\mathbb{B}^{\theta}(\nabla h_{\theta}(\vec{\rho}))
	= \mathbb{A}^{\theta}(\nabla h^*_{\theta} (\vec{\mu}))
	\mathbb{B}^{\theta}(\vec{\mu})= \mathbb{I},
\end{equation}
where $\mathbb{I}$ is the identity matrix and $\mathbb{A} \mathbb{B}$ is the 
standard matrix multiplication. 

Let the assumptions of \eqref{H6A} be satisfied, i.e.\ for fixed $R>1$,
let $\theta\in[R^{-1},R]$ and let $\vec\rho:=\na h^*_\theta(\vec\mu)$ fulfill
$\rho=\sum_{i=1}^N\rho_i\le R$. Assumption \eqref{LMB21} can be formulated as
$$
  L:=\inf_{\{\theta,\,\vec{\rho},\,\vec{x}: \theta\in (R^{-1},R), \, 
	\rho_i \in (0,R), \, |\vec{x}|=1\}} \sum_{i,j=1}^N
	\frac{\partial^2 h_{\theta}}{\partial \rho_i \partial \rho_j}(\vec{\rho})x_i x_j
	\ge \frac{K}{R^2}.
$$
Therefore, the (real) eigenvalues of the symmetric matrix
$\mathbb{A}^{\theta}(\vec{\rho})$ are greater or equal to $L$. 
Since $\mathbb{B}^{\theta}$ is the inverse of $\mathbb{A}^{\theta}$,
the eigenvalues of $\mathbb{B}^{\theta}(\nabla h_{\theta})
=\mathbb{B}^{\theta}(\vec{\mu})$ are less or equal to $L^{-1}$.
 Consequently, $|\mathbb{B}^{\theta}(\vec{\mu})|\le CL^{-1}$,
which proves the first part of \eqref{H6A}.

Next, we prove the second part of \eqref{H6A}. We start with some preparations.
It follows from \eqref{EPP1} that
$$
  h_{\theta}(\vec{\rho}) = \vec\rho \cdot \nabla h_{\theta}(\vec{\rho}) 
	- h^*_{\theta}(\nabla h_{\theta}(\vec{\rho})).
$$
We differentiate this identity with respect to $\theta$:
\begin{align*}
  \frac{\partial h_{\theta}}{\partial \theta}(\vec{\rho})
	&= \vec\rho \cdot \frac{\partial \nabla h_{\theta}}{\pa\theta}(\vec{\rho}) 
	- \frac{\partial h^*_{\theta}}{\partial \theta} (\nabla h_{\theta}(\vec{\rho})) 
	-\nabla h^*_{\theta}(\nabla h_{\theta}(\vec{\rho})) 
	\cdot \frac{\partial \nabla h_{\theta}(\vec{\rho})}{\partial \theta} \\
  &=- \frac{\partial h^*_{\theta}}{\partial \theta} (\nabla h_{\theta}(\vec{\rho})),
\end{align*}
since $\nabla h^*_{\theta}(\nabla h_{\theta}(\vec{\rho})) = \vec\rho$
by the properties of the convex conjugate. Recalling that 
$\nabla h_{\theta}$ is a bijection from $\mathbb{R}_+^N$ to $\mathbb{R}^N$,
it follows for $\mu\in \mathbb{R}^N$ and $\theta >0$ that
\begin{equation}\label{EPP1th}
  \frac{\partial h^*_{\theta}(\vec{\mu})}{\partial \theta} 
	= -\frac{\partial h_{\theta}}{\partial \theta}(\nabla h_{\theta}^*(\vec{\mu})).
\end{equation}
We differentiate this identity with respect to $\mu$ and use \eqref{EP2}: 
\begin{equation}\label{EPP1mu}
  \frac{\partial \nabla h^*_{\theta}}{\partial \theta}(\vec{\mu}) 
	= -\nabla^2 h_{\theta}^*(\vec{\mu})
	\frac{\partial \nabla h_{\theta}}{\partial \theta}(\nabla h_{\theta}^*(\vec{\mu}))
	= -(\nabla^2 h_{\theta}(\nabla h^*_{\theta}(\vec{\mu})))^{-1}
	\frac{\partial \nabla h_{\theta}}{\partial \theta}(\nabla h_{\theta}^*(\vec{\mu})).
\end{equation}

Now, we are ready to prove the second part of \eqref{H6A}. Let $R>0$ be given.
The assumptions of \eqref{H6A} imply that $\rho:=\sum_{i=1}^N \rho_i \le R$
for $\vec{\rho}= \nabla h_{\theta}^*(\vec{\mu})$ and that there exists a constant
$C>0$ depending on $R$ such that for all $t\in (R^{-1},R)$, it follows that 
$\rho_t\le CR$, where 
$\rho_t:=\sum_{i=1}^N (\rho_t)_i$ and $\vec{\rho}_t:=\nabla h_t^*(\vec{\mu})$.
Hence, we can use \eqref{AUXc} and \eqref{grow_h_2} or \eqref{grow_h_1} in 
\eqref{EPP1th}--\eqref{EPP1mu} to deduce that for all $t\in (R^{-1},R)$,
$$
  \left|\frac{\partial \nabla h^*_{t}(\vec{\mu})}{\partial t} \right| 
	+ \left| \frac{\partial h^*_{t}(\vec{\mu})}{\partial t}\right| \le C(R).
$$
Thus, $\nabla h^*_{t}(\vec{\mu})$ and $h^*_{t}(\vec{\mu})$ are Lipschitz 
continuous with respect to $t$ and consequently, the second part of \eqref{H6A} holds.

{\bf (iii)} In the third step, we prove \eqref{H6C}. For this,
we define $\vec{\rho}:= \nabla h^*_{\theta}(\vec{\mu})$. 
It follows from \eqref{EP3} that \eqref{H6C} is equivalent to
\begin{equation}\label{H6Ce}
  \theta \sum_{i,j=1}^N (\mathbb{A}^{\theta}(\vec{\rho}))^{-1}_{ij}
	\le K_2 \rho.
\end{equation}
By \eqref{LMB21}, the lowest eigenvalue $\lambda_{\min}$ of 
$\mathbb{A}^{\theta}$ satisfies $\lambda_{\min}\ge K\theta/\rho$.
Thus, the largest eigenvalue $\lambda_{\max}$ of $(\mathbb{A}^{\theta})^{-1}$ 
satisfies $\lambda_{\max} = \lambda_{\min}^{-1} \leq \rho/(K\theta)$.
Denoting $\vec{1}:=(1,\ldots, 1)$, we infer that
$$
  \sum_{i,j=1}^N (\mathbb{A}^{\theta}(\vec{\rho}))^{-1}_{ij}
	= (\mathbb{A}^{\theta}(\vec{\rho}))^{-1}\vec{1} \cdot \vec{1} 
	\le \lambda^{-1}_{\min} |\vec{1}|^2 
	\le \frac{N\rho}{K\theta}
$$
and \eqref{H6Ce} follows with $K_2:=N/K$.

{\bf (iv)} It remains to show the estimates for the pressure. Let 
$\vec{\rho}\in \mathbb{R}^N_+$ and $\theta\in \mathbb{R}_+$ be arbitrary. 
Then, for $\tau\in (0,1)$, because of \eqref{1.p},
\begin{align*}
  p(\theta,\vec{\rho}) & =p(\theta,\tau\vec{\rho})
	+ \int_{\tau}^1 \frac{\partial}{\partial t} p(\theta,t\vec{\rho}) dt \\
  &= \nabla h_{\theta}(\tau\vec{\rho})\cdot \tau \vec{\rho} 
	- h_{\theta}(\tau \vec{\rho}) + \int_{\tau}^1 
	\frac{\partial}{\partial t}\big(\nabla h_{\theta}(t\vec{\rho})\cdot t \vec{\rho} 
	- h_{\theta}(t \vec{\rho})\big)dt \\
  &= \nabla h_{\theta}(\tau\vec{\rho})\cdot \tau \vec{\rho} 
	- h_{\theta}(\tau \vec{\rho}) + \int_{\tau}^1 t\sum_{i,j=1}
	\frac{\partial^2 h_{\theta}}{\partial\rho_i \partial \rho_j}(t\vec{\rho}) 
	\rho_i\rho_j dt.
\end{align*}
Hypothesis (H6a) allows us to pass to the limit $\tau\to 0_+$ such that
$$
  p(\theta,\vec{\rho}) = \int_0^1 t\sum_{i,j=1}
	\frac{\partial^2 h_{\theta}}{\partial\rho_i \partial\rho_j}(t\vec{\rho}) 
	\rho_i\rho_j dt.
$$ 
We infer from Hypothesis (H7) that
$$
  c_p|\vec{\rho}|^2 \int_0^1 \frac{\theta t}{t\rho} + t^{\gamma-1}\rho^{\gamma-2} dt  
	\le p(\theta,\vec{\rho})\le c_P|\vec{\rho}|^2 \int_0^1 
	\bigg(\frac{\theta t}{t\rho} + t^{\gamma-1}\rho^{\gamma-2}
	+ t^{\omega-1}\rho^{\omega-2}\bigg) dt.
$$
We use Young's inequality and observe that $\omega\in (1,\gamma)$
to deduce \eqref{prrr} with properly chosen $c_p$ and $C_p$.
\end{proof}

The following property is needed in the construction of the approximate solution.

\begin{lemma} \label{l A1}
Let the free energy $h_\theta$ fulfil Hypotheses (H6) and (H7); 
thus the pressure satisfies \eqref{prrr}. Then 
for any sequence $(\theta^{(n)})_{n\in\N}\subset (0,\infty)$,
$(\vec\mu^{(n)})_{n\in\N}\subset\R^N$, if both $\theta^{(n)}$ and $\rhotot^{(n)}$
are bounded and $\min_{1\leq i\leq N}\mu_i^{(n)}\to \infty$ then
$\rhotot^{(n)}\to 0$.	
\end{lemma}

\begin{proof}
It follows from \eqref{1.p}, \eqref{grow_h_1} and \eqref{prrr} that
$$
\rhotot\min_{1\leq i\leq N}\mu_i\leq\sum_{i=1}^N\rho_i\mu_i =
p + h_\theta(\vec\rho)\leq C( 1+\rhotot\theta + \rhotot^\gamma +
\rhotot^{5(\gamma+\nu-\zeta)/6} + \theta^{5(3\beta - \zeta)/6}),
$$
which immediately implies the statement.
\end{proof}

The following lemma plays an important role in the construction of the
approximate solutions.

\begin{lemma}\label{lem.refrho}
Let the free  energy $h_\theta$ fulfil Hypothesis (H6).
Then for every $\overline\rhotot>0$ and $w\in L^\infty(\Omega)$,
the algebraic equation
\begin{equation}\label{eq.q0}
\int_{\Omega}\sum_{i=1}^N\frac{\pa h_\theta^*}{\pa\mu_i}(\theta q_0\vec{1})\, dx
= |\Omega|\overline{\rhotot},\quad \theta:= e^w ,\quad \vec{1} = (1,\ldots,1)\in\R^N,
\end{equation}
has a unique solution $q_0\in\R$. Moreover, for every $\overline\rhotot>0$ and
$i=1,\ldots,N$, the mappings
\begin{align*}
  & \widetilde{q}_0: L^\infty(\Omega)\to\R, \quad w\mapsto q_0, \\
  & \widetilde{\mathfrak{R}}_i:L^\infty(\Omega)\to L^\infty(\Omega), \quad
	w\mapsto \frac{\pa h_{\exp(w)}^*}{\pa\mu_i}
	\big(\exp(w) \widetilde {q}_0(w)\vec{1}\big)
\end{align*}
are continuous.
\end{lemma}

\begin{proof}
Set $\pa_i=\pa/\pa\mu_i$ and $\pa_{ij}^2=\pa^2/\pa\mu_i\pa\mu_j$.
For every $\theta_0>0$, the function
$f_{\theta_0}(q_0) = \sum_{i=1}^N\pa_{i} h_{\theta_0}^*(\theta_0 q_0\vec{1})$
is strictly increasing, since $f_{\theta_0}'(q_0) = \theta_0\sum_{i,j=1}^N
\pa^2_{ij} h_{\theta_0}^*(\theta_0 q_0\vec{1})>0$,
due to the convexity of $h^*_\theta$.
It follows that $F_\theta(q_0)=\int_{\Omega}\sum_{i=1}^N\pa_i h_\theta^*
(\theta q_0\vec{1})\, dx$ is strictly increasing, provided that
$\theta\in L^\infty(\Omega)$ is uniformly positive in $\Omega$. Thus,
\eqref{eq.q0} has at most one solution $q_0\in\R$. To prove
the existence of a solution to \eqref{eq.q0}, we need to show that
$\lim_{q_0\to -\infty}F_\theta(q_0) = 0$ and
$\lim_{q_0\to +\infty}F_\theta(q_0) = \infty$.

For $\theta=e^w$, $w\in L^\infty(\Omega)$, let $\theta_m$, $\theta_M>0$
be such that $\theta_m\leq\theta\leq\theta_M$ a.e.~in $\Omega$. The fact that
for $\theta_0>0$, the mapping
$\lambda\in\R\mapsto \sum_{i=1}^N\pa_{i} h_{\theta_0}^*
(\lambda\vec{1})\in\R$ is increasing, implies that
$$
0\leq
\int_\Omega\sum_{i=1}^N\pa_i h_\theta^*(\theta_m q_0\vec{1})\, dx
\leq F_\theta(q_0) \leq
\int_\Omega\sum_{i=1}^N\pa_i h_\theta^*(\theta_M q_0\vec{1})\, dx .
$$
Therefore it is sufficient to show that
$$
\lim_{\lambda\to -\infty}\int_\Omega\sum_{i=1}^N\pa_i h_\theta^*
(\lambda\vec{1})\, dx = 0,\quad
\lim_{\lambda\to +\infty}\int_\Omega\sum_{i=1}^N\pa_i h_\theta^*
(\lambda\vec{1})\, dx = +\infty.
$$
By the monotonicity of $\lambda\mapsto \sum_{i=1}^N\pa_{i} h_{\theta(x)}^*
(\lambda\vec{1})$ for $\lambda\in\R$ and a.e.~$x\in\Omega$,
the dominated convergence theorem, and Fatou's lemma, it is sufficient to show that
\begin{equation}\label{lim.q0}
\lim_{\lambda\to -\infty}\sum_{i=1}^N\pa_i h_{\theta_0}^*
(\lambda\vec{1}) = 0,\quad
\lim_{\lambda\to +\infty}\sum_{i=1}^N\pa_i h_{\theta_0}^*
(\lambda\vec{1}) = \infty\quad\mbox{for all }\theta_0>0.
\end{equation}
Actually, the limits in \eqref{lim.q0} follow from the monotonicity of $\lambda
\mapsto \sum_{i=1}^N\pa_{i} h_{\theta_0}^*(\lambda\vec{1})$, Assumption
\eqref{H7C}, Lemma \ref{l A1}, and the fact that $\vec\mu=\na h_{\theta_0}(\vec\rho)$
is the inverse mapping of $\vec\rho=\na h_{\theta_0}^*(\vec\mu)$.
This means that \eqref{eq.q0} has exactly one solution $q_0\in\R$.

To prove the second part of the lemma, let $(w_n)_{n\in\N}\subset L^\infty(\Omega)$
be such that $w_n\to w$ strongly in $L^\infty(\Omega)$.
Define $\theta_n = e^{w_n}$, $\theta=e^w$. Clearly, there exist constants $L_1$,
$L_2>0$ such that $L_1\leq\theta_n\leq L_2$ a.e.~in $\Omega$, $n\in\N$.
The previous argument allows us to deduce that the corresponding sequence $(q_{0,n})$
is bounded in $\R$ and therefore (up to subsequences) convergent to a suitable
$q_0\in\R$. The fact that $h_\theta^*\in C^2((0,\infty)^N)$ depends smoothly on
$\theta$ implies that $\na h_{\theta_n}^*(\theta_n q_{0,n}\vec{1})\to \na h_{\theta}^*
(\theta q_{0}\vec{1})$ in $L^\infty(\Omega)$ and therefore the limits
$\theta$, $q_0$ satisfy \eqref{eq.q0}.
In particular, $q_0$ is uniquely determined by $w=\log\theta$. This finishes
the proof.
\end{proof}


\section{A priori estimates for smooth solutions}\label{sec.smooth}

This section is devoted to the derivation of suitable a priori estimates for
smooth solutions. Although we consider later weak solutions only, the computations
help us to identify the key estimates of the existence proof. In fact,
we need several regularizations for the full model, 
which may abstruse the main arguments.

We assume the existence of a smooth solution to the following Navier--Stokes--Fourier system with chemically reacting species:
\begin{align}
& \diver\left(\rho_i\vel -\sum_{j=1}^N M_{ij}\na\frac{\mu_j}{\theta} - M_i\na\frac{1}{\theta}\right) = r_i ,
\quad i=1,\ldots,N
\qquad\mbox{in }\Omega, \label{2.massbal} \\
& \diver(\rhotot\vel\otimes\vel - \St) + \na p = \rhotot \bm{b}
\qquad\mbox{in }\Omega, \label{2.mombal} \\
&\diver\left(\rhotot e\vel -\kappa(\theta)\na\theta - \sum_{i=1}^N M_i\na\frac{\mu_i}{\theta}\right) - (\St-p\,\mathbb{I}):\na\vel = 0
\qquad\mbox{in }\Omega, \label{2.enerbal}\\
 & \left(\sum_{j=1}^N M_{ij}\na\frac{\mu_j}{\theta} + M_i\na\frac{1}{\theta}\right)\cdot\bm{\nu} = 0, \quad i=1,\ldots,N
 \qquad\mbox{on }\pa\Omega, \label{2.bc.J} \\
& \vel\cdot\bm{\nu} = 0, \quad
(\mathbb{I}-\bm{\nu}\otimes\bm{\nu})(\mathbb{S}(\vel)\bm{\nu} + \alpha_1\vel) = 0\qquad\mbox{on }\pa\Omega,
\label{2.bc.vel} \\
& \left(\kappa(\theta)\na\theta + \sum_{i=1}^N M_i\na\frac{\mu_i}{\theta}\right)\cdot\bm{\nu} - \alpha_2(\theta_0-\theta) = 0 \qquad\mbox{on }\pa\Omega . \label{2.bc.theta}
\end{align}
Note that on this level, we can freely switch from the internal energy balance \eqref{2.enerbal} to the total energy balance  or the entropy equality; see Lemma \ref{lem.entbal} below. In order to use the procedure also during the construction of the solution, we immediately obtain the integrated entropy and total energy balances. To obtain the weak formulation of the entropy and total energy balances, we can proceed as in the proof of the lemma below, we just multiply the corresponding strong formulation additionally by a smooth test function $\psi$, which gives several additional terms containing the derivatives of $\psi$. In case of the entropy (in)equality, later on, we also require that the test function is nonnegative.

\begin{lemma}[Entropy and total energy balances]\label{lem.entbal}
Let $(\vec\rho,\vel,\theta)$ be a smooth solution to\linebreak
\eqref{2.massbal}--\eqref{2.bc.theta} such that $\theta>0$. Then
 the entropy equality
\begin{align}
  \int_\Omega\bigg(&\sum_{i,j=1}^N M_{ij}\na\frac{\mu_i}{\theta}\cdot
	\na\frac{\mu_j}{\theta}	+ \kappa(\theta)|\na\log\theta|^2
	+ \frac{1}{\theta}\St:\na\vel
	- \sum_{i=1}^N r_i\frac{\mu_i}{\theta}\bigg)\,dx \nonumber \\
	&= \alpha_2\int_{\pa\Omega}\frac{\theta-\theta_0}{\theta}\, ds \label{2.ent}
\end{align}
and the total energy equality
\begin{equation}\label{2.ener}
  \alpha_1\int_{\pa\Omega}|\vel|^2 \,ds
	+ \alpha_2\int_{\pa\Omega}(\theta-\theta_0)\,ds
	= \int_\Omega\rhotot\bm{b}\cdot\vel \,dx
\end{equation}
hold.
\end{lemma}

\begin{proof}
We multiply \eqref{2.massbal} by $\mu_i/\theta$ and sum over $i=1,\ldots,N$; 
we add this equation to that one obtained by multiplying \eqref{2.enerbal} by 
$-1/\theta$; and we integrate the resulting equation over $\Omega$. 
After a few integration by parts, the terms involving the coefficients $M_i$ 
cancel, and, taking into account the boundary conditions
\eqref{2.bc.J}--\eqref{2.bc.theta}, it follows that
\begin{align}
  &\int_\Omega\bigg( \rho e \vel\cdot \na\frac{1}{\theta}\bigg) + \sum_{i,j=1}^N M_{ij}\na\frac{\mu_i}{\theta}
	\cdot\na\frac{\mu_j}{\theta}
	+ \kappa(\theta)|\na\log\theta|^2\bigg)\, dx \label{2.aux} \\
	&\phantom{xx}{}
	+ \int_\Omega\bigg(\sum_{i=1}^N \diver(\rho_i\vel)\frac{\mu_i}{\theta}
	-\frac{p}{\theta}\diver\vel + \frac{1}{\theta}\St:\na\vel\bigg)\,dx
	= \alpha_2\int_{\pa\Omega}\frac{\theta-\theta_0}{\theta}\,ds
	+ \int_\Omega\sum_{i=1}^N r_i\frac{\mu_i}{\theta} \,dx. \nonumber
\end{align}
We claim that some of the terms cancel, namely
\begin{equation}\label{2.aux2}
  \int_\Omega\bigg(\rho e\vel\cdot \na\frac{1}{\theta} + \sum_{i=1}^N \diver(\rho_i\vel) \frac{\mu_i}{\theta} - \frac{p}{\theta}\diver\vel\bigg)\,dx = 0.
\end{equation}
To prove this, we use the thermodynamic relations \eqref{1.p} to deduce that
(recall that $h_\theta(\vec{\rho}) = \rho \psi(\vec{\rho},\theta)$)
\begin{align*}
  -\vel\cdot \sum_{i=1}^N & \nabla \rho_i\frac{\mu_i}{\theta}
	- \rho e\vel \cdot\na\frac{1}{\theta} + \frac{p\diver \vel}{\theta}
	-\sum_{i=1}^N \rho_i\mu_i \frac{\diver \vel}{\theta} \\
	&= \frac{\vel}{\theta}\cdot \bigg(\sum_{i=1}^N
	- \frac{\pa h_\theta}{\pa\rho_i} \na\rho_i
	+ \bigg(h_\theta-\theta \frac{\pa h_\theta}{\pa\theta}\bigg)
	\frac{\nabla\theta}{\theta}\bigg)
	- \frac{h_\theta}{\theta} \diver\vel  \\
  &= -\frac{\vel}{\theta}\cdot \na h_\theta - \frac{h_\theta}{\theta}
	\diver \vel - h_\theta\vel\cdot\nabla \frac{1}{\theta}
	= \diver \Big(\frac{\vel h_\theta}{\theta}\Big).
\end{align*}
Hence, we deduce \eqref{2.aux2} after an integration by parts,
and \eqref{2.aux} simplifies to \eqref{2.ent}.

Next, we multiply \eqref{2.mombal} by $\vel$ and add to the resulting equation
the energy balance \eqref{2.enerbal}, integrate over $\Omega$, and integrate
by parts. The integrals involving $\St:\na\vel$ and $p$ cancel and we end up with
$$
  -\int_\Omega\rhotot(\vel\otimes\vel):\na\vel \,dx
	- \int_{\pa\Omega}(\St\vel)\cdot\bm{\nu}\,ds
	= \alpha_2\int_{\pa\Omega}(\theta_0-\theta)\,ds
	+ \int_\Omega\rhotot\bm{b}\cdot\vel \,dx.
$$
The first term vanishes. Indeed, the sum of \eqref{2.massbal} from $i=1,\ldots,N$
yields $\diver(\rhotot\vel)=0$. Multiplying this equation by $|\vel|^2/2$ and
integrating over $\Omega$, an integration by parts gives
$$
  0 = \frac12\int_\Omega\rhotot\vel\cdot\na|\vel|^2 \,dx
	= \int_\Omega\rhotot(\vel\otimes\vel):\na\vel \,dx.
$$
By \eqref{2.bc.vel}, it yields the second identity \eqref{2.ener}, finishing the proof of the lemma.
\end{proof}

The entropy and total energy balances yield some a priori estimates. We define
$$
  q_i = \frac{\mu_i}{\theta}, \quad
	\Pi(\vec{q})_i = q_i - \frac{1}{N}\sum_{j=1}^N q_j, \quad i=1,\ldots,N,
$$
recalling that $\Pi=\mathbb{I}-\vec{1}\otimes\vec{1}/N$ projects onto
$\operatorname{span}\{\vec{1}\}^\perp$.

\begin{lemma}[Estimates from the entropy balance]\label{lem.est1}
The following a priori estimates hold:
\begin{align}
  \|\vel\|_{H^1(\Omega)}
	+ \|\Pi(\vec{q})\|_{H^1(\Omega)} &\le C, \label{2.vq} \\
  \|\na\log\theta\|_{L^2(\Omega)} + \|\na\theta^{\beta/2}\|_{L^2(\Omega)}	
	+ \|1/\theta\|_{L^1(\pa\Omega)} &\le C, \label{2.thetaL2} \\
\|\theta\|_{L^1(\pa\Omega)} + 	\|\log\theta\|_{H^1(\Omega)}
	+ \|\theta^{\beta/2}\|_{H^1(\Omega)}^{2/\beta}
	+ \|\theta\|_{L^{3\beta}(\Omega)}
	&\le C\big(1+\|\rhotot\|_{L^{6/5}(\Omega)}\big),
	\label{2.thetaH1}
\end{align}
where here and in the following, $C>0$ denotes a generic constant dependent only
on the given data.
\end{lemma}

\begin{proof}
We claim that every term on the left-hand side of \eqref{2.ent} is nonnegative.
In view of \eqref{1.M2}, we need to consider only the last two terms.
We deduce from Hypothesis (H3) and the Korn inequality (Lemma \ref{lem.korn}
in Appendix \ref{app.aux}),
taking into account that $\Omega$ is not axially symmetric thanks to
Hypothesis (H1), that for all $\vel \in H^1_{\bm{\nu}}(\Omega;\R^3)$,
\begin{align*}
  \int_\Omega\frac{1}{\theta}\St:\mathbb{D}(\vel)\,dx
	\ge C\|\vel\|_{H^1}^2.
\end{align*}

The $L^2(\Omega)$ bound for $\na \Pi (\vec{q})$ is a consequence of
\eqref{1.M2}, and \eqref{1.r} gives an $L^2(\Omega)$ bound for $\Pi (\vec q)_i$,
$i=1,\ldots,N$. At this point, we need the nonvanishing reaction terms.
Thus, $\Pi(\vec q)$ is bounded in $H^1(\Omega)$.

By Hypothesis (H3),
\begin{equation*}
  \int_\Omega \kappa(\theta)|\na\log\theta|^2 \,dx
	\ge \kappa_1\int_\Omega\bigg(|\na\log\theta|^2
	+ \frac{4}{\beta^2}|\na\theta^{\beta/2}|^2\bigg)\,dx,
\end{equation*}
which gives the $L^2(\Omega)$ bounds for $\na\log\theta$ and $\na\theta^{\beta/2}$.
The entropy balance \eqref{2.ent} implies that $1/\theta$ is bounded in $L^1(\pa\Omega)$.
The total energy balance \eqref{2.ener} and the $H^1(\Omega)$ bound for $\vel$
together with the continuous embedding $H^1(\Omega)\hookrightarrow L^6(\Omega)$
show that
$$
  \int_{\pa\Omega}\alpha_2\theta \,ds \le \int_{\pa\Omega}\alpha_2\theta_0 \,ds
	+ C\|\bm{b}\|_{L^\infty(\Omega)}\|\rhotot\|_{L^{6/5}(\Omega)}
	\|\vel\|_{H^1(\Omega)} \le C(1+\|\rhotot\|_{L^{6/5}(\Omega)}).
$$
If $\beta<2$, we conclude that
$$
  \|\log\theta\|_{L^1(\pa\Omega)}+\|\theta^{\beta/2}\|_{L^1(\pa\Omega)}
	\le C\big(1+\|\theta\|_{L^1(\pa\Omega)}\big)
	\le C\big(1+\|\rhotot\|_{L^{6/5}(\Omega)}\big),
$$
and the Poincar\'e inequality yields the remaining estimates.
If $\beta>2$, we find that
\begin{align*}
  \int_\Omega|\na\theta|^2 dx
	&= \int_{\{\theta\le K\}}\theta^2|\na\log\theta|^2\, dx
  + \frac{4}{\beta^2}\int_{\{\theta>K\}}\theta^{2-\beta}|\na\theta^{\beta/2}|^2\, dx \\
	&\le K^2\int_\Omega|\na\log\theta|^2\, dx
	+ \frac{4}{\beta^2}K^{2-\beta}\int_\Omega|\na\theta^{\beta/2}|^2\, dx \le C.
\end{align*}
By the Poincar\'e inequality, $\|\theta\|_{H^1(\Omega)}\le
C(1+\|\rhotot\|_{L^{6/5}(\Omega)})$, and this allows us to control also
the $L^6(\Omega)$ norm of $\theta$. A bootstrapping argument then yields a control
of the $L^1(\Omega)$ norm of $\theta^{\beta/2}$. Applying the Poincar\'e inequality
to $\theta^{\beta/2}$ finishes the proof.
\end{proof}

Exploiting the bounds \eqref{prrr} for the pressure, we are able to derive an
$L^s(\Omega)$ bound for $\rhotot$ with $s>\gamma$, provided that $\gamma > 3/2$ 
and $\beta>2/3$.

\begin{lemma}[Estimate for the total mass density]\label{lem.dens}
Let $\nu$ be defined by \eqref{def.nu}. Then
there exists a constant $C>0$ depending only on the given data such that
$$
  \|\rhotot\|_{L^{\gamma+\nu}(\Omega)} \le C.
$$
\end{lemma}

\begin{proof}
The proof is based on estimates from the momentum balance \eqref{2.mombal}
using the Bogovskii operator $\B$. We refer to Theorem \ref{thm.bog} in
Appendix \ref{app.aux} for some properties of this operator. We multiply
\eqref{2.mombal} by $\bm{\phi}=\B(\rhotot^\nu-\langle \rhotot^\nu\rangle)$, where
$\langle\rhotot^\nu\rangle=|\Omega|^{-1}\int_\Omega\rhotot^\nu \,dx$. Then
\begin{equation}\label{2.aux4}
  \int_\Omega p\rhotot^\nu \,dx
	= \int_\Omega\big(p\langle\rhotot^\nu\rangle - \rhotot(\vel\otimes\vel):\na\bm{\phi}
	+ \St:\na\bm{\phi} - \rhotot\bm{b}\cdot\bm{\phi}\big)\,dx.
\end{equation}
 Recall that, by Hypothesis (H7), $\int_\Omega p \rhotot^\nu \, dx \geq c_p \int_\Omega (\rhotot^{\gamma +\nu} + \rhotot^{1+\nu}\theta) \, dx$.
We estimate the right-hand side of \eqref{2.aux4} term by term.
We start with two
delicate terms which lead to the restrictions on the exponent $\nu$.
We have for $\alpha>3/2$,
\begin{align*}
  \bigg|\int_\Omega\rhotot(\vel\otimes\vel):\na\bm{\phi}\,dx\bigg|
	&\le \|\vel\|_{L^6(\Omega)}^2\|\rhotot\|_{L^\alpha(\Omega)}
	\|\na\bm{\phi}\|_{L^{3\alpha/(2\alpha-3)}(\Omega)} \\
	&\le C\|\vel\|_{L^6(\Omega)}^2\|\rhotot\|_{L^\alpha(\Omega)}
	\|\rho^\nu - \langle\rho^\nu\rangle\|_{L^{3\alpha/(2\alpha-3)}(\Omega)} \\
	&\le C\|\vel\|_{L^6(\Omega)}^2\|\rhotot\|_{L^\alpha(\Omega)}
	\big(\|\rhotot^\nu\|_{L^{3\alpha/(2\alpha-3)}(\Omega)}
	+ \|\rhotot^\nu\|_{L^1(\Omega)}\big),
	\end{align*}
and choosing $\alpha = \gamma+\nu_1$ and $3\alpha/(2\alpha-3)
= (\gamma+\nu_1)/\nu_1$, we end up with $\nu_1 = 2\gamma-3$.

Next, in view of Hypothesis (H3),
	\begin{align*}
  \bigg|\int_\Omega\St:\na\bm{\phi}\,dx\bigg|
	&\le C(1+\|\theta\|_{L^{3\beta}(\Omega)}) \|\na \vel\|_{L^2(\Omega)} \|\na\bm{\phi}\|_{L^{6\beta/(3\beta-2)}(\Omega)} \\
	&\le C(1+\|\rhotot\|_{L^{6/5}(\Omega)}) \big(\|\rhotot^{\nu_2}\|_{L^{6\beta/(3\beta-2)}(\Omega)}
	+ \|\rhotot^{\nu_2}\|_{L^1(\Omega)}\big) \\
	&\le C\big(\|\rhotot\|_{L^{\gamma+\nu_2}(\Omega)}^{1+\nu_2}
	+ \|\rhotot\|_{L^1(\Omega)}\big),
	\end{align*}
	provided $6\beta/(3\beta-2) = (\gamma+\nu_2)/\nu_2$, i.e. $\nu_2 = \gamma (3\beta-2)/(3\beta+2)$. Furthermore, because of $3/2<\gamma+\nu$,
	\begin{align*}
  \bigg|\int_\Omega\rhotot\bm{b}\cdot\bm{\phi}\,dx\bigg|
	&\le \|\bm{b}\|_{L^\infty(\Omega)}\|\rhotot\|_{L^{3/2}(\Omega)}
	\|\bm{\phi}\|_{L^{3}(\Omega)}
	\le C\|\bm{b}\|_{L^\infty(\Omega)}\|\rhotot\|_{L^{3/2}(\Omega)}
	\|\bm{\phi}\|_{W^{1,3/2}(\Omega)} \\
	&\le C\|\bm{b}\|_{L^\infty(\Omega)}\|\rhotot\|_{L^{\gamma + \nu}(\Omega)}
	\big(\|\rhotot^\nu\|_{L^{3/2}(\Omega)} + \|\rhotot^\nu\|_{L^1(\Omega)}\big)
	\leq C (\|\rhotot\|_{L^{\gamma + \nu}(\Omega)}^{1+\nu}+1),
\end{align*}
since the restriction $\nu \leq 2\gamma-3$ yields $\frac 32 \nu < \gamma + \nu$.
Finally, by \eqref{prrr},
$$
  \bigg|\int_\Omega p\langle\rhotot^\nu\rangle \,dx\bigg|
	\le C_p\int_\Omega(1+\rhotot \theta+\rhotot^\gamma)\,dx\int_\Omega\rhotot^\nu \,dx.
$$
As we control the $L^1(\Omega)$ norm of the density (see \eqref{1.int}),
we can control $\int_{\Omega} \rho^\gamma \, dx \int_{\Omega} \rho^\nu \, dx$ by
$C\|\rhotot\|_{L^{\gamma +\nu}(\Omega)}^\lambda$ for some $\lambda < \gamma +\nu$,
by interpolating between the $L^1$ and $L^{\gamma +\nu}$ norms.
Hence, we only need to deal with the part of the first integral containing the temperature.
	
Let us first consider the case $\nu \leq 1$. Then $\int_\Omega\rhotot^\nu \,dx$
is bounded by a constant and
$$
  \int_\Omega\rhotot\theta\,dx = \int_{\{\rhotot\le K\}}\rhotot\theta\,dx
	+ \int_{\{\rhotot > K\}}\rhotot\theta\,dx
	\le K\int_\Omega\theta\,dx + K^{-\nu}\int_\Omega\theta\rhotot^{1+\nu}\,dx.
$$
The first term on the right-hand side is bounded, and the last term can be absorbed
by the left-hand side of \eqref{2.aux4} for sufficiently large $K$.
Next, note that for $\nu >1$ we have $2\gamma-3 >1$, i.e.\
$\gamma >2$. Then, by H\"older's inequality,
$$
  \int_\Omega \rhotot \theta\,dx \int_\Omega \rhotot^\nu \, dx
	\leq C \|\rhotot\|_{L^{\gamma}(\Omega)} \|\theta\|_{L^{\gamma/(\gamma-1)}(\Omega)}
  \|\rhotot\|_{L^{\nu}(\Omega)}.
$$
It follows from $\gamma >2$ and $\beta > 2/3$ that $\gamma/(\gamma-1) <3\beta$.
Hence, using once more $\gamma >2$,
$$
  \int_\Omega \rhotot \theta\,dx \int_\Omega \rhotot^\nu \,dx
	\leq C \|\rhotot\|_{L^{\gamma+\nu}(\Omega)}^{1+\nu} \|\rhotot\|_{L^{6/5}(\Omega)}
	\leq C   \|\rhotot\|_{L^{\gamma+\nu}(\Omega)}^{2+\nu}.
$$
Collecting all estimates, we deduce from $2+\nu<\gamma+\nu$ that
$$
  c_p \int_\Omega \rhotot^{\gamma +\nu}\, dx
	\le \int_\Omega p\langle\rhotot^\nu\rangle \, dx
	\leq C(\|\rhotot\|_{L^{\gamma + \nu}(\Omega)}^{\lambda}+1),
$$
where $\lambda < \gamma + \nu$. This leads to the desired estimate of $\rhotot$.
\end{proof}
	
\begin{lemma}[Estimate for the pressure]\label{lem.p}
For
$$
\alpha = \min\Big\{1+ \frac{\nu}{\gamma}, \frac{(1+\nu)3\beta}{\nu +3\beta}\Big\} >1,
$$
there exists a constant $C>0$  such that
$$
  \|p(\vec\rho,\theta)\|_{L^{\alpha}(\Omega)} \le C.
$$
\end{lemma}

\begin{proof}
Because of \eqref{prrr}, we have
$p(\vec\rho,\theta) \le C_p(1+\rhotot^\gamma + \rhotot \theta)$.
Taking into account Lemma \ref{lem.dens}, it is sufficient to verify that
$\rhotot \theta \in L^{\alpha}(\Omega)$:
\begin{align*}
\int_\Omega (\rhotot \theta)^{\alpha} \, dx  &= \int_\Omega (\rhotot \theta^{1/(1+\nu)})^{\alpha} \theta^{\alpha \nu/(1+\nu)} \, dx  \\
&\leq \Big(\int_\Omega \rhotot^{1 + \nu} \theta \,dx\Big)^{\alpha/(1+\nu)} \Big(\int_\Omega \theta^{\alpha \nu/(1+\nu-\alpha)} \, dx\Big)^{(1+\nu-\alpha)/(1+\nu)} \leq C,
\end{align*}
provided that $\alpha \nu/(1+\nu-\alpha) \leq 3\beta$. This is true if
$\alpha = 3\beta(1+\nu)/(\nu +3\beta)$.
\end{proof}


\section{Weak sequential compactness for smooth solutions}\label{sec.comp}

In this section, we focus on the weak sequential stability of a weak solution and formulate it as an independent result. Then, in Section~\ref{sec.scheme}, we will just adapt the method introduced here and use it to prove the existence of a weak solution.  The main result of this part is the following theorem.

\begin{theorem}\label{CC.thm} Let Hypotheses (H1)--(H7) be satisfied. Let the sequence
$(\bm{b}_{\delta},\overline\rho_{\delta}, (\theta_0)_\delta)$  fulfil
\begin{equation}\label{MB1}
\begin{aligned}
\bm{b}_{\delta} &\to \bm{b} &&\textrm{strongly in } L^{p}(\Omega;\R^3) \textrm{ for all } p<\infty,\\
\bm{b}_{\delta} &\rightharpoonup^* \bm{b} &&\textrm{weakly}^* \textrm{ in } L^{\infty}(\Omega;\R^3),\\
\overline\rho_{\delta} &\to \overline\rho >0  &&\textrm{in }\mathbb{R},\\
(\theta_0)_\delta & \to \theta_0 &&\textrm{strongly in } L^{1}(\partial \Omega).
\end{aligned}
\end{equation}
Let $(\vec{\rho}_\delta,\vel_\delta,\theta_\delta)$ be a sequence of weak solutions
to \eqref{2.massbal}--\eqref{2.bc.theta},
corresponding to $\bm{b}_{\delta}$, $\overline\rho_{\delta}$ and $(\theta_0)_\delta$.
Let $\gamma > 3/2$ and $\beta>2/3$. Then $(\vec{\rho}_\delta,\vel_\delta,
\theta_\delta)$ satisfies the uniform bounds stated in  Lemmata
\ref{lem.est1}--\ref{lem.p}, and there exists a subsequence (not relabeld) such that
for $s = \min\{3\beta/(\beta+1),2\} \in (1,2]$,
\begin{equation}\label{cA:1}
\begin{aligned}
  \vec\rhotot_\delta\rightharpoonup \vec\rhotot
	&\quad\mbox{weakly in }L^{\gamma+\nu}(\Omega;\R^N), \quad \nu = \nu (\beta,\gamma) \text{ is from  Lemma } \ref{lem.dens}, \\
	\vel_\delta\rightharpoonup\vel &\quad\mbox{weakly in }H^1(\Omega;\R^3)
	\mbox{ and strongly in }L^q(\Omega;\R^3), \quad q<6, \\
  \theta_\delta\rightharpoonup \theta &\quad\mbox{weakly in }W^{1,s}(\Omega)
	\mbox{ and strongly in }L^q(\Omega), \quad q<3\beta,
\end{aligned}
\end{equation}
where the triple $(\vec{\rho},\vel,\theta)$ is a variational entropy solution corresponding to $(\bm{b},\overline\rho, \theta_0)$. In addition, if $\gamma > 5/3$ and $\beta>1$, then it is also a weak solution. Moreover, $\vec{\rho_\delta} \to \vec{\rho}$ strongly in $L^1(\Omega;\R^N)$.
\end{theorem}

We shall prove the theorem in several steps and each step is described in one of the following subsections.

\subsection{Convergence results based on a priori estimates}\label{sec.conv.smooth}

Based on Lemmata \ref{lem.est1}--\ref{lem.p}, we collect all weak convergence results that allow us to pass to the limit in the weak formulation. It remains to show that the densities $\vec\rho_\delta$ converge pointwise to identify the pressure and the chemical potentials as functions of the densities and the temperature.

\subsubsection*{Limit in the mass balance}

First, using the facts that (up to subsequences) $\vec\rho_\delta \to \vec \rho$ weakly in $L^{\gamma +\nu}(\Omega;\R^N)$, $\theta_\delta \to \theta$ strongly in $L^r(\Omega)$, $1\leq r<3\beta$, together with Hypothesis (H4), we see that
$$
  M_{ij}(\vec{\rho_\delta},\theta_\delta) \rightharpoonup
	\overline{M_{ij} (\vec{\rho_\delta},\theta_\delta)}, \quad
	M_i(\vec{\rho_\delta},\theta_\delta)/\theta_\delta \rightharpoonup
	\overline{M_{i} (\vec{\rho_\delta},\theta_\delta)/\theta_\delta}
  \quad\mbox{weakly in }L^1(\Omega),
$$
where $i$, $j=1,2,\dots, N$ as $\delta \to 0$ and a bar over a quantity denotes
its weak limit. Since the partial densities converge only weakly, we cannot generally
identify the weak limits with $M_{ij}(\vec{\rho},\theta)$ and
$M_i(\vec{\rho},\theta)$, respectively. Furthermore, due to \eqref{2.vq},
$$
\Pi\Big(\frac{\vec{\mu_\delta}}{\theta_\delta}\Big) \to \Pi(\vec{q})\quad
\mbox{strongly in }L^r(\Omega;\R^N),\ r<6,
$$
and we deduce from Hypothesis (H5) that
$$
\vec{r}\Big(\Pi\Big(\frac{\mu_\delta}{\theta_\delta}\Big),\theta_\delta\Big)
\to \vec{r}(\Pi(\vec{q}),\theta) \quad\mbox{strongly in }L^1(\Omega;\R^N).
$$
We do  not know at this moment whether $\vec{q} = \vec{\mu}/\theta$, where $\vec{\mu}$ is given by \eqref{1.p}. Therefore, letting $\delta \to 0$ in the weak formulation \eqref{w.rhoi} of the mass balance,
we infer from Hypotheses (H4) and (H5) that, for all
$\phi_1$,\dots,$\phi_N\in W^{1,\infty}(\Omega)$,
\begin{equation*}
   \sum_{i=1}^N\int_\Omega\bigg(-\rho_i\vel + \sum_{j=1}^N\overline{M_{ij}
	(\vec{\rho_\delta},\theta_\delta)}
	\na q_j - \overline{\frac{M_{i} (\vec{\rho_\delta},\theta_\delta)}{\theta_\delta}}
	\frac{\nabla \theta}{\theta}\bigg)
	\cdot\na\phi_i \, dx
	= \sum_{i=1}^n\int_\Omega r_i(\vec{q},\theta)\phi_i \, dx. 
\end{equation*}

\subsubsection*{Limit in the momentum balance}

By Lemma \ref{lem.p},
\begin{align}\label{MB3}
	p_\delta:=p(\vec\rho_\delta,\theta_\delta)\rightharpoonup \overline{p(\vec\rho_\delta,\theta_\delta)} =:p
	&\quad\mbox{weakly in }L^{\alpha}(\Omega)\mbox{ as }\delta\to 0.
\end{align}
At this point, it is not clear whether $\overline{p(\vec{\rho_\delta},\theta_\delta)}=:p = p(\rho,\theta)$ and this will be
proved later.
The weak convergence of (a subsequence of) $\vel_\delta$ in $H^1(\Omega)$
and the strong convergence of $(\theta_\delta)$ in $L^r(\Omega)$, $r\geq 2$, 
imply that
\begin{equation} \label{ConSt}
  \St_\delta \rightharpoonup \St\quad\mbox{weakly in }L^q(\Omega)
	\quad\mbox{for some }q\in [1,2),
\end{equation}
where $\St_\delta$ and $\St$ are the stress tensors \eqref{1.S} associated to
$\theta_\delta,\vel_\delta$ and $\theta,\vel$, respectively. We use Hypothesis (H3)
to control the viscosities.
Therefore, we can perform the limit $\delta\to 0$ in the weak formulation
\eqref{w.rhov} of the momentum balance; hence for all
$\bm{u}\in W_{\bm{\nu}}^{1,\infty}(\Omega)$,
\begin{equation*}
\int_\Omega(-\rhotot\vel\otimes\vel+\St(\theta, \nabla \vel)):\na\bm{u}\,dx
+ \int_{\pa\Omega}\alpha_1\vel\cdot\bm{u}\,ds
= \int_\Omega(\overline{p(\vec\rho_\delta,\theta_\delta)}\diver\bm{u} + \rhotot\bm{b}\cdot\bm{u})\,dx.
\end{equation*}

\subsubsection*{Limit in the entropy inequality}

In view of \eqref{1.p}, Hypothesis (H6) (in particular \eqref{grow_h_1} and
\eqref{grow_h_2}), and the bounds on the temperature and densities, we obtain
$$
\partial_\theta h_{\theta_\delta} (\vec{\rhotot_\delta}) = \rhotot_\delta s(\vec{\rhotot_\delta},\theta_\delta) \rightharpoonup \overline{\rhotot_\delta s(\vec{\rhotot_\delta},\theta_\delta)}\quad\mbox{weakly in }L^q(\Omega),\ q>6/5.
$$
Using the weak lower semicontinuity in several terms, the previous weak limits,
Hypotheses (H3)--(H6), and Lemma \ref{lem.fatou}
we conclude from \eqref{w.rhos} in the limit $\delta\to 0$ that
\begin{align*}
\int_\Omega &\bigg[\overline{\rhotot_\delta s(\vec{\rhotot_\delta},\theta_\delta)} \vel + \sum_{i=1}^N q_i \bigg( \sum_{j=1}^N \overline{\frac{M_{ij} (\vec{\rho_\delta},\theta_\delta)}{\theta_\delta}}\na q_j
- \overline{\frac{M_{i} (\vec{\rho_\delta},\theta_\delta)}{\theta_\delta}}\frac{\na \theta}{\theta} \bigg) \nonumber \\
&\phantom{xx}{}- \bigg( \frac {\kappa(\theta)}{\theta}\na\theta + \sum_{i=1}^N \overline{\frac{M_{i} (\vec{\rho_\delta},\theta_\delta)}{\theta_\delta}}\na q_i\bigg)\bigg] \cdot\nabla\Phi\, dx \nonumber \\
&\phantom{xx}{}+\int_\Omega \bigg(\sum_{i,j=1}^N
\overline{M_{ij}(\vec\rhotot_\delta,\theta_\delta)}\nabla \frac{\mu_i}{\theta}\cdot\nabla \frac{\mu_j}{\theta}
+ \kappa(\theta)|\nabla\log\theta|^2
+ \frac{\St : \na\vel}{\theta} -\sum_{i=1}^N r_i \frac{\mu_i}{\theta}\bigg) \Phi\, dx \nonumber  \\
& \leq \alpha_2\int_{\pa\Omega}\frac{\theta-\theta_0}{\theta}\Phi\, ds,
\end{align*}
for every $\Phi\in W^{1,\infty}(\Omega)$, $\Phi\geq 0$ a.e.~in $\Omega$.

\subsubsection*{Limit in the total energy balance}

The problem with the total energy balance is more complex. We can easily pass to
the limit if the test function is constant, yielding the
global total energy equality \eqref{t.rhovE}. To obtain
a suitable limit in the weak formulation \eqref{w.rhovE} of the total energy balance,
we have to assume that $\gamma > 5/3$ and $\beta>1$. This ensures that
$\gamma+\nu>2$ and
\begin{align*}
  \rhotot_\delta |\vel_\delta|^2 \vel_\delta \rightharpoonup \rhotot |\vel|^2 \vel
  &\quad\mbox{weakly in }L^r(\Omega;\R^3), \\
  \St(\theta_\delta,\nabla \vel_\delta)\vel_\delta \rightharpoonup \St(\theta,\nabla \vel)\vel &\quad\mbox{weakly in }L^r(\Omega;\R^3)
\end{align*}
for some $r>1$. Moreover, in view of Hypothesis (H6),
$$
\rhotot_\delta e(\vec{\rhotot_\delta},\theta_\delta)\vel_\delta \rightharpoonup \overline{\rhotot_\delta e(\vec{\rhotot_\delta},\theta_\delta)}\vel
\quad \mbox{weakly in }L^r(\Omega;\R^3)
$$
for some $r>1$. Therefore, letting $\delta \to 0$ in \eqref{w.rhovE}, it follows for all
$\varphi\in W^{1,\infty}(\Omega)$ that
\begin{align*}
 \int_\Omega&\bigg(-\frac 12 \rhotot  |\bm{v}|^2 \bm{v} - \overline{\rhotot_\delta e(\vec{\rhotot_\delta},\theta_\delta)}\vel+ \kappa(\theta) \nabla \theta  + \sum_{i=1}^N \overline{\frac{M_{i} (\vec{\rho_\delta},\theta_\delta)}{\theta_\delta}}\na q_i
\nonumber \\
&{}+\St \bm{v} -\overline{p(\vec\rho_\delta,\theta_\delta)}\bm{v}\bigg)\cdot\na \varphi \, dx
+\int_{\pa\Omega}(\alpha_1|\bm{v}|^2+ \alpha_2 (\theta-\theta_0))\varphi\, ds
= \int_\Omega\rhotot\bm{b}\cdot\bm{v} \varphi\, dx. 
\end{align*}

It remains to verify that
\begin{equation}\label{MB:end}
\vec\rhotot_\delta \to \vec\rhotot \quad \textrm{a.e. in } \Omega
\end{equation}
as well as to identify $\Pi(\vec{q})$ with $\Pi(\vec{\mu}/\theta)$, where $\vec{\mu}$ is given by \eqref{1.p}.
The rest of this section is devoted to the proof of \eqref{MB:end}.
For the sake of notational simplicity, we introduce the following notation
for the double limit $(\delta,\eps)\to 0$, i.e.,
$\overline{f(\vec\rho_\varepsilon,\vec\rho_\delta)}\in L^1(\Omega)$ is defined by
$$
\int_{\Omega} \overline{f(\vec \rho_\varepsilon,\vec\rho_\delta)} \chi\, dx := \lim_{\varepsilon \to 0} \lim_{\delta \to 0} \int_{\Omega}f(\vec\rho_\varepsilon,\vec\rho_\delta) \chi \, dx \quad \textrm{for all }\chi \in L^{\infty}(\Omega).
$$
We also do not mention explicitly that we are working with subsequences and therefore, we do not relabel any sequence. Since we use only countably many relabelings, such procedure can be made rigorous by the standard diagonal procedure.
Note that in all cases considered below, the order of the limit passages is not important; for the sake of clarity, we will assume that we let first $\delta \to 0$ and afterwards $\varepsilon \to 0$.

To end this first part, we also introduce the truncation function $T_k:\R_+\to\R_+$ of class $C^1(\R_{+})$, which will be needed later. For arbitrary $k\in\N$ and
$z\ge 0$, we set
\begin{equation}\label{def.Tk}
\begin{aligned}
  T_1(z) &:= \left\{\begin{array}{ll}
	z &\quad\mbox{for }0\le z\le 1, \\
	\mbox{concave, increasing, $C^1$-function} &\quad\mbox{for }1<z<3, \\
	2 &\quad\mbox{for }z\ge 3.
	\end{array}\right.\\
T_k(z)&:=kT_1(z/k).
\end{aligned}
\end{equation}


\subsection{Effective viscous flux}

We first focus on an effective viscous flux identity. We follow the procedure
developed in \cite{Fei01,Lio98} very closely, and the proof is presented here
for the sake of completeness.

\begin{lemma}[Effective viscous flux identity]\label{lem.eff}\sloppy
Let $(\bm{b}_{\delta},\rhotot_{\delta}, \St_{\delta}, \theta_{\delta}, \vel_{\delta})$ satisfy \eqref{1.S}, 
\eqref{MB1}, \eqref{cA:1}, \eqref{MB3}, and \eqref{ConSt}. Then it holds for every $k\in \mathbb{N}$ and $T_k$,
defined in \eqref{def.Tk}, that
\begin{equation}\label{evf}
  \overline{p_\delta T_k(\rhotot_\delta)}- p\overline{T_k(\rho_\delta)}
	= \bigg(\lambda_2(\theta)+\frac43\lambda_1(\theta)\bigg)
	\big(\overline{T_k(\rhotot_\delta)\diver \vel_\delta}
	- \overline{T_k(\rhotot_\delta)}\diver\vel\big).
\end{equation}
\end{lemma}

\begin{proof}
Thanks to our assumptions and since we consider the proper, not relabeled subsequence, all terms in \eqref{evf} are well defined.
We introduce an auxiliary function $\phi_\delta$ as the solution to
\begin{equation}\label{sol1}
  \Delta\phi_\delta = T_k(\rhotot_\delta)\quad\mbox{in }\Omega, \quad
	\phi_\delta=0\quad\mbox{on }\pa \Omega.
\end{equation}
As $\pa\Omega$ is of class $C^2$, we have $\phi_\delta\in W^{2,q}(\Omega)$
for all $q<\infty$. (The proof would also work for open bounded domains $\Omega$
by arguing locally, since the regularity holds true away from the boundary.)
Since $(T_k(\rho_\delta))$ is bounded in $L^\infty(\Omega)$,
the sequence $(\phi_\delta)$ is bounded in $W^{2,q}(\Omega)$ for all $q<\infty$,
implying, up to a subsequence, that $\phi_\delta\to\phi$ weakly in $W^{2,q}(\Omega)$
and strongly in $W^{1,q}(\Omega)$ for all $q<\infty$, where $\phi$ solves
\begin{equation}\label{sol2}
  \Delta\phi = \overline{T_k(\rhotot_\delta)}\quad\mbox{in }\Omega, \quad
	\phi=0\quad\mbox{on }\pa \Omega.
\end{equation}
We set $\vcg{\psi}_\delta:=\na\phi_\delta$ and $\vcg{\psi}:=\na\phi$. The convergence properties of $\phi_{\delta}$ yield the strong convergence $\vcg{\psi}_{\delta}\to \vcg{\psi}$ in $L^r(\Omega; \R^3)$ for any $1\leq r\leq \infty$ and the weak convergence $\na \vcg{\psi}_{\delta} \rightharpoonup \na \vcg{\psi}$ in $L^{q}(\Omega; \mathbb{R}^{3\times 3})$ for all $q<\infty$.

Relation \eqref{w.rhov} is a weak formulation of
$$
  \diver(-\mathbb{T}_\delta + \rhotot_\delta\vel_\delta\otimes\vel_\delta)
	= \rhotot_\delta\bm{b}_{\delta},
	\quad\mbox{where }\mathbb{T}_\delta = -p_\delta\mathbb{I} + \St_\delta,
$$
and therefore, we can apply the div-curl lemma (see Lemma \ref{lem.divcurl})
to the matrix-valued functions
$\mathbb{T}_\delta-\rhotot_\delta\vel_\delta\otimes\vel_\delta$
and $\na\vcg{\psi}_\delta$. Since the divergence of the former sequence
is bounded in $L^{r}(\Omega;\R^3)$ for some $r>1$
and the curl of $\na\vcg{\psi}_\delta$ vanishes, we infer that
$$
  (\mathbb{T}_\delta-\rhotot_\delta\vel_\delta\otimes\vel_\delta):\na\vcg{\psi}_\delta
  \rightharpoonup \overline{\mathbb{T}_\delta
	-\rhotot_\delta\vel_\delta\otimes\vel_\delta}:\na\vcg{\psi}
	\quad\mbox{weakly in }L^1(\Omega).
$$
Since $\vel_\delta\otimes\vel_\delta\to \vel\otimes\vel$ strongly in $L^{q/2}(\Omega)$
for $q<6$ and $\rhotot_\delta\rightharpoonup\rhotot$ weakly in $L^{\gamma+\nu}(\Omega)$
for $\gamma>3/2$, the product converges weakly to the product of the limits, i.e.
$$
  \overline{\rhotot_\delta\vel_\delta\otimes\vel_\delta}
	=\rhotot\vel\otimes\vel \quad\mbox{in }L^1(\Omega).
$$
Hence,
\begin{equation}\label{3.aux1}
  \overline{(\mathbb{T}_\delta-\rhotot_\delta\vel_\delta\otimes\vel_\delta)
	:\na\vcg{\psi}_\delta} = \overline{\mathbb{T}_\delta}:\na\vcg{\psi}
	- (\rhotot\vel\otimes\vel):\na\vcg{\psi} \quad \textrm{a.e.\ in }\Omega,
\end{equation}
which is the starting point of further investigations.

First, we focus on the term involving the tensorial product of velocities. Note that
the sum of \eqref{w.rhoi} gives $\diver(\rhotot_\delta\vel_\delta)=0$ in the weak
sense and that
$$
 \operatorname{curl}(\na\vcg{\psi}_\delta \vel_\delta)=\na \left(\na\vcg{\psi}_\delta \vel_\delta\right) - \left(\na \left( \na\vcg{\psi}_\delta \vel_\delta\right)\right)^T=  \na\vcg{\psi}_\delta \left(\na \vel_\delta\right)^T-\na \vel_\delta\left(\na\vcg{\psi}_\delta \right)^T
$$
is bounded in $L^q(\Omega; \mathbb{R}^{3\times 3})$ for $q<2$ due to the properties of $\vcg{\psi}_{\delta}$ and \eqref{cA:1}. Second, \eqref{cA:1} implies that
$(\rhotot_\delta\vel_\delta)$ is bounded in $L^{s}(\Omega;\R^3)$ for some $s>6/5$ and
$(\na\vcg{\psi}_\delta \vel_\delta)$ is bounded in $L^{q}(\Omega;\R^3)$ for all $q<6$. Therefore, using the div-curl lemma again,
$$
  \overline{(\rhotot_\delta\vel_\delta\otimes\vel_\delta):\na\vcg{\psi}_\delta}
  = \overline{\rhotot_\delta\vel_\delta}\cdot\overline{\na\vcg{\psi}_\delta \vel_\delta} = (\rhotot\vel) \cdot (\na\vcg{\psi} \vel)
	= (\rhotot\vel\otimes\vel):\na\vcg{\psi} \quad\mbox{in }L^1(\Omega)
$$
(and thus a.e.\ in $\Omega$),
where the second equality follows from the a.e.\ convergence of $(\vel_{\delta})$. Hence, we deduce from \eqref{3.aux1} that
$$
  \overline{\mathbb{T}_\delta:\na\vcg{\psi}_\delta} = \overline{\mathbb{T}_\delta}:\na\vcg{\psi}.
$$
Recall that
$\overline{\mathbb{T}_\delta}=\overline{\St_\delta}-\overline{p_\delta}\mathbb{I}
=\overline{\St_\delta}-p\mathbb{I}$. In view of the definitions of
$\vcg{\psi}_{\delta}$ and $\vcg{\psi}$, this shows that
\begin{equation*}
  \overline{\St_\delta:\na^2\phi_\delta - p_\delta\Delta\phi_\delta}
	= \overline{\St_\delta}:\na^2\phi - p\Delta\phi.
\end{equation*}
Finally, by the definitions of $\phi_\delta$ and $\phi$ (see \eqref{sol1} and \eqref{sol2}, respectively), we obtain
\begin{equation}\label{start2}
  \overline{p_\delta T_k(\rhotot_\delta)} - p\overline{T_k(\rhotot_\delta)}
	=\overline{\St_\delta:\na^2\phi_\delta} -  \overline{\St_\delta}:\na^2\phi
	\quad \textrm{a.e.\ in }\Omega.
\end{equation}
The left-hand side corresponds to that one of \eqref{evf}. It remains to identify
the terms on the right-hand side.

The right-hand side is uniquely defined, so we just need to identify it almost everywhere in $\Omega$. Using convergences \eqref{cA:1} and the Egorov theorem, we can find for any $\varepsilon>0$ a measurable set $\Omega_{\varepsilon}\subset \Omega$ such that $|\Omega\setminus \Omega_{\varepsilon}|\le \varepsilon$ and
$\theta_{\delta} \to \theta$ strongly in $L^{\infty}(\Omega_{\varepsilon})$.
Consequently, using definition \eqref{1.S} of the viscous stress tensor,
the previous convergence result, and convergences \eqref{cA:1} again, we can
identify the weak limits in $\Omega_{\varepsilon}$ and conclude
(with the help of \eqref{sol1} and \eqref{sol2}) that
\begin{align*}
\overline{\St_\delta:\na^2\phi_\delta}
&=2\lambda_1(\theta)\overline{\bigg(\mathbb{D}(\vel_\delta) - \frac13\diver\vel_\delta\mathbb{I}\bigg):\nabla^2 \phi_{\delta}}
	+ \lambda_2(\theta)\overline{\diver\vel_\delta\mathbb{I}:\nabla^2 \phi_{\delta}} \\
&=2\lambda_1(\theta)\overline{\mathbb{D}(\vel_\delta):\nabla^2 \phi_{\delta}}
	+ \bigg(\lambda_2(\theta)-\frac23\lambda_1(\theta)\bigg)\overline{\diver\vel_\delta\Delta \phi_{\delta}}, \\
\overline{\St_\delta}:\na^2\phi &=2\lambda_1(\theta)\mathbb{D}(\vel):\nabla^2 \phi
	+ \bigg(\lambda_2(\theta)-\frac23\lambda_1(\theta)\bigg)\diver\vel\Delta \phi
	\quad\mbox{a.e. in }\Omega_\eps.
\end{align*}
Then, substituting these expressions into \eqref{start2},
\begin{align}
  \overline{p_\delta T_k(\rhotot_\delta)} - p\overline{T_k(\rhotot_\delta)}
	&=2\lambda_1(\theta)\left(\overline{\mathbb{D}(\vel_\delta)
	:\nabla^2\phi_{\delta}}-\mathbb{D}(\vel):\nabla^2 \phi\right) \nonumber \\
	&\phantom{xx}{}+ \bigg(\lambda_2(\theta)-\frac23\lambda_1(\theta)\bigg)
	\big(\overline{\diver\vel_\delta\Delta \phi_{\delta}} -\diver\vel\Delta \phi\big)
	\quad\mbox{a.e. in }\Omega_\eps.
  \label{start3}
\end{align}
But since $|\Omega\setminus \Omega_{\varepsilon}|\le \varepsilon$ for any
$\varepsilon>0$, relation \eqref{start3} holds a.e.\ in $\Omega$.

Thus, it remains to identify the first term, i.e.,
we wish to relate the difference involving $\na^2\phi_\delta$ and $\na^2\phi$
with a difference involving $\Delta \phi_{\delta}$ and $\Delta \phi$. For this, let
$\chi\in C_0^\infty(\Omega)$ be a test function. In the following, we use only
formal manipulations which can, however, be justified by approximation by smooth
functions. It follows that
\begin{align*}
&\int_{\Omega} \big(\overline{\mathbb{D}(\vel_\delta):\nabla^2\phi_{\delta}}-\mathbb{D}(
\vel):\nabla^2 \phi\big)\chi\, dx= \int_{\Omega} \big(\overline{\na \vel_\delta:\nabla^2\phi_{\delta}}-\na
\vel:\nabla^2 \phi\big)\chi\, dx \\
&=\int_{\Omega} \big(\overline{\na (\vel_\delta  \chi):\nabla^2\phi_{\delta}}-\na(
\vel \chi):\nabla^2 \phi\big)\, dx -\int_{\Omega} \big(\overline{(\vel_\delta\otimes \nabla \chi):\nabla^2\phi_{\delta}}-(
\vel\otimes \nabla \chi):\nabla^2 \phi\big)\, dx.
\end{align*}
The second term vanishes because of the strong convergence of $\vel_{\delta}$.
Thus, integrating by parts twice,
\begin{align*}
\int_{\Omega} &\big(\overline{\na (\vel_\delta  \chi):\nabla^2\phi_{\delta}}-\na(
\vel \chi):\nabla^2 \phi\big)\, dx = \int_{\Omega} \big(\overline{\diver(\vel_\delta  \chi)\Delta\phi_{\delta}}-\diver(
\vel \chi)\Delta \phi\big)\, dx\\
&= \int_{\Omega} \big(\overline{\diver\vel_\delta  \Delta\phi_{\delta}}-\diver
\vel \Delta \phi\big) \chi\, dx+ \int_{\Omega} \big(\overline{(\vel_\delta  \cdot \na \chi)\Delta\phi_{\delta}}-(
\vel \cdot \na \chi)\Delta \phi\big)\, dx,
\end{align*}
and the second term again vanishes because of the strong convergence of $\vel_{\delta}$.
As the test function $\chi$ is arbitrary, we infer that
$$
\overline{\mathbb{D}(\vel_\delta):\nabla^2\phi_{\delta}}-\mathbb{D}(
\vel):\nabla^2 \phi=\overline{\diver\vel_\delta  \Delta\phi_{\delta}}-\diver
\vel \Delta \phi \quad\mbox{a.e. in }\Omega.
$$
Thus, inserting this expression into \eqref{start3} and using the properties of $\phi_{\delta}$ and $\phi$, i.e., \eqref{sol1} and \eqref{sol2}, respectively, we deduce \eqref{evf} a.e. in $\Omega$.
\end{proof}

\subsection{Estimates based on the convexity of the free energy}

In this part, we show that the left-hand side of the effective viscous flux identity \eqref{evf} gives us the important description of the possible oscillations of the total density $\rhotot_{\delta}$, provided we assume the convexity of the free energy with respect to the partial densities. This is summarized in the following lemma.

\begin{lemma}\label{lem.Wk}
Let the free energy satisfy Hypothesis (H6) and the sequence $(\vec{\rhotot}_{\delta}, \theta_{\delta}, \vec{\mu}_{\delta}, p_{\delta})$
with $\theta_\delta>0$ a.e.\ in $\Omega$ fulfil
\begin{align}
  \vec{\rhotot}_{\delta}\rightharpoonup \vec{\rhotot}
  &\quad\textrm{weakly in } L^1(\Omega; \mathbb{R}^N),\label{Kon1}\\
  p_{\delta}\rightharpoonup p &\quad\textrm{weakly in } L^1(\Omega),\label{Kon7}\\
  \theta_{\delta}\to \theta &\quad\textrm{strongly in } L^1(\Omega),\label{Kon2}\\
  \ln \theta_{\delta}\to \ln \theta &\quad\textrm{strongly in } L^1(\Omega),
	\nonumber \\ 
  \Pi(\vec{\mu}_{\delta}/\theta_{\delta}) \to
	\overline{\Pi(\vec{\mu}_{\delta}/\theta_{\delta})}
	&\quad\textrm{strongly in } L^1(\Omega; \mathbb{R}^N),\label{Kon4}\\
  \vec{\mu}_{\delta}=\na_{\vec{\rhotot}} h_{\theta_{\delta}}(\vec{\rhotot}_{\delta})
	&\quad\textrm{a.e. in }\Omega, \nonumber \\ 
  p_{\delta}=p(\vec{\rhotot}_{\delta},\theta_{\delta})
	&\quad\textrm{a.e. in }\Omega, \nonumber 
\end{align}
where $p(\vec{\rhotot}_{\delta},\theta_{\delta})
:= -h_{\theta_{\delta}}(\vec{\rhotot}_{\delta})+\vec{\mu}_{\delta} \cdot
\vec{\rhotot}_{\delta}=h^*_{\theta_{\delta}}(\vec{\mu}_{\delta})$ and $\pa_i h_\theta^*(\vec{\mu}_\delta) = \rho_{\delta,i}$, where $\pa_i=\pa/\pa\mu_i$.
For $k\in \mathbb{N}$, define (for a proper subsequence)
\begin{equation}\label{Wk}
  W_k := \overline{p_\delta T_k(\rhotot_\delta)}
	- p \overline{T_k(\rho_\delta)},
\end{equation}
where $\rhotot_{\delta}:=\sum_{i=1}^N \rhotot_{\delta,i}$ and
$\rhotot:=\sum_{i=1}^N \rhotot_i$.
Then for all $k\in \N$,
\begin{align}
  0\le W_k \le W_{k+1} &\quad\textrm{a.e. in }\Omega,\label{Wk1}\\
  0\le \theta \big(\overline{\rhotot_{\delta}T_k(\rhotot_{\delta})}
	- \rhotot \overline{T_k(\rhotot_{\delta})}\big)\le K_2 W_k
	&\quad\textrm{a.e. in }\Omega,\label{Wk2}
\end{align}
where $K_2>0$ is given in \eqref{H6C}.
\end{lemma}

\begin{proof}
{\em Step 1: Introduction of a proper set.} We start the proof by introducing the proper subsets of $\Omega$. Indeed, since we know that all weak limits exist, we need to identify them on sufficiently large subsets of $\Omega$. Hence, using \eqref{Kon2}--\eqref{Kon4} and the Egorov theorem, we know that for arbitrary $\eta>0$, there exists a measurable set $\Omega_{\eta}$ such that $|\Omega \setminus \Omega_{\eta}|\le \eta$ and
\begin{align*}
  \theta_{\delta}\to \theta &\quad\textrm{strongly in }
	L^{\infty}(\Omega_{\eta}), \\ 
  \ln \theta_{\delta}\to \ln \theta &\quad\textrm{strongly in }
	L^{\infty}(\Omega_{\eta}), \\ 
  \Pi(\vec{\mu}_{\delta}/\theta_{\delta}) \to
	\overline{\Pi(\vec{\mu}_{\delta}/\theta_{\delta})}
	&\quad\textrm{strongly in } L^{\infty}(\Omega_{\eta}; \mathbb{R}^N), 
\end{align*}	
and consequently also
\begin{equation}
  \Pi \vec{\mu}_{\delta} \to \overline{\Pi\vec{\mu}_{\delta}}
	\quad\textrm{strongly in } L^{\infty}(\Omega_{\eta}; \mathbb{R}^N).\label{Kon5E}
\end{equation}
Furthermore, introducing the linear mapping $\Pro$: $\R^N\to\R^N$ by
$\Pro(\vec{u})_i:=u_i-u_N$ for $i=1,\ldots,N$, we know that
$|\Pro(\vec{u})|\le C|\Pi \vec{u}|$, since
$$ |\Pro(\vec{u})|^2 = \sum_{i=1}^N\left|
\left(u_i - \frac{1}{N}\sum_{j=1}^N u_j\right) - 
\left(u_N - \frac{1}{N}\sum_{j=1}^N u_j\right)\right|^2\leq
2(N+1)|\Pi(\vec{u})|^2$$
for every $u\in\R^N$, and due to the linearity of $\Pro$, we deduce from \eqref{Kon5E} that
\begin{equation*}
  \Pro \vec{\mu}_{\delta} \to \overline{\Pro \vec{\mu}_{\delta}}
	\quad\textrm{strongly in } L^{\infty}(\Omega_{\eta}; \mathbb{R}^N). 
\end{equation*}
Therefore, we just need to show that \eqref{Wk1}--\eqref{Wk2} holds true
a.e.\ in $\Omega_{\eta}$, and since $|\Omega \setminus \Omega_{\eta}|\le \eta$,
the equalities will necessarily hold also a.e.\ in $\Omega$.

\medskip\noindent{\em Step 2: Using the conjugate $h_{\theta}^*$.}
We characterize $W_k$ as
\begin{equation*} 
W_k=\frac 12\overline{(p(\vec{\rhotot}_\eps,\theta_\eps)-p(\vec{\rhotot}_\delta,\theta_\delta))(T_k(\rhotot_\eps)-T_k(\rhotot_\delta))}\quad \textrm{a.e. in }\Omega.
\end{equation*}
Thus, we need to compute the differences
$p(\vec{\rhotot}_\eps,\theta_\eps)-p(\vec{\rhotot}_\delta,\theta_\delta)$
and $T_k(\rhotot_\eps)-T_k(\rhotot_\delta)$ for $\eps>0$, $\delta>0$.
As the pressure is equal to the Legendre transform of $h_\theta$, we can write
\begin{equation}\label{pred}
\begin{aligned}
  p(\vec{\rhotot}_\eps,\theta_\eps)-p(\vec{\rhotot}_\delta,\theta_\delta)
	&= h_{\theta_\eps}^*(\vec\mu_\eps) - h_{\theta_\delta}^*(\vec\mu_\delta)\\
&=(h_{\theta}^*(\vec\mu_\eps) - h_{\theta}^*(\vec\mu_\delta))+(h_{\theta_\eps}^*(\vec\mu_\eps) - h_{\theta}^*(\vec\mu_\eps))+(h_{\theta}^*(\vec\mu_\delta) - h_{\theta_\delta}^*(\vec\mu_\delta))\\
&=: Y^1_{\eps,\delta} + Y^2_{\eps} + Y^3_{\delta}.
\end{aligned}
\end{equation}
The first term $ Y^1_{\eps,\delta}$ is formulated as
$$
 Y^1_{\eps,\delta}= \int_0^1\frac{d}{dt}h_\theta^*(\vec z_t)\, dt
	= \int_0^1\sum_{i=1}^N\pa_i h_\theta^*(\vec z_t)(\mu_{\eps,i}-\mu_{\delta,i})\, dt,
$$
where $\vec z_t:= t\vec\mu_\eps+(1-t)\vec\mu_\delta$. Using the decomposition
\begin{equation}\label{3.P}
  \mu_{\eps,i}-\mu_{\delta,i} = \Pro(\vec\mu_\eps-\vec\mu_\delta)_i
	+ (\mu_{\eps,N}-\mu_{\delta,N}), \quad i=1,\ldots,N,
\end{equation}
it follows that
\begin{equation}\label{3.e1}
   Y^1_{\eps,\delta}
	= \int_0^1\bigg(\sum_{i=1}^{N-1}\pa_i h_\theta^*(\vec z_t)
	\Pro(\vec\mu_\eps-\vec\mu_\delta)_i + \sum_{i=1}^N\pa_i h_\theta^*(\vec z_t)
	(\mu_{\eps,N}-\mu_{\delta,N})\bigg)\,dt.
\end{equation}
In a very similar way, using the fact that $\vec{\rhotot}_{\eps}=\nabla h^*_{\theta_{\eps}}(\vec{\mu}_{\eps})$, we find that
\begin{align}
  T_k (\rhotot_\eps)&-T_k(\rhotot_\delta)
	= T_k\bigg(\sum_{i=1}^N\pa_i h_{\theta_\eps}^*(\vec\mu_\eps)\bigg)
	- T_k\bigg(\sum_{i=1}^N\pa_i h_{\theta_{\delta}}^*(\vec\mu_\delta)\bigg) \nonumber \\
  &= \bigg[T_k\bigg(\sum_{i=1}^N\pa_i h_{\theta}^*(\vec\mu_\eps)\bigg)
	- T_k\bigg(\sum_{i=1}^N\pa_i h_{\theta}^*(\vec\mu_\delta)\bigg)\bigg] \nonumber \\
  &\phantom{xx}{}
	+ \bigg[T_k\bigg(\sum_{i=1}^N\pa_i h_{\theta_\eps}^*(\vec\mu_\eps)\bigg)
	-T_k\bigg(\sum_{i=1}^N\pa_i h_{\theta}^*(\vec\mu_\eps)\bigg)\bigg] \nonumber \\
  &\phantom{xx}{} + \bigg[T_k\bigg(\sum_{i=1}^N\pa_i h_{\theta}^*(\vec\mu_\delta)\bigg)
	- T_k\bigg(\sum_{i=1}^N\pa_i h_{\theta_{\delta}}^*(\vec\mu_\delta)\bigg)\bigg]
  =: Z^1_{\eps,\delta} + Z^2_{\eps}+Z^3_{\delta}, \label{pred2}
\end{align}
and rewrite $Z^1_{\eps,\delta}$ as
\begin{align}
  Z^1_{\eps,\delta}&= \int_0^1 \frac{d}{ds}T_k\bigg(\sum_{i=1}^N\pa_i h_\theta^*(\vec z_s)\bigg)\, ds
	= \int_0^1\Lambda_k(s)\sum_{i,j=1}^N\pa_{ij}^2 h_\theta^*(\vec z_s)
	(\mu_{\eps,j}-\mu_{\delta,j})\, ds, \label{3.e2}
\end{align}
where we defined for arbitrary $0\le s\le 1$ and $k\in\N$,
$$
  \Lambda_k(s) := T_k'\bigg(\sum_{i=1}^N\pa_i h_\theta^*(\vec z_s)\bigg)
	= T_1'\bigg(\frac{1}{k}\sum_{i=1}^N\pa_i h_\theta^*(\vec z_s)\bigg).
$$
The sum over $i,j=1,\ldots,N$ in \eqref{3.e2}
can be rewritten, using \eqref{3.P}, as
$$
  \sum_{i,j=1}^N \pa_{ij}^2 h_\theta^*(\vec z_s)
	(\mu_{\eps,i}-\mu_{\delta,i})
	= \sum_{i,j=1}^{N}\pa_{ij}^2 h_\theta^*(\vec z_s)
	\big( \Pro(\vec\mu_\eps-\vec\mu_\delta)_i
	+ (\mu_{\eps,N}-\mu_{\delta,N})\big)
$$
such that
\begin{align}
  Z^1_{\eps,\delta}&= \int_0^1 \Lambda_k(s)\left(\sum_{i,j=1}^{N}
	\pa_{ij}^2 h_\theta^*(\vec z_s)
	\Pro(\vec\mu_\eps-\vec\mu_\delta)_i
	+ (\mu_{\eps,N}-\mu_{\delta,N})\sum_{i,j=1}^N\pa_{ij}^2 h_\theta^*(\vec z_s)
	\right)\, ds. \label{3.e22}
\end{align}

\medskip\noindent {\em Step 3: Proof of \eqref{Wk1}.}
We restrict our analysis to $\Omega_{\eta}$, since for every fixed $k\in \N$,
\begin{equation}\label{smalle}
\sup_{\eps, \delta>0}\int_{\Omega \setminus \Omega_{\eta}}\big|(p(\vec{\rhotot}_\eps,\theta_\eps)-p(\vec{\rhotot}_\delta,\theta_\delta))(T_k(\rhotot_\eps)-T_k(\rhotot_\delta))\big|\, dx \to 0 \quad\mbox{as }\eta\to 0,
\end{equation}
which follows from assumption \eqref{Kon7}. Furthermore, we define the sets
$$
\Omega_{\eta,R}^{\eps,\delta}:= \{x\in \Omega_{\eta}\colon \; \rhotot_{\eps}(x)+\rhotot_{\delta}(x)\le R\}.
$$
Similarly as before, we again have for every fixed $k\in \N$,
\begin{equation}\label{smalle2}
\sup_{\eps, \delta>0}\int_{\Omega_{\eta} \setminus \Omega_{\eta,R}^{\eps,\delta}}\big|(p(\vec{\rhotot}_\eps,\theta_\eps)-p(\vec{\rhotot}_\delta,\theta_\delta))(T_k(\rhotot_\eps)-T_k(\rhotot_\delta))\big|\, dx\to 0 \quad\mbox{as }R\to\infty.
\end{equation}
Therefore, we focus on the behavior of the sequence on the sets $\Omega_{\eta,R}^{\eps,\delta}$.
We use Hypothesis~(H6) to show that some terms $Y_{\eps,\delta}^j$
and $Z_{\eps,\delta}^j$ vanish in the limit. Since (by construction) $\rhotot_{\eps}(x)+\rhotot_{\delta}(x) \leq R$ for $x\in\Omega_{\eta,R}^{\eps,\delta}$, it follows
on this set that
\begin{equation}\label{sss1}
\sum_{i=1}^N (\pa_i h_{\theta_{\eps}}^* (\vec{\mu}_{\eps})+\pa_i h_{\theta_{\delta}}^* (\vec{\mu}_{\delta}))\le R.
\end{equation}
Consequently, since we know that $\theta_{\delta}$, $\theta_{\eps}$, and $\theta$
are bounded from below and above in $\Omega_{\eta}$, Hypothesis (H6) (and its consequence \eqref{H6A}) implies that
$$
\sum_{i=1}^N (\pa_i h_{\theta}^* (\vec{\mu}_{\eps})+\pa_i h_{\theta}^* (\vec{\mu}_{\delta}))\le C(R,\omega,\eta) \quad\mbox{a.e.\ in }\Omega_{\eta,R}^{\eps,\delta},
$$
where $\omega$ refers to the modulus of continuity of $h_\theta^*$,
introduced in \eqref{H6A}.
Finally, thanks to the continuity, for any $t\in (0,1)$ and
$\vec z_t= t\vec\mu_\eps+(1-t)\vec\mu_\delta$,
\begin{equation}\label{ztb}
\sum_{i=1}^N \pa_i h_{\theta}^* (\vec{z}_{t})\le C(R,\omega,\eta, h^*_\theta)
\quad\mbox{a.e.\ in }\Omega_{\eta,R}^{\eps,\delta},
\end{equation}
and Hypothesis (H6) (used to show \eqref{H6C}) implies  that
\begin{equation}\label{ztb2}
| \pa_{ij} h_{\theta}^* (\vec{z}_{t})|\le C(R,\omega,\eta, h^*_\theta)
\quad\mbox{a.e.\ in }\Omega_{\eta,R}^{\eps,\delta}.
\end{equation}

With these auxiliary results, we can now focus on the limiting process. Let the function $\chi\in L^{\infty}(\Omega)$ be arbitrary and nonnegative. It follows from definition \eqref{pred}, the uniform convergence of $(\theta_\eps)$ in $\Omega_{\eta}$, estimate \eqref{sss1}, and Hypothesis (H6) that
$$
\limsup_{\eps \to 0} \limsup_{\delta\to 0} \int_{\Omega_{\eta,R}^{\eps,\delta}}
\big(|Y^2_{\eps}| + |Y^3_{\delta}|\big)\, dx =0.
$$
Similarly, again thanks to Hypothesis (H6) and the proper definition of the set $\Omega_{\eta,R}^{\eps,\delta}$,
$$
\limsup_{\eps \to 0} \limsup_{\delta\to 0} \|Z^2_{\eps} + Z^3_{\delta}\|_{L^{\infty}(\Omega_{\eta,R}^{\eps,\delta})}=0.
$$
Hence, using the weak convergence \eqref{Kon7} and definitions \eqref{pred} and \eqref{pred2},
\begin{align*}
  \lim_{\eps \to 0}&\lim_{\delta \to 0} \int_{\Omega_{\eta,R}^{\eps,\delta}}
	\big(p(\vec\rho_\eps,\theta_\eps)-p(\vec\rho_\delta,\theta_\delta)\big)
	\big(T_k(\rhotot_\eps)-T_k(\rhotot_\delta)\big)\chi\, dx\\
&=\lim_{\eps \to 0}\lim_{\delta \to 0} \int_{\Omega_{\eta,R}^{\eps,\delta}} Y^1_{\eps,\delta}Z^1_{\eps,\delta}\chi\, dx
=\lim_{\eps \to 0}\lim_{\delta \to 0} \int_{\Omega_{\eta,R}^{\eps,\delta}}
	(I^1_{\eps,\delta} + I_{\eps,\delta}^2 + I_{\eps,\delta}^3)\chi\, dx,
\end{align*}
where we identify the terms $I^j_{\eps,\delta}$ for $j=1,2,3$ with the help of
\eqref{3.e1} and \eqref{3.e22} as
\begin{align}
	\label{I1}
	 I_{\eps,\delta}^1 &= (\mu_{\eps,N}-\mu_{\delta,N})^2
	\int_0^1\int_0^1\sum_{i,j,\ell=1}^N\pa_i h_\theta^*(\vec z_t)
	\pa_{j\ell}^2 h_\theta^*(\vec z_s)\Lambda_k(s)\,ds \, dt, \\
	\nonumber 
   I_{\eps,\delta}^2 &= (\mu_{\eps,N}-\mu_{\delta,N})
	\int_0^1\int_0^1\sum_{i,j,\ell=1}^N\big(\pa_i h_\theta^*(\vec z_t)
	\pa_{j\ell}^2 h_\theta^*(\vec z_s) + \pa_j h_\theta^*(\vec z_t)
	\pa_{i\ell}^2 h_\theta^*(\vec z_s)\big) \\ \nonumber
	\nonumber 
	&\quad \times \Pro(\vec\mu_\eps-\vec\mu_\delta)_i
	\Lambda_k(s)\, ds \, dt, \\
	\nonumber
	 I_{\eps,\delta}^3 &=\int_0^1\int_0^1\sum_{i,j,\ell=1}^{N}
	\pa_i h_\theta^*(\vec z_t)\pa_{j\ell}^2 h_\theta^*(\vec z_s)
	\Pro(\vec\mu_\eps-\vec\mu_\delta)_i\Pro(\vec\mu_\eps-\vec\mu_\delta)_j
	\Lambda_k(s)\, ds \, dt.
\end{align}

We start with the easiest term, which is $I_{\eps,\delta}^3$. We deduce from
\eqref{ztb} and \eqref{ztb2} and the uniform convergence of $\Pi \vec{\mu}_{\eps}$ in $\Omega_{\eta}$ (and consequently also of $\Pro (\vec{\mu}_{\eps})$) that
\begin{equation}\label{Ic3}
\lim_{\eps \to 0}\lim_{\delta \to 0}  \int_{\Omega_{\eta,R}^{\eps,\delta}}|I_{\eps,\delta}^3|\chi\, dx =0.
\end{equation}
Next, since $h_\theta^*$ is convex, the Hessian of $h_\theta^*$
is positive semidefinite, i.e.\ $\sum_{j,\ell=1}^N\pa_{j\ell}^2 h_\theta^*(\vec z_s)\ge 0$. Furthermore, we know that the function $\pa_i h_\theta^*(\vec{z}_t)$
is nonnegative (recall that $\pa_i h_\theta^*(\vec{z}_t) = \rho_i$)
and then it obviously follows that $I_{\eps,\delta}^1\ge 0$.

It remains to analyze the integral $I_{\eps,\delta}^2$. For this, we use the
Cauchy--Schwarz and Young inequalities for some $\kappa>0$, and
the positive semi-definiteness of $(\pa_{j\ell}^2 h_\theta^*)$:
\begin{align*}
  |I_{\eps,\delta}^2| &\le |\mu_{\eps,N}-\mu_{\delta,N}|
  |\Pro(\vec\mu_\eps-\vec\mu_\delta)|
	\int_0^1\int_0^1\sum_{i,j,\ell=1}^N\pa_i h_\theta^*(\vec z_t)
	\pa_{j\ell}^2 h_\theta^*(\vec z_s)	\Lambda_k(s)\, ds \, dt\\
  &\phantom{xx}{}+\int_0^1\bigg|\sum_{i,\ell=1}^N
	\pa_{i\ell}^2 h_\theta^*(\vec z_s)
	\bigg(|\mu_{\eps,N}-\mu_{\delta,N}|^2\Lambda_k(s)\int_0^1\sum_{j=1}^N\pa_j
	h_\theta^*(\vec z_t)\, dt\bigg)^{1/2} \\
  &\phantom{xx}{}\times \Pro(\vec\mu_\eps-\vec\mu_\delta)_i
	\bigg(\Lambda_k(s)\int_0^1\sum_{j=1}^N\pa_j h_\theta^*(\vec z_t)\, dt\bigg)^{1/2}
	\bigg|\, ds \\
  &\le 2\big(I_{\eps,\delta}^1\big)^{1/2}\bigg(|\Pro(\vec\mu_\eps-\vec\mu_\delta)|^2
	\int_0^1\int_0^1\sum_{i,j,\ell=1}^N\pa_i h_\theta^*(\vec z_t)
	\pa_{j\ell}^2 h_\theta^*(\vec z_s)	\Lambda_k(s)\, ds \, dt\bigg)^{1/2}\\
  &\le C(R,\eta)\big(I_{\eps,\delta}^1\big)^{1/2}|\Pro(\vec\mu_\eps-\vec\mu_\delta)|
	\le \kappa I_{\eps,\delta}^1+ C(R,\eta,\kappa)|\Pro(\vec\mu_\eps-\vec\mu_\delta)|^2,
\end{align*}
where in the last line we used \eqref{ztb}--\eqref{ztb2}. As $\kappa>0$ can be taken arbitrarily small, we use the uniform convergence of $(\Pro (\vec\mu_\eps))$ and \eqref{Ic3} to deduce that for any $\chi \in L^\infty(\Omega)$,
\begin{equation}\label{Comp23}
  \lim_{\eps \to 0}\lim_{\delta \to 0} \int_{\Omega_{\eta,R}^{\eps,\delta}}
  (p(\vec\rho_\eps,\theta_\eps)-p(\vec\rho_\delta,\theta_\delta))
	\big(T_k(\rhotot_\eps)-T_k(\rhotot_\delta)\big)\chi\, dx
  = \lim_{\eps \to 0}\lim_{\delta \to 0}
	\int_{\Omega_{\eta,R}^{\eps,\delta}} I^1_{\eps,\delta}\chi\, dx\ge 0.
\end{equation}
This inequality, together with \eqref{smalle} and \eqref{smalle2}, shows that $W_k\ge 0$. Moreover, since for all $k$ we know that $\Lambda_k(s)\le \Lambda_{k+1}(s)$ (it follows from $T'_k\le T_{k+1}'$), we see from the definition of $I^1_{\eps,\delta}$ that $W_k \le W_{k+1}$.

\medskip\noindent {\em Step 4: Proof of \eqref{Wk2}.} Here, we can repeat the
arguments from Step 3 almost step by step. Indeed, we have
$$
2\theta\big(\overline{\rhotot_{\delta}T_k(\rhotot_{\delta})} - \rhotot \overline{T_k(\rhotot_{\delta})}\big)=\theta\overline{(\rhotot_{\delta}-\rhotot_{\eps})(T_k(\rhotot_{\delta})-T_k(\rhotot_{\eps}))}.
$$
Again, we just need to identify the inequality on the set $\Omega_{\eta,R}^{\eps,\delta}$, since the remaining parts vanish due to \eqref{Kon1}.
Proceeding in the same way as in \eqref{pred2}, we write
\begin{equation*}
\begin{aligned}
  \rhotot_\eps-\rhotot_\delta
  &= \sum_{i=1}^N\big(\pa_i h_{\theta}^*(\vec\mu_\eps)
	- \pa_i h_{\theta}^*(\vec\mu_\delta)\big)
	+ \sum_{i=1}^N\big(\pa_i h_{\theta_\eps}^*(\vec\mu_\eps)
	-\pa_i h_{\theta}^*(\vec\mu_\eps)\big)\\
  &\phantom{xx}{} + \sum_{i=1}^N\left(\pa_i h_{\theta}^*(\vec\mu_\delta)
	- \pa_i h_{\theta_{\delta}}^*(\vec\mu_\delta)\right)
	=:\widehat{Z}^1_{\eps,\delta} + \widehat{Z}^2_{\eps}+\widehat{Z}^3_{\delta}
\end{aligned}
\end{equation*}
and rewrite $\widehat{Z}^1_{\eps,\delta}$ as (compare with \eqref{3.e22})
\begin{equation*}
  \widehat{Z}^1_{\eps,\delta}= \int_0^1 \bigg(\sum_{i,j=1}^{N}
	\pa_{ij}^2 h_\theta^*(\vec z_s)\Pro(\vec\mu_\eps-\vec\mu_\delta)_i
	+ (\mu_{\eps,N}-\mu_{\delta,N})\sum_{i,j=1}^N\pa_{ij}^2 h_\theta^*(\vec z_s)
	\bigg)\, ds. 
\end{equation*}
Hence, repeating the procedure from the previous step, we deduce that for any bounded nonnegative function $\chi$,
\begin{equation}\label{comp24}
\lim_{\eps \to 0}\lim_{\delta \to 0} \int_{\Omega_{\eta,R}^{\eps,\delta}} \theta (\rho_\eps-\rho_\delta)
	\big(T_k(\rhotot_\eps)-T_k(\rhotot_\delta)\big)\chi\, dx =\lim_{\eps \to 0}\lim_{\delta \to 0} \int_{\Omega_{\eta,R}^{\eps,\delta}} \widehat{I}^1_{\eps,\delta}\chi\, dx,
\end{equation}
where
\begin{equation}
	\label{I1A}
	 \widehat{I}_{\eps,\delta}^1 = \theta(\mu_{\eps,N}-\mu_{\delta,N})^2
	\int_0^1\int_0^1\sum_{i,j,\ell,m=1}^N\pa_{\ell m} h_\theta^*(\vec z_t)
	\pa_{ij}^2 h_\theta^*(\vec z_s)\Lambda_k(s)\,ds \, dt.
\end{equation}
Hypothesis (H6), namely its consequence \eqref{H6C}, definitions \eqref{I1} and \eqref{I1A} show that $\widehat{I}_{\eps,\delta}^1\le K_2 I_{\eps,\delta}^1$ and comparing \eqref{Comp23} and \eqref{comp24} then leads to \eqref{Wk2}.
\end{proof}

The next lemma combines the results coming from the convexity of $h_\theta$ and the effective viscous flux identity, which finally lead to a uniform bound on $(W_k)$,
defined in \eqref{Wk}.

\begin{lemma}\label{lem.wk.bound}
Let all assumptions of Lemma~\ref{lem.Wk} and \eqref{cA:1} be satisfied. Then there exists a constant $C>0$ such that for all $k\in \mathbb{N}$,
\begin{equation}\label{bound.Wk}
\int_{\Omega}\frac{W_k}{\lambda_1(\theta)+\lambda_2(\theta)}\, dx
\le C \sup_{\delta>0}\int_{\Omega}\frac{\lambda_1(\theta_{\delta})+\lambda_2(\theta_\delta)}{\theta_\delta} (\diver \vel_\delta)^2\, dx.
\end{equation}
\end{lemma}

Consequently, if the right-hand side in \eqref{bound.Wk} is finite,
we infer from the monotonicity of $(W_k)$ (see \eqref{Wk1}) and monotone
convergence that the sequence
$(W_k/(\lambda_1(\theta) + \lambda_2(\theta))_{k\in\N}$ is strongly converging
in $L^1(\Omega)$ to a nonnegative integrable function.

\begin{proof}
We start the proof with a simple inequality, using \eqref{Wk2}, \eqref{Wk},
and \eqref{evf}:
\begin{align*}
  &\overline{\frac{\theta_{\delta}}{\lambda_1(\theta_{\delta})
	+\lambda_2(\theta_\delta)}(T_k(\rhotot_\eps)-T_k(\rhotot_\delta))^2}
  = \frac{\theta}{\lambda_1(\theta)+\lambda_2(\theta)}
	\overline{(T_k(\rhotot_\eps)-T_k(\rhotot_\delta))^2} \\
	&\quad\le \frac{\theta}{\lambda_1(\theta)+\lambda_2(\theta)}
	\overline{(\rhotot_\eps-\rhotot_\delta)(T_k(\rhotot_\eps)-T_k(\rhotot_\delta))}
  \le \frac{CW_k}{\lambda_1(\theta)+\lambda_2(\theta)} \\
	&\quad = C\frac{\lambda_2(\theta)+\frac43\lambda_1(\theta)}{
	\lambda_1(\theta)+\lambda_2(\theta)}\big(\overline{T_k(\rhotot_\delta)
	\diver \vel_\delta} - \overline{T_k(\rhotot_\delta)}\diver\vel\big).
\end{align*}
Using \eqref{evf} and \eqref{Wk2} again, we observe that
$\overline{T_k(\rhotot_\delta)\diver \vel_\delta}
- \overline{T_k(\rhotot_\delta)}\diver\vel\ge 0$, leading to
\begin{align*}
  \overline{\frac{\theta_{\delta}}{\lambda_1(\theta_{\delta})
	+\lambda_2(\theta_\delta)}(T_k(\rhotot_\eps)-T_k(\rhotot_\delta))^2}&\le C
	\big(\overline{T_k(\rhotot_\delta)\diver \vel_\delta}
	- \overline{T_k(\rhotot_\delta)}\diver\vel\big) \nonumber \\
  &= C\overline{\diver \vel_\delta (T_k(\rhotot_\delta) - T_k(\rhotot_\eps))}.
\end{align*}
It follows from the Cauchy--Schwarz inequality that
\begin{align*}
  \int_\Omega&\overline{\frac{\theta_{\delta}}{\lambda_1(\theta_{\delta})
	+\lambda_2(\theta_\delta)}(T_k(\rhotot_\eps)-T_k(\rhotot_\delta))^2}\,dx
	\le C\int_\Omega\overline{\diver \vel_\delta (T_k(\rhotot_\delta)
	- T_k(\rhotot_\eps))}\,dx \\
	&\le C\bigg(\int_{\Omega} \frac{\lambda_1(\theta_{\delta})
	+\lambda_2(\theta_\delta)}{\theta_{\delta}}(\diver \vel_\delta)^2
	\, dx\bigg)^{1/2} \bigg(\int_{\Omega} \frac{\theta_{\delta}(T_k(\rhotot_\eps)
	-T_k(\rhotot_\delta))^2}{\lambda_1(\theta_{\delta})+\lambda_2(\theta_\delta)}\, dx
	\bigg)^{1/2},
\end{align*}
which yields
\begin{align*}
\int_{\Omega}\overline{\frac{\theta_{\delta}(T_k(\rhotot_\eps)-T_k(\rhotot_\delta))^2
}{\lambda_1(\theta_{\delta})+\lambda_2(\theta_\delta)}}\, dx
\le C\sup_{\delta>0}\int_{\Omega} \frac{\lambda_1(\theta_{\delta})+\lambda_2(\theta_\delta)}{\theta_{\delta}}(\diver \vel_\delta)^2\, dx, \\
\int_{\Omega} \overline{\diver \vel_\delta (T_k(\rhotot_\delta)- T_k(\rhotot_\eps))}\, dx \le C\sup_{\delta>0}\int_{\Omega} \frac{\lambda_1(\theta_{\delta})+\lambda_2(\theta_\delta)}{\theta_{\delta}}(\diver \vel_\delta)^2\, dx.
\end{align*}
We infer that
\begin{align*}
  \int_\Omega\frac{W_k}{\lambda_1(\theta)+\lambda_2(\theta)}\,dx
	&= \int_\Omega\frac{\lambda_2(\theta)+\frac43\lambda_1(\theta)}{\lambda_1(\theta)
	+\lambda_2(\theta)}\big(\overline{T_k(\rho_\delta)\diver\vel_\delta}
	-\overline{T_k(\rho_\delta)}\diver\vel\big)\,dx \\
	&= \int_\Omega\frac{\lambda_2(\theta)+\frac43\lambda_1(\theta)}{\lambda_1(\theta)
	+\lambda_2(\theta)}\overline{\diver\vel_\delta(T_k(\rho_\delta)-T_k(\rho_\eps))}
	\,dx \\
  &\le C\sup_{\delta>0}\int_{\Omega} \frac{\lambda_1(\theta_{\delta})
	+\lambda_2(\theta_\delta)}{\theta_{\delta}}(\diver \vel_\delta)^2\, dx,
\end{align*}
finishing the proof.
\end{proof}

The following lemma is the key step in the proof. It shows that we
can use the truncation function $T_k(\rhotot)$ instead of $\rhotot$ in all estimates,
and this change creates only a small error, which can be neglected in a suitable
topology.

\begin{lemma}\label{lem.Qk}
Let all assumptions of Lemmata~\ref{lem.Wk} and \ref{lem.wk.bound} be satisfied and let
\begin{equation*}
  \sup_{\delta>0}\int_{\Omega} \frac{\lambda_1(\theta_{\delta})+\lambda_2(\theta_\delta)
  }{\theta_{\delta}}(\diver \vel_\delta)^2\, dx<\infty.
\end{equation*}
The quantities
\begin{align}\label{def.Qk}
  Q_k &:= \overline{p_{\eps}(T_k(\rhotot_\eps)-\rhotot_\eps T_k'(\rhotot_\eps))}-p\overline{(T_k(\rhotot_\eps)-\rhotot_\eps T_k'(\rhotot_\eps))},\\
  O_k &:= T_k(\rhotot)-\overline{T_k(\rhotot_{\eps})}\label{def.Ok}
\end{align}
are nonnegative and satisfy for all $k\in\N$,
\begin{equation}
\theta O_k^2 \le CW_k \quad \textrm{and} \quad \lim_{k\to\infty} \int_{\Omega}  \frac{Q_k+\theta O_k^2}{\lambda_1(\theta)+\lambda_2(\theta)}\, dx=0.\label{sOO}
\end{equation}
\end{lemma}

\begin{proof}
For the identification of $Q_k$, we can repeat the proof of Lemma~\ref{lem.Wk} step by
step, where instead of \eqref{pred2}--\eqref{3.e2}, we use a similar computation for
$(T_k(\rhotot_\eps)-\rhotot_\eps T_k'(\rhotot_\eps))-(T_k(\rhotot_\delta)-\rhotot_\delta T_k'(\rhotot_\delta))$. Heuristically, this means that we replace
$\Lambda_k$, which corresponds to the derivative of $T_k(s)$, by
$\widetilde{\Lambda}_k$, which corresponds to the derivative of $T_k(s)-sT_k'(s)$.
This leads to the identification of the limit (compare with \eqref{Comp23}):
\begin{equation*}
\begin{aligned}
&\lim_{\eps \to 0}\lim_{\delta \to 0} \int_{\Omega_{\eta,R}^{\eps,\delta}}  \big(p(\vec\rho_\eps,\theta_\eps)-p(\vec\rho_\delta,\theta_\delta)\big)
	\big(T_k(\rhotot_\eps)-\rhotot_\eps T_k'(\rhotot_\eps)-T_k(\rhotot_\delta)+\rhotot_\delta T_k'(\rhotot_\delta)\big)\chi\, dx\\
&\quad =\lim_{\eps \to 0}\lim_{\delta \to 0} \int_{\Omega_{\eta,R}^{\eps,\delta}} \widetilde{I}^1_{\eps,\delta}\chi\, dx,
\end{aligned}
\end{equation*}
where the term
$$
  \widetilde{I}_{\eps,\delta}^1 := \theta(\mu_{\eps,N}-\mu_{\delta,N})^2
	\int_0^1\int_0^1\sum_{i=1}^N\pa_i h_\theta^*(\vec z_t)\sum_{j,\ell=1}^N
	\pa_{j\ell}^2 h_\theta^*(\vec z_s)\widetilde\Lambda_k(s)\, ds \, dt
$$
corresponds to \eqref{I1}, just replacing the function $\Lambda_k$ by
$$
  \widetilde{\Lambda}_k(s) := -\sum_{i=1}^N\pa_i h_\theta^*(\vec z_s)\,
	T_k''\bigg(\sum_{i=1}^N\pa_i h_\theta^*(\vec z_s)\bigg)
	= -\frac{1}{k}\sum_{i=1}^N\pa_i h_\theta^*(\vec z_s)\,
	T_1''\bigg(\frac{1}{k}\sum_{i=1}^N\pa_i h_\theta^*(\vec z_s)\bigg).
$$
Since $T_1$ is concave, we have $\widetilde{\Lambda}_k\ge 0$. An elementary computation shows that there exists $C>0$ such that for all $s\ge 0$,
$$
  -sT_1''(s) \le C(T_2'(s)-T_{1/2}'(s)) = C(T_1'(s/2)-T_1'(2s)),
$$
which implies that $\widetilde{\Lambda}_k(s)\le C(\Lambda_{2k}(s)-\Lambda_{k/2}(s))$.
Therefore,
\begin{equation}\label{3.QW}
  0 \le Q_k \le C(W_{2k}-W_{k/2}).
\end{equation}
Finally, by Lemma \ref{lem.Wk}, $(W_k)$ is monotone and, by Lemma \ref{lem.wk.bound},
$(W_k/(\lambda_1(\theta)+\lambda_2(\theta)))$ is bounded in $L^1(\Omega)$ and
monotone as well. We deduce from the monotone convergence theorem that
$(W_k/(\lambda_1(\theta)+\lambda_2(\theta)))$ is strongly converging in
$L^1(\Omega)$ which implies that
$$
  \lim_{k\to \infty} \int_{\Omega} \frac{W_{2k}-W_{k/2}}{\lambda_1(\theta)
	+\lambda_2(\theta)}\, dx =0.
$$
Inequality \eqref{3.QW} then directly leads to
$$
  \lim_{k\to \infty} \int_\Omega \frac{Q_k}{\lambda_1(\theta)+\lambda_2(\theta)}
	\, dx = 0.
$$

Now, we focus on $O_k$, defined in \eqref{def.Ok}. Since $T_k$ is concave,
$O_k$ is nonnegative. First, we show that $O_k\to 0$ a.e.\ in $\Omega$. As
$\vec{\rhotot}_\eps \rightharpoonup \vec \rhotot$ weakly
in $L^1(\Omega;\R^N)$, the sequence
$(\rhotot_{\eps})$ is uniformly equiintegrable. Therefore, by the weak lower semicontinuity of the $L^1(\Omega)$ norm,
\begin{align*}
\int_{\Omega}|O_k|\, dx &\le \int_{\Omega}|\rhotot - T_k(\rhotot)| \, dx + \liminf_{\eps \to 0} \int_{\Omega} |\rhotot_{\eps}-T_k(\rhotot_{\eps})|\, dx \\
&\leq\int_{\{\rhotot >k\}}|\rhotot - k| \, dx + \liminf_{\eps \to 0} \int_{\{\rhotot_{\eps} >k\}} |\rhotot_{\eps}-k|\, dx\\
&\leq\int_{\{\rhotot >k\}}\rhotot \, dx + \liminf_{\eps \to 0} \int_{\{\rhotot_{\eps} >k\}} \rhotot_{\eps}\, dx.
\end{align*}
Because of the weak convergence of $(\vec{\rhotot}_\eps)$, it holds that
$|\{\rhotot >k\}|+|\{\rhotot_{\eps} >k\}|\le C/k$ and consequently, the uniform
integrability of $(\rhotot_{\eps})$ shows that
$$
  \lim_{k\to \infty}\int_{\Omega}|O_k|\, dx =0,
$$
which implies for a subsequence that $O_k \to 0$ a.e.\ in $\Omega$. Second,
we show the proper bound on $O_k$, which will enable us to use the dominated
convergence theorem. We deduce from the weak lower semicontinuity of the
$L^2(\Omega)$ norm that
$$
  |O_k|^2 = |\overline{T_k(\rhotot_{\eps})}-T_{k}(\rhotot)|^2 \le \overline{|T_k(\rhotot_{\eps})-T_{k}(\rhotot)|^2}.
$$
Substituting the algebraic inequality
$$
  |T_k(s_1)-T_k(s_2)|^2 \le s_1T_k(s_1) +s_2 T_k(s_2) - s_1T_k(s_2) - s_2T_k(s_1)
$$
into the relation for $O_k$, we obtain
$$
|O_k|^2 \le \overline{\rhotot_{\eps}T_k(\rhotot_{\eps}) +\rhotot T_k(\rhotot) - \rhotot_{\eps}T_k(\rhotot) - \rhotot T_k(\rhotot_{\eps})}=\overline{\rhotot_{\eps}T_k(\rhotot_{\eps})} - \rhotot \overline{T_k(\rhotot_{\eps})}.
$$
Consequently, \eqref{Wk2} implies that
$$
\begin{aligned}
\frac{\theta|O_k|^2}{\lambda_1(\theta)+\lambda_2(\theta)}
\le \frac{K_2 W_k}{\lambda_1(\theta)+\lambda_2(\theta)} \quad \textrm{a.e. in }\Omega.
\end{aligned}
$$
Arguing as before, the right-hand side is convergent in $L^1(\Omega)$, and we can
apply the dominated convergence theorem to conclude the second part of
\eqref{sOO}.
\end{proof}

\subsection{Renormalized continuity equation}

We prove that the weak limit $\rhotot$ of $(\rhotot_\delta)$ is a renormalized
solution to the mass continuity equation. Note that the proof is different to the standard ones. Indeed, for our purpose, we just need  the pressure having at least linear growth with respect to the density, which is not the case in other works.

\begin{lemma}\label{lem.renorm}
Let all assumptions of Lemmata~\ref{lem.Wk} and \ref{lem.wk.bound} be satisfied and let the sequence $(\rhotot_{\delta},\vel_{\delta})$ satisfy \eqref{con-ren}, i.e., it solves for every $b \in C^{1}(\mathbb{R})$ with compactly supported derivative and all $\psi\in W^{1,\infty}(\Omega)$ the renormalized continuity equation
\begin{equation}\label{ren.cont.eq}
  -\int_\Omega b(\rhotot_{\delta})\vel_{\delta}\cdot\na\psi dx
	+ \int_\Omega\big(b'(\rhotot_{\delta})\rhotot_{\delta} - b(\rhotot_{\delta})\big)
	\diver\vel_{\delta}\psi \, dx = 0.
\end{equation}
Then the weak limit $(\rhotot, \vel)$ satisfies it as well.
\end{lemma}

\begin{proof}
We use \eqref{ren.cont.eq} with $b=T_k$:
$$
  -\int_\Omega T_k(\rhotot_\delta)\vel_\delta\cdot\na\psi \, dx
	+\int_\Omega\big(\rhotot_\delta T_k'(\rhotot_\delta)-T_k(\rhotot_\delta)\big)
	\diver \vel_\delta\psi \, dx =0,
$$
pass to the limit $\delta\to 0$, and use the fact that
$\vel_\delta\to \vel$ strongly in $L^2(\Omega)$ to infer that
$$
  -\int_\Omega \overline{T_k(\rhotot_\delta)}\vel\cdot\na\psi \, dx
	+ \int_\Omega\overline{(\rhotot_\delta T_k'(\rhotot_\delta)-T_k(\rhotot_\delta))
	\diver\vel_\delta}\psi \, dx=0,
$$
which is the weak formulation of
\begin{equation*}
\diver (\overline{T_k(\rhotot_\delta)}\vel) + \overline{(\rhotot_\delta T_k'(\rhotot_\delta)-T_k(\rhotot_\delta))
	\diver\vel_\delta}=0 \quad\textrm{in the sense of distributions}.
\end{equation*}
Since $\overline{T_k(\rhotot_\delta)}\in L^{\infty}(\Omega)$ and $\vel \in
W^{1,2}(\Omega)$, this equation can be renormalized; see, e.g.,
\cite[Theorem 10.29]{NoSt04}.
It follows for any Lipschitz continuous function $\beta$ with compactly supported
derivative and any function $\psi\in W^{1,\infty}(\Omega)$ that
\begin{equation}\label{sstt}
\begin{split}
-\int_{\Omega}&
b\big(\overline{T_k(\rhotot_\delta)}\big) \vel\cdot \nabla \psi\, dx
+ \int_{\Omega}\big(\overline{T_k(\rhotot_\delta)}
b'\big(\overline{T_k(\rhotot_\delta)}\big)
-b\big(\overline{T_k(\rhotot_\delta)}\big)\big)\diver\vel\psi\, dx\\
&= \int_\Omega b'\big(\overline{T_k(\rhotot_\delta)}\big)
\overline{(T_k(\rhotot_\delta)-\rhotot_\delta T_k'(\rhotot_\delta))
	\diver\vel_\delta}\psi \, dx.
\end{split}
\end{equation}
To identify and estimate the term on the right-hand side, we use definition
\eqref{def.Qk} of $Q_k$, the effective viscous flux identity \eqref{evf},
and a similar identity with $\rhotot_\delta T_k'(\rhotot_\delta)$ instead of
$T_k(\rhotot_\delta)$:
\begin{equation*}
\begin{split}
  Q_k&=\bigg(\lambda_2(\theta)+\frac43\lambda_1(\theta)\bigg)
	\big(\overline{(T_k(\rhotot_\delta)-\rhotot_\delta T_k'(\rhotot_\delta))\diver \vel_\delta}
	- \overline{(T_k(\rhotot_\delta)-\rhotot_\delta T_k'(\rhotot_\delta))}\diver\vel\big).
\end{split}
\end{equation*}
We insert this relation into \eqref{sstt} to obtain
\begin{equation}\label{sstt2}
\begin{split}
-\int_{\Omega}&
b\big(\overline{T_k(\rhotot_\delta)}\big) \vel\cdot \nabla \psi\, dx
+ \int_{\Omega}\big(\overline{T_k(\rhotot_\delta)}
b'\big(\overline{T_k(\rhotot_\delta)}\big)
-b\big(\overline{T_k(\rhotot_\delta)}\big)\big)\diver\vel\psi\, dx\\
&= \int_\Omega b'\big(\overline{T_k(\rhotot_\delta)}\big)
\overline{(T_k(\rhotot_\delta)-\rhotot_\delta T_k'(\rhotot_\delta))}
\diver\vel\psi \, dx + \int_\Omega \frac{b'(\overline{T_k(\rhotot_\delta)})Q_k
}{\lambda_2(\theta)+\frac43\lambda_1(\theta)}\psi \, dx.
\end{split}
\end{equation}

Our goal is  to perform the limit $k\to \infty$ in \eqref{sstt2}
to recover \eqref{ren.cont.eq}
with $(\rhotot_\delta,\vel_\delta)$ replaced by $(\rhotot,\vel)$. We first show that
$\overline{T_k(\rhotot_\delta)} \to \rhotot$ as $k\to \infty$ a.e.\ in $\Omega$.
To this end, we recall the weak lower semicontinuity of the $L^1(\Omega)$ norm
leading to
\begin{align*}
  \int_\Omega\big|\overline{T_k(\rhotot_\delta)}-\rhotot\big|\, dx
	&\le \liminf_{\delta\to 0}\int_\Omega|T_k(\rhotot_\delta)-\rhotot_\delta|\,dx
	\leq \liminf_{\delta\to 0}\int_{\{\rhotot_\delta\ge k\}}(\rhotot_\delta-T_k(\rhotot_\delta))\,dx \\
	&\le \sup_{\delta >0}\int_{\{\rhotot_\delta\ge k\}}\rhotot_\delta\,dx.
\end{align*}
Since $(\rhotot_{\delta})$ is weakly convergent in $L^1(\Omega)$
(assumption \eqref{Kon1}), we have $|\{\rhotot_\delta\ge k\}|\le C/k$, and since
the sequence is uniformly equiintegrable, we deduce that
\begin{equation}\label{3.b1}
  \limsup_{k\to \infty} \int_\Omega\big|\overline{T_k(\rhotot_\delta)}-\rhotot\big|\, dx
  \le \limsup_{k\to \infty} \bigg(\sup_{\delta > 0}\int_{\{\rhotot_\delta\ge k\}}
	\rhotot_\delta\,dx\bigg)=0
\end{equation}
which shows the pointwise convergence of $\overline{T_k(\rhotot_\delta)}$
(at least for a subsequence). Similarly,
\begin{equation}\label{3.b2}
  \lim_{k\to\infty}\int_\Omega\big|\overline{T_k(\rhotot_\delta)-\rhotot_\delta
	T_k'(\rhotot_\delta)}\big|\, dx = 0.
\end{equation}
The concavity of $T_k$ implies that $|T_{k}(s) - sT_k'(s)|\le T_k(s)$ and
consequently,
$$
\big|\overline{T_k(\rhotot_\delta)-\rhotot_\delta
	T_k'(\rhotot_\delta)}\big| \le \overline{T_k(\rhotot_\delta)}
	\quad \textrm{a.e. in } \Omega.
$$
Thus, since $b$ is bounded and its derivative is compactly supported, there
exists $C>0$ only depending on $b$ (but not on $k$) such that
$$
  \big|b\big(\overline{T_{k} (\rhotot_{\delta})}\big)\big|
	+ \big|\overline{T_k(\rhotot_\delta)} b'\big(\overline{T_k(\rhotot_\delta)}
	\big) - b\big(\overline{T_k(\rhotot_\delta)}\big)\big|
	+ \big|b'\big(\overline{T_k(\rhotot_\delta)}\big)
	\overline{(T_k(\rhotot_\delta)-\rhotot_\delta T_k'(\rhotot_\delta))}\big|\le C.
$$
This estimate and \eqref{3.b1}--\eqref{3.b2} lead to
\begin{align*}
  b\big(\overline{T_k(\rhotot_\delta)}\big) \to b(\rho)
	&\quad\textrm{strongly in }L^q(\Omega), \\
	b'\big(\overline{T_k(\rhotot_\delta)}\big)(\overline{T_k(\rhotot_\delta)
	-\rhotot_\delta T_k'(\rhotot_\delta)}) \to 0 &\quad\textrm{strongly in }L^q(\Omega), \\
	b'\big(\overline{T_k(\rhotot_\delta)}\big)\overline{T_k(\rhotot_\delta)}
	- b\big(\overline{T_k(\rhotot_\delta)}\big)
	\to b'(\rhotot)\rhotot-b(\rhotot) &\quad\textrm{strongly in }L^q(\Omega)
\end{align*}
for any $q<\infty$.
These convergence results, estimate \eqref{sOO} for $Q_k$, and the
boundedness of $b'$ allow us to pass
to the limit $k\to\infty$ in \eqref{sstt2} which gives \eqref{ren.cont.eq}
for $(\rhotot,\vel)$.
\end{proof}

\subsection{Strong convergence of the densities}

We finally show that $\vec{\rhotot}_{\delta}\to \vec{\rhotot}$ a.e.\
in $\Omega$. First, we prove that the total mass density $\rhotot_{\delta}$
converges pointwise and then, exploiting the compactness of $\Pi \vec{\mu}_{\delta}$
and the convexity of the free energy, we deduce the compactness of the sequence
$(\vec{\rhotot}_{\delta})$. We start with the convergence of the total mass densities.

\begin{lemma}[Strong convergence of the total mass densities]\label{thm.conv}\sloppy\ \
Let the assumptions of Lemmata~\ref{lem.eff}--\ref{lem.wk.bound} be satisfied and let
$(\rhotot_{\delta},\vel_{\delta})$ solve the renormalized
mass continuity equation \eqref{ren.cont.eq}. Then
\begin{equation}\label{sc.den}
\rhotot_{\delta} \to \rhotot \qquad \textrm{strongly in }L^1(\Omega)\mbox{ as }
\delta\to 0.
\end{equation}
\end{lemma}

\begin{proof}
The idea of the proof is to use the effective flux identity \eqref{evf} and combine it with the properties of $W_k$, defined in \eqref{Wk}. We prove the key property
\begin{equation}
\lim_{k\to \infty} \int_{\Omega}\big( \overline{T_{k}(\rhotot_{\delta})\diver \vel_{\delta}}-\overline{T_{k}(\rhotot_{\delta})}\diver \vel\big)\, dx =0. \label{zero22}
\end{equation}
We claim that if $(\rhotot, \vel)$ or
$(\rhotot_{\delta}, \vel_{\delta})$, respectively, are renormalized solutions, then
\begin{equation}
0=\int_{\Omega} T_k(\rhotot) \diver \vel\, dx = \int_{\Omega} T_k(\rhotot_{\delta}) \diver \vel_{\delta}\, dx. \label{nula}
\end{equation}
For this purpose, we choose $\psi=1$ in \eqref{ren.cont.eq} such that for any bounded function $b$ with compactly supported derivative, we have
$$
\int_{\Omega} \big(b'(\rhotot)\rhotot - b(\rhotot)\big)\diver \vel\, dx =0,
$$
and a similar identity holds for $\rhotot_{\delta}$. For any continuous
compactly supported function $f$ in $(0,\infty)$, we define
$$
  b_f(s):=-s\int_{s}^{\infty}\frac{f(t)}{t^2}\,dt, \quad s\ge 0.
$$
Then $b_{f}$ is bounded, has a compactly supported derivative, and it holds that
$b_{f}'(s) s -b_f (s) = f(s)$ for $s\ge 0$. Hence,
$$
  0=\int_{\Omega}\big(b_{f}'(\rhotot) \rhotot - b_f (\rhotot)\big)\diver \vel \,dx
  = \int_{\Omega} f(\rhotot)\diver \vel\, dx.
$$
Let $(f^m)$ be a sequence of compactly supported functions satisfying
$f^m\nearrow T_k$ pointwise as $m\to\infty$. Choosing $f=f^m$ in the previous
identity and passing to the limit $m\to\infty$ leads to \eqref{nula}.
Of course, the same argument applies to $(\rhotot_{\delta}, \vel_{\delta})$.
Thus,
\begin{equation*}
  \int_\Omega \overline{T_k(\rhotot_\delta)\diver\vel_\delta}\,  dx =\lim_{\delta \to 0}\int_\Omega T_k(\rhotot_\delta)\diver\vel_\delta \,dx = 0,
\end{equation*}
and the first part in \eqref{zero22} vanishes. It remains to estimate the second part.
It follows from \eqref{def.Ok}, \eqref{sOO}, \eqref{nula} and the
Cauchy--Schwarz inequality that
\begin{align*}
\lim_{k\to\infty}&\left|\int_\Omega \overline{T_k(\rhotot_\delta)} \diver\vel \, dx\right| = \lim_{k\to\infty}\left|\int_\Omega ( \overline{T_k(\rhotot_\delta)}-T_k(\rhotot) )\diver\vel \, dx\right|
\le \lim_{k\to\infty}\int_\Omega | O_k|\, |\diver\vel| \, dx \nonumber \\
&\le  \left(\int_\Omega \frac{\lambda_1(\theta)+\lambda_2(\theta)} {\theta}
|\diver\vel|^2 \, dx\right)^{1/2} \lim_{k\to\infty}\left(\int_\Omega \frac{\theta O_k^2}{\lambda_1(\theta)+\lambda_2(\theta)} \, dx\right)^{1/2}=0, 
\end{align*}
and \eqref{zero22} holds true.
Thus, using \eqref{evf} and \eqref{Wk} in \eqref{zero22}, it follows that
\begin{align*}
  0 &= \lim_{k\to\infty}\int_\Omega\big(\overline{T_k(\rhotot_\delta)\diver\vel_\delta}
	- \overline{T_k(\rhotot_\delta)}\diver\vel\big)\, dx \\
	&= \lim_{k\to\infty}\int_\Omega\frac{\overline{p_\delta T_k(\rhotot_\delta)}
	- p\overline{T_k(\rhotot_\delta)}}{\lambda_2(\theta)+\frac43\lambda_1(\theta)}\,dx
  = \lim_{k\to\infty}\int_\Omega\frac{W_k}{\lambda_2(\theta)
	+ \frac43 \lambda_1(\theta)}\, dx.
\end{align*}
Hypothesis (H3) and the nonnegativity of $W_k$ yield $\lim_{k\to\infty}\int_\Omega W_k\,dx=0$, which means that $\lim_{k\to\infty} W_k=0$ a.e.~in $\Omega$. The fact 
that $W_k\le W_{k+1}$ for $k\geq 1$ a.e.~in $\Omega$ implies that
\begin{equation*}
W_k \equiv 0 \quad \textrm{a.e. in } \Omega,\ k\in\N. 
\end{equation*}
Hence, \eqref{Wk2} shows that $\theta(\overline{\rhotot_\delta T_k(\rhotot_\delta)}
-\rhotot\overline{T_k(\rhotot_\delta)})=0$ in $\Omega$ and so, since $\theta>0$ a.e.\ in $\Omega$,
\begin{equation}\label{finale}
  0 = \int_\Omega\big(\overline{\rhotot_\delta T_k(\rhotot_\delta)}
  -\rhotot\overline{T_k(\rhotot_\delta)}\big)dx
  = \lim_{\eps \to 0} \lim_{\delta \to 0} \int_{\Omega} (\rhotot_\delta
	- \rhotot_{\eps})(T_k(\rhotot_{\delta}) - T_k(\rhotot_{\eps}))\, dx.
\end{equation}
Finally, using the weak lower semicontinuity of the $L^1(\Omega)$ norm,
the Cauchy--Schwarz inequality, and \eqref{finale},
\begin{align*}
&\lim_{\eps \to 0} \int_{\Omega} |\rhotot_{\eps}-\rhotot|\, dx\le \lim_{\eps \to 0} \lim_{\delta\to 0}\int_{\Omega} |\rhotot_{\eps}-\rhotot_{\delta}|\, dx\\
 &\le  \lim_{\eps \to 0} \lim_{\delta\to 0}\int_{\{\rhotot_{\delta}>k\}\cup\{\rhotot_{\eps}>k\}} |\rhotot_{\eps}-\rhotot_{\delta}|\, dx + \lim_{\eps \to 0} \lim_{\delta\to 0}\int_{\Omega} |T_k(\rhotot_{\eps})-T_k(\rhotot_{\delta})|\, dx\\
&\le \sup_{\eps, \delta >0} \int_{\{\rhotot_{\delta}>k\}\cup\{\rhotot_{\eps}>k\}} |\rhotot_{\eps}-\rhotot_{\delta}|\, dx + |\Omega|^{1/2}\lim_{\eps \to 0} \lim_{\delta\to 0} \left(\int_{\Omega} |T_k(\rhotot_{\eps})-T_k(\rhotot_{\delta})|^2\, dx\right)^{1/2}\\
&\le \sup_{\eps, \delta >0} \int_{\{\rhotot_{\delta}>k\}\cup\{\rhotot_{\eps}>k\}} |\rhotot_{\eps}-\rhotot_{\delta}|\, dx + |\Omega|^{1/2}\lim_{\eps \to 0} \lim_{\delta\to 0} \left(\int_{\Omega} (\rhotot_{\eps}-\rhotot_{\delta})(T_k(\rhotot_{\eps})-T_k(\rhotot_{\delta}))\, dx\right)^{1/2}\\
&= \sup_{\eps, \delta >0} \int_{\{\rhotot_{\delta}>k\}\cup\{\rhotot_{\eps}>k\}} |\rhotot_{\eps}-\rhotot_{\delta}|\, dx.
\end{align*}
The left-hand side is independent of $k$ and therefore, we may perform the limit
$k\to \infty$ also on the right-hand side. Then, thanks to uniform equiintegrability of $(\rhotot_{\delta})$, we deduce \eqref{sc.den}, which finishes the proof.
\end{proof}

We end this subsection, and also the proof of Theorem~\ref{CC.thm}, by the following lemma, which shows the compactness of the vector $\vec{\rhotot}$ of mass densities. In addition, we show that either the total mass density vanishes, i.e., all partial mass densities are zero, or all partial mass densities are strictly positive a.e.\ in $\Omega$.

\begin{lemma}[Strong convergence of the vector of mass densities]\label{thm.conv.total}
Let all assumptions of Lemmata~\ref{lem.eff}--\ref{lem.wk.bound} be satisfied and
let $(\rhotot_{\delta},\vel_{\delta})$ solve the renormalized
mass continuity equation \eqref{ren.cont.eq}. Then
\begin{equation}\label{sc.den.2}
\vec{\rhotot}_{\delta} \to \vec{\rhotot} \textrm{ strongly in }
L^1(\Omega; \mathbb{R}^N)\quad\mbox{as }\delta\to 0.
\end{equation}
Moreover, for a.e.\ $x\in \Omega$ and all $i=1,\ldots,N$,
\begin{equation}\label{aeposit}
\rhotot(x)>0 \implies \rhotot_i (x) >0.
\end{equation}
\end{lemma}

\begin{proof}
Taking into account the strong convergence results, derived in this section,
Egorov's theorem implies that for any $\eta>0$, there exists a set $\Omega_{\eta}$
such that $|\Omega \setminus \Omega_{\eta}|\le \eta$ and
\begin{align*}
\rhotot_{\delta} \to \rhotot &\quad\textrm{strongly in } L^{\infty}(\Omega_{\eta}),\\
\Pi \vec{\mu}_{\delta} \to \overline{\Pi \vec{\mu}_{\delta}}
&\quad\textrm{strongly in } L^{\infty}(\Omega_{\eta}; \mathbb{R}^N),\\
\theta_{\delta} \to \theta &\quad\textrm{strongly in } L^{\infty}(\Omega_{\eta}),\\
\ln \theta_{\delta} \to \ln \theta &\quad\textrm{strongly in }
L^{\infty}(\Omega_{\eta}).
\end{align*}

We collect some facts about the free energy density 
$h_{\theta}$. Since \eqref{LMB2} and \eqref{H7C} hold and the mapping 
$(\vec{\rhotot},\theta)\in\R^N_+\times\R_+\mapsto \na_{\vec{\rhotot}}
h_{\theta}(\vec{\rhotot})\in\R^N$ is continuous, for every $\kappa >0$, there exists $C(\kappa)>0$ such that for all $\theta \in (\kappa, \kappa^{-1})$,
\begin{equation}
\begin{split}
  \rho_i >\kappa &\implies \mu_i = \frac{\pa h_{\theta}}{\pa\rho_i}(\vec{\rhotot})
  \ge -C(\kappa),\\
  |\vec{\rho}| \le \kappa^{-1} &\implies \mu_i = \frac{\pa h_{\theta}}{\pa\rho_i}
	(\vec{\rho})\le C(\kappa) \textrm{ for all } i=1,\ldots, N.\label{kappa}
\end{split}
\end{equation}

Finally, we focus on the strong convergence of the densities. We consider two cases:
the total density vanishes or the total density remains positive.
The first case is easy, so we assume that the total density is positive.
We define the set
$$
\Omega_{\eta,\kappa}:=\{x\in \Omega_{\eta};\, \rhotot(x)\ge 2N\kappa\}.
$$
Thanks to the uniform convergence of $(\rhotot_{\delta})$,
there exists $\delta_0>0$ such that for all $0<\delta<\delta_0$, it holds that
$\rhotot_{\delta}(x)\ge N\kappa$ for $x\in \Omega_{\eta,\kappa}$. Consequently, for
$x\in \Omega_{\eta,\kappa}$ and $0<\delta<\delta_0$, there exists
$i\in\{1,\ldots,N\}$ such that
$$
  \kappa\le \rho_{\delta,i}(x) \le C(\eta),
$$
since $(\rho_{\delta,i})$ converges uniformly in $\Omega_{\eta,\kappa}\subset
\Omega_\eta$. Therefore, using \eqref{kappa}, we infer that
$|\mu_{\delta,i}(x)|\le C(\kappa,\eta)$. In addition, thanks to the uniform
convergence of $(\rho_{\delta})$ and \eqref{kappa},
$\mu_{\delta,j}(x)\le C_1(\eta,\kappa)$ for $j=1,\ldots,N$.
By the uniform convergence of $\Pi \vec{\mu}_{\delta}$, we obtain
$$
  \sum_{j=1}^N \mu_{\delta,j}(x)
	= N\mu_{\delta,i}(x)-N(\Pi\vec{\mu}_\delta)_i (x)
	\ge -C_2(\eta,\kappa,N) \quad\mbox{for }x\in\Omega_{\eta,\kappa}.
$$
The bounds $\mu_{\delta,j}\le C(\eta,\kappa)$ for $j=1,\ldots,N$ and $\sum_{j=1}^N\mu_{\delta,j}\ge -C_2(\eta,\kappa,N)$ in $\Omega_{\eta,\kappa}$
imply
$$
  \mu_{\delta,j}\ge -C_2(\eta,\kappa,N) - \sum_{\ell\neq j}\mu_{\delta,\ell}
	\ge C_2(\eta,\kappa,N) - (N-1)C(\eta,\kappa) \quad\mbox{in }
	\Omega_{\eta,\kappa},\ j=1,\ldots,N.
$$
We conclude that
\begin{equation}\label{mubund}
C_3(\kappa,N,\eta)\le \mu_{\delta,j}(x) \le C_4(\kappa, \eta,N)
\quad\mbox{for }x\in\Omega_{\eta,\kappa},\ j=1,\ldots,N.
\end{equation}
Consequently, since $\vec{\mu}_{\delta}=\na_{\vec{\rho}} h_{\theta_{\delta}}
(\vec{\rho}_{\delta})$ and $\theta_\delta\to\theta$ a.e.\ in $\Omega$, we have
\begin{equation}\label{mubund1}
\lim_{\delta \to 0} |\vec{\mu}_{\delta}(x)
-\na_{\vec{\rho}} h_{\theta(x)}(\vec{\rho}_{\delta}(x))|=0,\quad
\mbox{for }x\in\Omega_{\eta,\kappa}.
\end{equation}
Using the convexity of $h_{\theta}$, \eqref{mubund1},
and the compactness of $(\rhotot_{\delta})$, $(\Pi \vec{\mu}_{\delta})$,
we have in $\Omega_{\eta,\kappa}$,
\begin{align*}
0&\le \lim_{\eps \to 0}\lim_{\delta \to 0}(\vec{\rhotot}_{\eps}-\vec{\rhotot}_{\delta}) \cdot (\na_{\vec{\rhotot}} h_{\theta}(\vec{\rhotot}_{\eps})-\na_{\vec{\rhotot}} h_{\theta}(\vec{\rhotot}_{\delta}))= \lim_{\eps \to 0}\lim_{\delta \to 0}(\vec{\rhotot}_{\eps}-\vec{\rhotot}_{\delta}) \cdot (\vec{\mu}_{\eps} - \vec{\mu}_{\delta})\\
&=\lim_{\eps \to 0}\lim_{\delta \to 0}(\vec{\rhotot}_{\eps}-\vec{\rhotot}_{\delta}) \cdot ((\vec{\mu}_{\eps} - \vec{\mu}_{\delta})-\Pi ((\vec{\mu}_{\eps} - \vec{\mu}_{\delta})))\le C\lim_{\eps \to 0}\lim_{\delta \to 0}|\rhotot_{\eps}-\rhotot_{\delta}|=0.
\end{align*}
The last but one equality follows from the strong convergence of
$\Pi\vec{\mu}_\eps$, while the last inequality follows from \eqref{mubund}.
Hence, using the strict convexity of $h_\theta$, this gives
$$
 \lim_{\eps \to 0}\lim_{\delta \to 0}|\vec{\rhotot}_{\eps}(x)-\vec{\rhotot}_{\delta}(x)|=0 \quad \text{ a.e. in } \Omega_{\eta,k}.
$$
Consequently,  using the weak lower semicontinuity of the $L^1(\Omega)$ norm,
\begin{align*}
\lim_{\delta\to 0}\int_{\Omega_{\eta}} |\vec{\rho}_{\delta}-\vec{\rho}|\, dx
&\le\lim_{\delta\to 0} \int_{\Omega_{\eta,\kappa}}
|\vec{\rho}_{\delta}-\vec{\rho}|\, dx + C\kappa \\
&\le \lim_{\eps\to 0} \lim_{\delta\to 0} \int_{\Omega_{\eta,\kappa}}
|\vec{\rho}_{\eps}-\vec{\rho}_{\delta}|\, dx + C\kappa = C\kappa.
\end{align*}
The limit $\kappa\to 0$ gives
$$
\lim_{\delta\to 0}\int_{\Omega_{\eta}} |\vec{\rho}_{\delta}-\vec{\rho}|\, dx =0.
$$
Thanks to the uniform equiintegrability of $(\vec{\rho}_{\delta})$, this limit
implies \eqref{sc.den.2}.

Relation \eqref{aeposit} can be proved by using the approach presented before. Indeed,
since $\vec{\rho}$ and $\Pi (\vec{\mu})$ 
are finite a.e.\ in $\Omega$, then
$\rhotot(x)>0$ implies the existence of some $i\in\{1,\ldots,N\}$ such that
$\rho_i(x)>0$. But then the whole vector $\vec{\mu}(x)$ must be finite and
consequently, $\rho_j(x)>0$ for all $j=1,\ldots, N$, which equals \eqref{aeposit}.
\end{proof}

To finish the proof of Theorem \ref{CC.thm}, recall that
$$
\vec{\mu_\delta} = \pa_{\vec{\rhotot}} h_{\theta_\delta}(\vec{\rho_\delta}) \to \pa_{\vec{\rhotot}} h_{\theta}(\vec{\rho})=: \vec{\mu}\quad\mbox{a.e. in }\Omega
\setminus\{\rhotot=0\}.
$$
Since we cannot exclude that $\rhotot=0$ on a set of positive measure, we redefine
$\vec{\mu}$ and set $\vec{\mu}:=0$ on $\{\rhotot=0\}$.
On the other hand, we know that
$$
\Pi\Big(\frac{\vec{\mu_\delta}}{\theta_\delta}\Big) \to \Pi (\vec{q})
\quad\mbox{a.e. in }\Omega,
$$
and therefore also
$$
\Pi(\vec{\mu_\delta}) =\theta_\delta \Pi\Big(\frac{\vec{\mu_\delta}}{\theta_\delta}\Big) \to \theta \Pi(\vec{q})\quad\mbox{a.e. in }\Omega.
$$
Therefore, $\Pi(\vec{\mu}) = \theta \Pi(\vec{q})$ a.e.\ in $\Omega
\setminus\{\rhotot =0\}$. This shows that $\sum_{j=1}^NM_{ij}\na q_j
= \sum_{j=1}^N M_{ij}$ $\times\na(\mu_j/\theta)$ as required in \eqref{w.rhos}.


\section{The approximate scheme}\label{sec.scheme}

\subsection{Formulation of the approximate equations}

In this subsection, we specify the approximate system, leading to a sequence
of smooth, approximate solutions.

\subsubsection{Internal energy balance.}
We replace the total energy balance \eqref{w.rhovE} by the internal energy balance
\eqref{2.enerbal},
whose weak formulation reads as
\begin{align}
\nonumber
\int_\Omega&\Big( -\rhotot e \bm v + \kappa(\theta)\nabla\theta + \sum_{i=1}^N
M_i\nabla\left( \frac{\mu_i}{\theta} \right) \Big)\cdot\na\phi_0 \, dx \\
\nonumber
&{}+ \int_\Omega(p\diver\bm  v - \St : \nabla \bm v)\phi_0 \, dx
+ \alpha_2\int_{\pa\Omega}\phi_0 (\theta-\theta_0)\, ds = 0\quad
\mbox{for all }\phi_0\in H^1(\Omega).
\end{align}

\subsubsection{Parameters of the approximate problem}
For the sake of clarity, we state here the parameters employed in the approximate scheme. A more detailed explanation is given in the part that follows:
\begin{center}
\begin{tabular}{ll}
$\eps\in (0,1)$: & lower-order regularization in all equations;\\
$\xi\in (0,1)$: & higher-order regularization in the internal energy equation;\\
$\delta\in (0,\xi/2)$: &
higher-order regularization in the partial mass densities equations;\\
$\chi\in (0,1)$: & quasilinear regularization in the partial mass densities equations;\\
$n\in\N$: & Galerkin approximation in the momentum equation;\\
$\eta\in (0,1)$: & regularization in the free energy and heat conductivity.
\end{tabular}
\end{center}

\subsubsection{Levels of approximations}
Let $(\vk{k})_{k\in\N}\subset \{\bm{v} \in C^\infty(\overline{\Omega}), \,
\bm{v}\cdot \vcg{\nu} =0 \text{ on }\partial \Omega\}$ be a complete
orthonormal system for $L^2(\Omega)$ and let
$X_n$ be the subspace of $L^2(\Omega)$ generated by $\vk{1},\ldots,\vk{n}$.
Since $X_n$ has dimension $n<\infty$, the quantity
\begin{align}\label{def.Kn}
K(n) := \sup_{\mathbf{u}\in X_n\setminus\{\mathbf{0}\}}\frac{\int_\Omega
|D^2 (|\mathbf{u}|^2)|^2 \, dx}{\int_\Omega |\mathbf{u}|^4 \, dx}
\end{align}
is finite. Let $\overline\rho > 0$ be an arbitrary positive constant (which in the
end will be the mean density of the solution).
Let $q_0=\widetilde{q}_0(\log\theta)$, $\vec{\mathfrak{R}} = (\mathfrak{R}_1,
\ldots,\mathfrak{R}_N)$, $\mathfrak{R}_i = \widetilde{\mathfrak{R}}_i(\log\theta)$
($i=1,\ldots,N$), $\mathfrak{R}=\sum_{j=1}^N\mathfrak{R}_j$, defined in
Lemma \ref{lem.refrho}. Note that $q_0$ is actually a constant.

Let $L\geq\max\{\beta, 3\beta-2, \gamma+\nu, 2\beta_0\}$ be arbitrary, where $\nu>0$
is as in Lemma \ref{lem.dens} and $\beta_0\geq 0$ is as in Hypothesis (H6), formula
\eqref{H6.b}. We introduce the ``regularized'' free energy for the approximate system,
\begin{equation}\label{h.reg}
h_\theta^{(\eta)}(\vec\rho) = h_\theta(\vec\rho) + \eta \overline{n}^{L}
+\eta \sum_{i=1}^N f_{\alpha_0}(\rho_i) + \eta |\vec\rho|^2,
\end{equation}
where $\alpha_0\in (\frac12,1)$ is the same as in Hypothesis (H6), formula
\eqref{H6.b}, and
$$
f_{\alpha_0}(s) =
\frac{s^{2(1-\alpha_0)}}{2(1-\alpha_0)(1-2\alpha_0)} ,
\qquad s\geq 0.
$$
All other thermodynamic quantities appearing in the approximate equations that are
defined in terms of the free energy are modified correspondingly:
\begin{align*}
\mu^{(\eta)}_i &= \mu_i + \eta L \frac{\overline{n}^{L-1}}{m_i}
+\eta f_{\alpha_0}'(\rho_i) + 2\eta \rho_i,\quad i=1,\ldots,N,\\
p^{(\eta)} &= p + \eta(L-1)\overline{n}^L
+ \eta\sum_{i=1}^N\frac{\rho_i^{2(1-\alpha_0)}}{2(1-\alpha_0)}
+\eta |\vec\rho|^2,\\
\rho e^{(\eta)} &= \rho e + \eta \overline{n}^{L}
+\eta \sum_{i=1}^N f_{\alpha_0}(\rho_i) +
\eta |\vec\rho|^2,\\
\rho s^{(\eta)} &= \rho s .
\end{align*}
We point out that $h_\theta^{(\eta)}$ satisfies the property from Lemma
\ref{l A1} and Hypothesis (H6); in particular, \eqref{H6.b} holds with the same
constants, since the mixed second-order derivatives of the free energy are not
changed by the terms added in \eqref{h.reg}.

We also need to regularize the heat conductivity:
\begin{equation}\label{reg.kappa}
\kappa^{(\eta)}(\theta) = \kappa(\theta) + \eta\theta^L .
\end{equation}

\subsubsection{Approximate problem}
The approximate equations involve $q_1,\ldots,q_N,\log\theta\in$ $H^2(\Omega)$,
$\vel\in X_n$ as unknown and are given as follows:
\begin{align}
\nonumber
 0 &= \sum_{i=1}^N\int_\Omega \bigg(-\rho_i\vel + \sum_{j=1}^N M_{ij}\na q_j
+ M_i\nabla\bigg(\frac{1}{\theta} \bigg) \bigg)\cdot\nabla\phi_i \, dx
- \sum_{i=1}^N\int_\Omega r_i\phi_i \, dx \\
\label{app.rhoi}
&\phantom{xx}{}+\delta\sum_{i=1}^N\int_\Omega D^2 q_i : D^2\phi_i \, dx
+\eps\sum_{i=1}^N\int_\Omega (\rho_i - \mathfrak{R}_i)
\phi_i \, dx\\
\nonumber
&\phantom{xx}{}+\eps^3 \sum_{i=1}^N\int_\Omega\left(
(1+\chi |\na q_i|^2)\na q_i \cdot\na\phi_i
+ (1+|q_i-q_0|^{1/2})(q_i-q_0) \phi_i \right) \, dx; \\
\nonumber
 0 &= \int_{\pa\Omega}\alpha_1\bm u\cdot \vel \, ds +
\int_\Omega \left( -\rho \vel\otimes\vel + \St \right):\nabla \bm u\, dx -
\int_\Omega (p\diver\bm u  + \rho\bm b \cdot\bm u)\, dx\\
\nonumber 
&\phantom{xx}{}+\eps^3\sum_{i=1}^d\sum_{j=1}^N\int_\Omega v_i\bigg(
(1+\chi |\na q_j|^2)\nabla q_j \cdot\nabla u_i
+ \frac{1}{2}(1+|q_j-q_0|^{1/2})(q_j-q_0) u_i \bigg)\,dx\\
\nonumber
&\phantom{xx}{}+\eps\int_\Omega\rho \vel \cdot\bm u \, dx
+\delta K(n)\int_\Omega |\vel|^2 \vel\cdot\bm u \, dx; \\
\nonumber
 0 &= \int_\Omega\bigg( -\rho e \vel + \kappa^{(\eta)}(\theta)\nabla\theta
+ \sum_{i=1}^N M_i\nabla q_i \bigg)\cdot\na\phi_0 \, dx
+ \int_\Omega(p\diver\vel - \St :\nabla\vel)\phi_0 \, dx\\
\label{app.rhoe}
&\phantom{xx}{}+ \int_{\pa\Omega} \alpha_2
(\theta-\theta_0)\phi_0\, ds
+\delta\int_\Omega (1+\theta) D^2 (\log\theta) : D^2\phi_0 \, dx \\
\nonumber
&\phantom{xx}{}+\xi\int_\Omega\left( (1+\theta)(1+|\na\log\theta|^2)
\na(\log\theta ) \cdot\na\phi_0 + (\log\theta)\phi_0\right) \, dx
-\eps\int_\Omega\mathfrak{R}\frac{|v|^2}{2}\phi_0 \, dx
\end{align}
for all $\phi_1,\ldots,\phi_N\in H^2(\Omega)$, $\mathbf{u}\in X_n$, and
$\phi_0\in H^2(\Omega)$. Given $q_1,\ldots,q_n$, $\theta$, we define
\begin{equation} \label{4.10a}
  \vec\rho:=\na (h_\theta^{(\eta)})^*(\theta\vec{q}) \quad\mbox{and}\quad
	\rho:=\sum_{i=1}^n\rho_i.
\end{equation}
The system will be reformulated as a fixed-point problem for an operator
$F : X\times [0,1]\to X$ defined by means of a linearized problem.


\subsection{Existence of a solution to the approximate equations}

We formulate \eqref{app.rhoi}--\eqref{app.rhoe} as a fixed-point problem for
a suitable operator.

\subsubsection{Reformulation as a fixed-point problem}
Define the spaces
$$
  V = W^{1,4}(\Omega; \R^N)\times X_n \times W^{1,4}(\Omega),\quad
  V_0 = H^2(\Omega;\R^N)\times X_n \times H^2(\Omega).
$$
Let $(\vec q\,^*,\vel^*,w^*)\in V$ and $\sigma\in [0,1]$ be arbitrary.
Noting that $\vec\rho(\mu^*,\theta^*)$ is defined as in \eqref{4.10a} and
$(\rho e)^* = (\rho e)(\vec\rho\,^*,\theta^*)$ is defined by means of \eqref{1.p},
we set for $i,j=1,\ldots,N$:
\begin{align*}
&\theta^* = e^{w^*},\quad \mu_i^* = \theta^* q_i^*,\quad
\vec\rho\,^* = \vec\rho(\vec\mu\,^*,\theta^*),\quad
\rho^* = \sum_{k=1}^N\rho_k^*,\quad
p^* = p(\rho^*,\theta^*) ,\\
& M_i^*=M_i(\vec\rho\,^*,\theta^*),\quad M_{ij}^*=M_{ij}(\vec\rho\,^*,\theta^*), \quad
r_i^* = r_i(\Pi(\vec{\mu}^*/\theta^*),\theta^*).
\end{align*}
The task is to find a solution $(\vec q,\bm v,w)\in V_0$ to the linear system
\begin{align}
\label{lin.rhoi}
  0 &= \sigma\sum_{i=1}^N\int_\Omega \bigg( -\rho_i^* \vel^*
	+ \sum_{j=1}^N M_{ij}^*\na q_j^* -  M_i^*\frac{\nabla w^*}{\theta^*} \bigg)
	\cdot\nabla\phi_i \,dx - \sigma\sum_{i=1}^N\int_\Omega r_i^*\phi_i \,dx \\
\nonumber
  &\phantom{xx}{}+\delta\sum_{i=1}^N\int_\Omega D^2 q_i : D^2\phi_i \,dx
+\sigma\eps\sum_{i=1}^N\int_\Omega (\rho_i^* - \widetilde{\mathfrak{R}}_i(w))
\phi_i \,dx\\
\nonumber
  &\phantom{xx}{}+\eps^3 \sum_{i=1}^N\int_\Omega\left(
(1+\chi |\na q_i^*|^2)\na q_i \cdot\na\phi_i
+ (1+|q_i^*-\widetilde q_0(w)|^{1/2})(q_i-\widetilde q_0(w))\phi_i\right)\,dx; \\
\label{lin.rhov}
  0 &= \int_{\pa\Omega}\alpha_1 \bm u\cdot \bm v
\,ds + \int_\Omega \St(\theta^*,\D(v)):\nabla \bm u\, dx \\
\nonumber
  &\phantom{xx}{}+ \sigma\int_\Omega\left(\left( -\rho^* \bm v^*\otimes \bm v^*
\right):\nabla\bm u -p^*\diver \bm u  - \rho^* \bm b \cdot \bm u \right)\, dx \\
\nonumber
  &\phantom{xx}{}+\eps^3\sum_{i=1}^d\sum_{j=1}^N\int_\Omega v_i^* \bigg(
(1+\chi |\na q_j^*|^2)\nabla q_j \cdot\nabla u_i + \frac{1}{2}
(1+|q_j-\widetilde q_0(w)|^{1/2})(q_j-\widetilde{q}_0(w)) u_i \bigg)\, dx\\
\nonumber
  &\phantom{xx}{}+\sigma\eps\int_\Omega\rho^* \bm v \cdot \bm u \,dx
+\delta K(n)\int_\Omega |\bm v^*|^2 \bm v\cdot \bm u \,dx; \\
\label{lin.rhoe}
  0 &= \sigma\int_\Omega\bigg( -(\rho e)^* \bm v^* + \kappa^{(\eta)}(\theta^*)\theta^*
	\nabla w^* + \sum_{i=1}^N M_i^*\nabla q_i^* \bigg)\cdot\na\phi_0 \,dx \\
\nonumber
  &\phantom{xx}{}+ \sigma\int_\Omega(p^*\diver \bm v^*
	- \St^* : \nabla \bm v^*)\phi_0 \,dx
+ \sigma\int_{\pa\Omega} \alpha_2(\theta^*-\theta_0^*)\phi_0\,ds\\
\nonumber
  &\phantom{xx}{}+\xi\int_\Omega\left( (1+e^{w^*})(1+|\na w^*|^2)\na w \cdot\na\phi_0
+ w\phi_0\right) \,dx \\
\nonumber
  &\phantom{xx}{}-\sigma\eps
\int_\Omega\widetilde{\mathfrak{R}}(w^*)\frac{|\bm v^*|^2}{2}\phi_0 \, dx
+\delta\int_\Omega (1+e^{w^*}) D^2 w : D^2\phi_0 \,dx
\end{align}
for all $\phi_1,\ldots,\phi_N\in H^2(\Omega)$, $\bm u\in X_n$, and
$\phi_0\in H^2(\Omega)$.

Lax--Milgram lemma ensures the existence of a unique solution
$(\vec q,\bm v,w)\in H^2(\Omega ; \R^N)$ $\times X_n\times H^2(\Omega)$ to
\eqref{lin.rhoi}--\eqref{lin.rhoe}. More precisely, first we solve \eqref{lin.rhoe}
for $w$, then insert $w$ in \eqref{lin.rhoi}
and solve this equation for $\vec q$, and finally insert both $\vec q$
and $w$ in \eqref{lin.rhov} and solve this equation for $\bm v$.
As a consequence, we have the mapping
$$
  F:V\times [0,1]\to V, \quad ((\vec q\,^*,\bm v^*,w^*),\sigma)\mapsto
	(\vec q,\bm v,w).
$$
We aim to show that $F(\cdot,1)$ has a fixed point in $V$. Then this fixed point solves
\eqref{app.rhoi}--\eqref{app.rhoe}. To this end, we apply the
Leray--Schauder theorem.

The operator $F$ has the following properties:
\begin{enumerate}
\item $F(\cdot,0)$ is constant. Clearly, if $\sigma=0$ then $w=0$, which implies
that $\vec q=\widetilde{q}_0(0)\vec 1$, which in turn implies that $\vel=\mathbf{0}$.
\item $F : V\times [0,1]\to V$
is continuous. Standard arguments show that $F$ is sequentially continuous.
\item $F : V\times [0,1]\to V$
is compact. Indeed, the compactness of $F$ follows from the fact that the
image of $F$ is bounded in $V_0$ (consequence of Lax--Milgram's Lemma), the compact Sobolev embedding
$H^2(\Omega)\hookrightarrow W^{1,4}(\Omega)$ (which holds, e.g., for $d=3$),
and the fact that $X_n$ is finite dimensional.
\end{enumerate}
It remains to prove that the set
$$\{ (\vec q,\bm v,w)\in V\, : \,
F((\vec q,\bm v,w),\sigma) = (\vec q,\bm v,w) \} $$
is bounded in $V$ uniformly with respect to $\sigma\in [0,1]$.
If this is shown, the existence of a fixed point for $F(\cdot,1)$ in $V$ follows from
the Leray--Schauder theorem.

Let $(\vec q,\bm v,w)\in V_0$ satisfy
$F((\vec q,\bm v,w),\sigma) = (\vec q,\bm v,w)$ for some $\sigma\in [0,1]$.
This means that the following relations hold:
\begin{align}
\label{fix.rhoi}
  0 &= \sigma\sum_{i=1}^N\int_\Omega \Big( -\rho_i \bm v + \sum_{j=1}^N M_{ij}\na q_j -  M_i\frac{\nabla w}{\theta} \Big)\cdot\nabla\phi_i \,dx - \sigma\sum_{i=1}^N\int_\Omega r_i\phi_i \,dx \\
\nonumber
  &\phantom{xx}{}+\delta\sum_{i=1}^N\int_\Omega D^2 q_i : D^2\phi_i \,dx
+\sigma\eps\sum_{i=1}^N\int_\Omega
(\rho_i - \mathfrak{R}_i)\phi_i \,dx\\
\nonumber
  &\phantom{xx}{}+\eps^3 \sum_{i=1}^N\int_\Omega\left(
(1+\chi |\na q_i|^2)\na q_i \cdot\na\phi_i + (1+|q_i-q_0|^{1/2})(q_i-q_0) \phi_i \right) \,dx; \\
\label{fix.rhov}
  0 &= \int_{\pa\Omega}\alpha_1 \bm u\cdot \bm v
\, ds + \int_\Omega \St:\bm u \, dx +
\sigma\int_\Omega \left( -\rho \bm v\otimes \bm v\right):\nabla \bm u\, dx
-\sigma\int_\Omega (p\diver \bm u  + \rho \bm b \cdot \bm u) \, dx\\
\nonumber
  &\phantom{xx}{}+\eps^3\sum_{i=1}^3\sum_{j=1}^N\int_\Omega v_i\Big(
(1+\chi |\na q_j|^2)\nabla q_j\cdot\nabla u_i + \frac{1}{2} (1+|q_j-q_0|^{1/2})(q_j-q_0) u_i \Big)\, dx\\
\nonumber
  &\phantom{xx}{}+\sigma\eps\int_\Omega\rho \bm v \cdot \bm u \,dx
+\delta K(n)\int_\Omega |\bm v|^2 \bm v\cdot \bm u \,dx; \\
\label{fix.rhoe}
  0 &= \sigma\int_\Omega\Big( -\rho e \bm v + \kappa^{(\eta)}(\theta)\theta\nabla w + \sum_{i=1}^N M_i\nabla q_i \Big)\cdot\na\phi_0 \,dx
+ \sigma\int_\Omega(p\diver \bm v - \St : \nabla \bm v)\phi_0 \,dx\\
\nonumber
  &\phantom{xx}{}+ \sigma\int_{\pa\Omega} \alpha_2(\theta-\theta_0)\phi_0\,ds
+\delta\int_\Omega (1+e^{w}) D^2 w : D^2\phi_0 \,dx\\
\nonumber
  &\phantom{xx}{}-\sigma\eps\int_\Omega\mathfrak{R}\frac{|\bm v|^2}{2}\phi_0 \,dx
+\xi\int_\Omega\left( (1+e^{w})(1+|\na w|^2)\na w \cdot\na\phi_0
+ w\phi_0\right) \,dx
\end{align}
for all $\phi_1,\ldots,\phi_N\in H^2(\Omega)$, $\bm u\in X_n$, and
$\phi_0\in H^2(\Omega)$.
We need to show that $(\vec{q},\vel,w)$ can be uniformly bounded in $V$ with
respect to $\sigma$. To find such estimate,
we derive global entropy and energy inequalities.

\subsubsection{Global entropy inequality for approximate solutions.}
Let us choose $\phi_i = q_i-{q}_0$ ($i=1,\ldots,N$) and
$\phi_0 = -\exp(-w) = -1/\theta$ in \eqref{fix.rhoi} and \eqref{fix.rhoe},
respectively, and sum the equations.
By arguing like in the derivation of \eqref{2.ent}, we find that
\begin{align}\nonumber
\sigma\int_\Omega & \bigg(\sum_{i,j=1}^N M_{ij}\nabla q_i\cdot\nabla q_j
+ \kappa^{(\eta)}(\theta)|\nabla w|^2 -\sum_{i=1}^N r_i q_i
+\frac{1}{\theta} \St : \nabla \bm v\bigg)\,dx\\
\label{ei.1}
  &{}+\delta\sum_{i=1}^N \|D^2 q_i\|_{L^2(\Omega)}^2
+\eps^3 \sum_{i=1}^N\int_\Omega \big( \chi |\na q_i|^4 + |\na q_i|^2 +
|q_i-q_0|^{5/2} + (q_i-q_0)^2\big)\,dx\\
\nonumber
  &{}+\sigma\eps\int_\Omega\mathfrak{R}\frac{|\bm v|^2}{2\theta}\,dx
+\sigma\eps\int_\Omega
\sum_{i=1}^N (\rho_i - \mathfrak{R}_i)(q_i-q_0)\,dx  +  R
= \sigma\int_{\pa\Omega}\alpha_2\frac{\theta-\theta_0}{\theta}\,ds,
\end{align}
where we defined
\begin{align*}
R &= \int_\Omega (1+\theta)\Big(
\delta D^2 (\log\theta) : D^2 \Big(-\frac{1}{\theta}\Big) + \xi
(1+|\na\log\theta|^2)\na(\log\theta) \cdot\na \Big(-\frac{1}{\theta}\Big) \Big) \, dx \\
&\phantom{xx}{}+\xi\int_\Omega\frac{1}{\theta}\log\frac{1}{\theta}\,dx.
\end{align*}
We deduce from the construction of $q_0$ and $\mathfrak R_i$ (Lemma
\ref{lem.refrho}) and the convexity of $h_\theta^{(\eta)}$ that
\begin{equation*}
\sum_{i=1}^N (\rho_i - \mathfrak{R}_i)(q_i-q_0) =
\frac{1}{\theta}\sum_{i=1}^N (\rho_i - \mathfrak{R}_i)\bigg(
\frac{\pa h_\theta^{(\eta)}}{\pa\rho_i}(\vec\rho) -
\frac{\pa h_\theta^{(\eta)}}{\pa\rho_i}(\vec{\mathfrak{R}}) \bigg)\geq 0.
\end{equation*}
Furthermore, straightforward computations and the assumption $\delta\leq\xi/2$ show that
\begin{equation}\label{est.R}
R \geq c\delta\int_\Omega (1+\theta)\theta^{-3}|D^2\theta|^2 \,dx
+ \frac{\xi}{2}\int_\Omega\bigg( (1+\theta^{-1})|\na\log\theta|^4
+\frac{1}{\theta}\log\frac{1}{\theta} \bigg)\,dx
\end{equation}
for some constant $c>0$.
Putting \eqref{ei.1}--\eqref{est.R} together gives the
{\em global entropy inequality}: 
\begin{align}\label{ei.fix}
\sigma & \int_\Omega\left(
\sum_{i,j=1}^N M_{ij}\nabla q_i\cdot\nabla q_j
+ \kappa^{(\eta)}(\theta)|\nabla\log\theta|^2
-\sum_{i=1}^N r_i q_i + \frac{1}{\theta} \St : \nabla \bm v
\right)\,dx\\
\nonumber
&{}+\eps^3\sum_{i=1}^N\int_\Omega (\chi |\na q_i|^4 + |\na q_i|^2 + |q_i-q_0|^{5/2} + (q_i-q_0)^2)\,dx
+ \delta\int_\Omega (1+\theta)\theta^{-3}|D^2\theta|^2 \,dx\\
\nonumber
&{}+\delta\sum_{i=1}^N \|D^2 q_i\|_{L^2(\Omega)}^2
+ \xi \int_\Omega\left( (1+\theta^{-1})|\na\log\theta|^4
+\frac{1}{\theta}\log\frac{1}{\theta} \right)\,dx
+\sigma\int_{\pa\Omega}\alpha_2\frac{\theta_0-\theta}{\theta}\,ds
\leq 0 .
\end{align}

\subsubsection{(Almost) mass conservation}
Choosing $\phi_i=1$ ($i=1,\ldots,N$) in \eqref{fix.rhoi} leads to
\begin{align*}
\sigma\int_\Omega\sum_{i=1}^N (\rho_i - \mathfrak{R}_i)\,dx =
\eps^2\int_\Omega\sum_{i=1}^N (1+|q_i-q_0|^{1/2})(q_0-q_i) \,dx .
\end{align*}
However, $\int_\Omega\sum_{i=1}^N \mathfrak{R}_i dx = |\Omega|\overline\rho$
by construction (see Lemma \ref{lem.refrho}), thus \eqref{ei.fix} shows that
\begin{equation}\label{mass.fix}
\sigma\bigg| \frac{1}{\Omega}\int_\Omega\rho \,dx - \overline{\rho}
\bigg|\leq C\eps^{1/5} .
\end{equation}

\subsubsection{Global total energy inequality for approximate solutions}
We choose now $\phi_i=-\frac{1}{2}|\bm v|^2$ $(i=1,\ldots,N)$, $\bm u=\bm v$,
$\phi_0=1$ in \eqref{fix.rhoi}--\eqref{fix.rhoe} and sum the equations.
By arguing like in the derivation of \eqref{2.ener}, we obtain
\begin{align}\label{energy.2}
\int_{\pa\Omega}&\alpha_1 |\bm v|^2 \,ds +
\sigma\int_{\pa\Omega}\alpha_2(\theta-\theta_0)\,ds
+\xi\int_\Omega \log\theta \,dx
+\sigma\eps\int_\Omega
\rho\frac{|\bm v|^2}{2} \,dx + (1-\sigma) \int_\Omega \St:\nabla \bm v\, dx\\
\nonumber
&= \sigma\int_\Omega \rho \bm b \cdot \bm v \, dx
-\delta K(n)\int_\Omega |\bm v|^4 \,dx
+\delta\int_\Omega D^2\bigg(\sum_{i=1}^N q_i\bigg): D^2\frac{|\bm v|^2}{2}\,dx.
\end{align}
It follows from \eqref{ei.fix} that
$$
  -\xi\int_\Omega\log\theta \,dx
  \le \xi\int_{\{\theta<1\}}\log\frac{1}{\theta}\,dx
  \le \xi\int_\Omega\frac{1}{\theta}\log\frac{1}{\theta} \,dx \leq C ,
$$
while the Cauchy--Schwarz inequality, definition \eqref{def.Kn}, and
\eqref{ei.fix} yield
\begin{align*}
\delta\bigg|\int_\Omega & D^2\bigg(\sum_{i=1}^N q_i\bigg): D^2\frac{|\bm v|^2}{2}
\,dx\bigg|
\leq \frac{\delta}{2}\sum_{i=1}^N\|D^2 q_i\|_{L^2(\Omega)}
\|D^2 (|\bm v|^2)\|_{L^2(\Omega)}\\
&\leq
\frac{\delta}{2} \sqrt{K(n)}\sum_{i=1}^N\|D^2 q_i\|_{L^2(\Omega)}
\|\bm v\|_{L^4(\Omega)}^{2}
\leq \frac{\delta}{8}\sum_{i=1}^N\|D^2 q_i\|_{L^2(\Omega)}^2
+ \frac{\delta}{2}K(n)\int_\Omega |\bm v|^4 \,dx \\
&\leq C + \frac{\delta}{2}K(n)\int_\Omega |\bm v|^4 \,dx .
\end{align*}
Therefore, \eqref{energy.2} leads to the {\em total energy inequality}
(for approximate solutions):
\begin{align}\label{energy.fix}
\int_{\pa\Omega}\alpha_1 |\bm v|^2 \,ds
&+\sigma\int_{\pa\Omega}\alpha_2(\theta-\theta_0)\,ds
+\frac{\delta K(n)}{2}\int_\Omega |\bm v|^4 \,dx\\
\nonumber
&{}+\sigma\eps\int_\Omega
\rho\frac{|\bm v|^2}{2} \,dx
\leq C + \sigma\int_\Omega \rho \bm b \cdot \bm v \, dx .
\end{align}
The right-hand side of \eqref{energy.fix} is estimated as
$$
\sigma\bigg|\int_\Omega \rho \bm b \cdot \bm v \, dx\bigg|
\leq C\sigma\int_\Omega \rho |\bm v| \, dx\leq
C\frac{\sigma}{\eps}\int_\Omega\rho \,dx
+ \frac{\sigma\eps}{4}\int_\Omega\rho |\bm v|^2 \,dx,
$$
which, together with \eqref{mass.fix}, \eqref{energy.fix}, gives the estimate
\begin{equation}\label{4.23a}
\int_{\pa\Omega}\alpha_1 |\bm v|^2 \,ds
+\sigma\int_{\pa\Omega}\alpha_2(\theta-\theta_0)\,ds
+\frac{\delta K(n)}{2}\int_\Omega |\bm v|^4 \,dx
+\sigma\eps\int_\Omega
\rho\frac{|\bm v|^2}{2} \,dx
\leq C(1+\eps^{-1}).
\end{equation}

\subsubsection{Conclusion: existence for the approximate problem}
Inequalities \eqref{ei.fix}, \eqref{mass.fix}, and \eqref{4.23a}
yield a $\sigma$-uniform bound for $(\vec q,\bm v,\log\theta)$ in the space
$V = W^{1,4}(\Omega; \R^N)\times X_n \times W^{1,4}(\Omega)$.
Leray--Schauder's fixed-point theorem implies the existence of a fixed point
$(\vec q,\bm v,\log\theta)$ for $F(\cdot,1)$, that is, a solution
$(\vec q,\bm v,\log\theta)\in H^2(\Omega;\R^N)\times X_n \times H^2(\Omega)$
to \eqref{app.rhoi}--\eqref{app.rhoe}. Moreover, estimates \eqref{ei.fix},
\eqref{mass.fix}, and \eqref{4.23a} hold with $\sigma=1$.


\subsection{Uniform estimates for approximate solutions}

An analogous version of Lemma \ref{lem.est1} holds also for the approximate solutions.
The proof employs \eqref{ei.fix}, \eqref{energy.fix} with $\sigma=1$ in place of
\eqref{2.ent}, and \eqref{2.ener} and is basically identical to the proof of Lemma
\ref{lem.est1}.
In short, \eqref{2.vq}--\eqref{2.thetaH1} are fulfilled also by the approximate
solutions built in the previous subsection. Here, we state some additional
estimates that are satisfied by the approximate solutions and that follow immediately
from \eqref{ei.fix}, \eqref{energy.fix}, and Lemma \ref{lem.refrho}:
\begin{align}
\label{tmp.est.1}
\|D^2 \vec q\|_{L^2(\Omega)} &\leq C\delta^{-1/2},\\
\label{tmp.est.2}
\|\nabla\vec q\|_{L^4(\Omega)} &\leq C\eps^{-3/4}\chi^{-1/4},\\
\label{tmp.est.2b}
\|\nabla\vec q\|_{L^2(\Omega)} +
\|\vec q - q_0\vec{1}\|_{L^2(\Omega)} &\leq C\eps^{-3/2},\\
\label{tmp.est.2c}
\|\vec q - q_0\vec 1\|_{L^{5/2}(\Omega)} & \leq C\eps^{-6/5},\\
\nonumber 
\|\theta^{-1/2}D^2\log\theta\|_{L^2(\Omega)}
+\|D^2\log\theta\|_{L^2(\Omega)} &\leq C\delta^{-1/2},\\
\label{tmp.est.4}
\|\theta^{-1/4}\nabla\log\theta\|_{L^4(\Omega)}
+\|\log\theta\|_{W^{1,4}(\Omega)} &\leq C\xi^{-1/4},\\
\nonumber 
\|\bm v\|_{L^4(\Omega)} &\leq C |\delta K(n)|^{-1/4} ,\\
\label{tmp.est.6}
\left\|\sqrt\rho \bm v\right\|_{L^2(\Omega)}
 &\leq C\eps^{-1/2} ,\quad i=1,\ldots,N,\\
\nonumber 
|q_0| &\leq C(\eps),\\
\label{tmp.est.8}
\|\mathfrak{R}_i\|_{L^1(\Omega)} &\leq \overline{\rho} ,\quad i=1,\ldots,N,\\
\label{tmp.est.9}
\|\theta^{L/2}\|_{H^1(\Omega)} &\leq C\eta^{-1/2}.
\end{align}

We prove an estimate for $\na\vec\rho\in L^2(\Omega;\R^{N\times 3})$.
On the one hand, Young's inequality gives
$$
\sum_{i=1}^N\na\rho_i\cdot\na\mu_i =
\sum_{i=1}^N\na\rho_i\cdot (\theta\na q_i + q_i\na\theta)
\leq \frac{\eta}{2}|\na\vec\rho|^2 + \frac{\|\theta\|_{L^\infty(\Omega)}^2}{2\eta}
\big(|\na\vec q|^2 + |\vec q|^2 |\na\log\theta|^2\big).
$$
On the other hand, we have
\begin{align*}
  \sum_{i=1}^N & \na\rho_i\cdot\na\mu_i =
  \sum_{i=1}^N\na\rho_i\cdot\left(
  \sum_{j=1}^N\frac{\pa^2 h_\theta^{(\eta)}}{\pa\rho_i\pa\rho_j}\na\rho_j
  +\frac{\pa^2 h_\theta^{(\eta)}}{\pa\rho_i\pa\theta}\na\theta \right)\\
  &\geq \sum_{i,j=1}^N \frac{\pa^2 h_\theta^{(\eta)}}{\pa\rho_i\pa\rho_j}
	\na\rho_i\cdot\na\rho_j
  -\frac{\eta}{2}\sum_{i=1}^N\rho_i^{-2\alpha_0}|\na\rho_i|^2
  -\frac{|\na\log\theta|^2}{2\eta}\|\theta\|_{L^\infty(\Omega)}^2
  \sum_{i=1}^N \rho_i^{2\alpha_0}
	\bigg|\frac{\pa^2 h_\theta^{(\eta)}}{\pa\rho_i\pa\theta}\bigg|^2 .
\end{align*}
We conclude from the these inequalities, Hypothesis (H6), more precisely \eqref{H6.b},
and \eqref{h.reg} that
\begin{align*}
  \eta & |\na\vec\rho|^2
  +\eta\sum_{i=1}^N\rho_i^{-2\alpha_0}|\na\rho_i|^2\\
  &\leq
  C(\|\theta\|_{L^\infty(\Omega)})\frac{|\na\log\theta|^2}{\eta}
  (1+|\vec\rho|^{2\beta_0})
  +C\frac{\|\theta\|_{L^\infty(\Omega)}^2}{\eta}(
  |\na\vec q|^2 + |\vec q|^2 |\na\log\theta|^2) .
\end{align*}
As $\vec{q}$ and $\log \theta$ are essentially bounded functions, also $\vec \rho$
is essentially bounded (this bound is uniform with respect to $\delta$ and $n$).
Therefore, we deduce from the previous estimate and bound \eqref{tmp.est.4} that
$$
  \|\nabla\vec\rho\|_{L^2(\Omega)}^2\leq C(\xi,\eta)\big(1+
  \|\na\vec q\|_{L^2(\Omega)}^2 + \|\vec q - q_0\vec 1\|_{L^4(\Omega)}^2
  +q_0^2 \big) .
$$
In view of Lemma \ref{lem.refrho} and the $\eps$-uniform $L^\infty(\Omega)$ bound
for $\log\theta$, $q_0 = \widetilde{q}_0(\log\theta)$ is bounded in $\eps$, so
we infer from \eqref{tmp.est.2b} the desired gradient bound for $\vec\rho$:
\begin{align}\label{narho.eps}
\|\nabla\vec\rho\|_{L^2(\Omega)}\leq C\eps^{-3/2},
\end{align}
where the constant $C>0$ is independent of $\eps$.


\section{Weak sequential compactness for approximate solutions}
\label{sec_conv}

\subsection{Limits $\delta\to 0$, $n\to\infty$}

The bounds \eqref{2.vq}--\eqref{2.thetaH1} and \eqref{tmp.est.1}--\eqref{tmp.est.8}
allow us to extract subsequences (not relabeled) such that, as $\delta\to 0$,
\begin{align*}
  &\vec q^{(\delta)}\rightharpoonup \vec q\quad \mbox{weakly in }W^{1,4}(\Omega;\R^N),
	\quad\vec q^{(\delta)}\to\vec q\quad \mbox{strongly in }L^{\infty}(\Omega;\R^N),\\
  &\log\theta^{(\delta)}\rightharpoonup \log\theta\quad \mbox{weakly in }
	W^{1,4}(\Omega), \quad
  \log\theta^{(\delta)}\to \log\theta\quad\mbox{strongly in }L^{\infty}(\Omega),\\
  &\bm v^{(\delta)}\rightharpoonup \bm v\quad \mbox{weakly in }H^1(\Omega;\R^3),
	\quad\bm v^{(\delta)}\to \bm v\quad\mbox{strongly in }L^{6-\nu}(\Omega;\R^3)
	\mbox{ for all }\nu>0,\\
  &\delta |D^2\vec q^{(\delta)}| + \delta K(n) |\bm v^{(\delta)}|^3
	+ \delta (1+\theta^{(\delta)}) |D^2\log\theta^{(\delta)}| \to 0\quad
  \mbox{strongly in }L^{4/3}(\Omega).
\end{align*}
The $L^\infty$ bound of $\log \theta^{(\delta)}$ implies that $\theta^{(\delta)}$
is bounded away from zero and infinity. The $W^{1,4}(\Omega)$ bound yields the
strong convergence of (a subsequence of) $\theta^{(\delta)}$ and thus the
convergence $\log\theta^{(\delta)}\to \log\theta$ a.e.\ and uniformly.
As a consequence of the strong $L^\infty(\Omega)$ convergence of $\vec q^{(\delta)}$
and $\log\theta^{(\delta)}$, we have
$$
  \vec\rho^{(\delta)}\to\vec\rho\quad\mbox{strongly in }L^\infty(\Omega;\R^N),
	\quad\min_{i=1,\ldots,N}\mbox{ess} \inf_\Omega\rho_i > 0.
$$
The previous convergences are sufficient to pass to the limit $\delta\to 0$ in
\eqref{app.rhoi}--\eqref{app.rhoe} and to obtain
an analogous system of equations without the terms proportional to $\delta$.
The terms that are nonlinear in $\na\vec q$, $\na\log\theta$ are monotone,
so we can use Minty's monotonicity technique to identify the limit of these terms.

The limit $n\to\infty$ is carried out in the same way as the limit $\delta\to 0$.
We have chosen not to take the two limits simultaneously because of the term
$\delta K(n)\int_{\Omega}|\bm v|^2 \bm v\cdot \bm u\, dx$ in \eqref{fix.rhov},
which vanishes as $\delta\to 0$ ($n\in\N$ fixed) but not necessarily when
both $\delta\to 0$, $n\to\infty$. The two limits could be carried out
simultaneously provided that one assumes that $\delta K(n)\to 0$.

After taking the limits $\delta\to 0$, $n\to\infty$ (in this order),
we are left with the following system:
\begin{align}
\label{app2.rhoi}
  0 &= \sum_{i=1}^N\int_\Omega \bigg( -\rho_i \bm v + \sum_{j=1}^N M_{ij}\na q_j
	+ M_i\nabla\bigg(\frac{1}{\theta} \bigg) \bigg)\cdot\nabla\phi_i \,dx
	- \sum_{i=1}^N\int_\Omega r_i\phi_i \,dx \\
\nonumber
  &\phantom{xx}{}+\eps^3 \sum_{i=1}^N\int_\Omega\big(
(1+\chi |\na q_i|^2)\na q_i \cdot\na\phi_i
+ (1+|q_i-q_0|^{1/2})(q_i-q_0) \phi_i \big) \,dx\\
\nonumber
  &\phantom{xx}{}+\eps\sum_{i=1}^N\int_\Omega (\rho_i - \mathfrak{R}_i)
\phi_i \,dx; \\
\label{app2.rhov}
  0 &= \int_{\pa\Omega}\alpha_1 \bm u\cdot \bm v \,ds +
\int_\Omega \left( -\rho \bm v\otimes \bm v + \St \right):\nabla \bm u\, dx -
\int_\Omega (p\diver \bm u  + \rho \bm b \cdot \bm u) \,dx\\
\nonumber
  &\phantom{xx}{}+\eps^3\sum_{i=1}^d\sum_{j=1}^N\int_\Omega v_i\bigg(
(1+\chi |\na q_j|^2)\nabla q_j \cdot\nabla u_i + \frac{1}{2}(1+|q_j-q_0|^{1/2})
(q_j-q_0) u_i \bigg)\,dx\\
\nonumber
  &\phantom{xx}{}+\eps\int_\Omega\rho \bm v \cdot \bm u \,dx; \\
\label{app2.rhoe}
  0 &= \int_\Omega\bigg( -\rho e \bm v + \kappa(\theta)\nabla\theta
	+ \sum_{i=1}^N M_i\nabla q_i \bigg)\cdot\na\phi_0 \,dx
	+ \int_\Omega(p\diver \bm v - \St : \nabla \bm v)\phi_0 \,dx\\
\nonumber
  &\phantom{xx}{}+ \int_{\pa\Omega} \alpha_2(\theta-\theta_0)\phi_0\,ds
-\eps\int_\Omega\mathfrak{R}\frac{|\bm v|^2}{2}\phi_0 \,dx\\
\nonumber
  &\phantom{xx}{}+\eps\int_\Omega\big( (1+\theta)(1+|\na\log\theta|^2)\na(\log\theta )
	\cdot\na\phi_0 + (\log\theta)\phi_0\big) \,dx
\end{align}
for all $\phi_1,\ldots,\phi_N\in W^{1,4}(\Omega)$, $\bm u\in W^{1,12}(\Omega;\R^3)$,
and $\phi_0\in W^{1,4}(\Omega)$.


\subsection{Limit $\chi\to 0$.}

Since the estimate of $\nabla \vec{q}\,^{(\chi)}$ in $L^4(\Omega;\R^{N\times 3})$
blows up when $\chi \to 0$, we lose the control of $\vec{\rho}$ in
$L^\infty(\Omega;\R^N)$. Therefore, we need to establish $\chi$-independent estimates
of $\rho$, which is the goal of the following lemma.

\subsubsection{Estimate for $\rho$ via the Bogovskii operator}
The following result is the analogue of Lemma \ref{lem.dens} for the approximate equations.

\begin{lemma}\label{lem.prhonu}
Let $L$, $\eta$ be as in \eqref{h.reg} and $0<\zeta\leq L/11$. It holds that
\begin{equation}
  \int_\Omega\rho^{\gamma+\zeta}dx + \eta\int_\Omega\rho^{L+\zeta}dx
	\leq C(1+ \eps^{(L+\zeta)/(2L)}), \label{est.p.rho.nu}
\end{equation}
where $C>0$ is independent of $\eta$ and $\eps$.
\end{lemma}

\begin{proof}
Let $\B$ be the Bogovskii operator (see Theorem \ref{thm.bog}).
Choosing $\bm u = \B(\rho^\zeta - \langle\rho^\zeta\rangle)$ in \eqref{app2.rhov}
and carrying out the computations contained in Lemma \ref{lem.dens}, it follows that
\begin{align}\label{intprho.1}
  &\int_\Omega p\rho^\zeta \,dx
	\leq  C(1+\|\rho\|_{L^{\gamma+\zeta}(\Omega)}^\lambda) + G, \quad\mbox{where} \\
  \nonumber
  & G = \eps^3\sum_{i=1}^d\sum_{j=1}^N\int_\Omega v_i
	\bigg((1+\chi |\na q_j|^2)\nabla q_j \cdot\nabla u_i
	+ \frac{1}{2}(1+|q_i-q_0|^{1/2})(q_j-q_0) u_i \bigg)\,dx \\
  \nonumber
  &\phantom{G=} +\eps\int_\Omega\rho \bm v \cdot \bm u \,dx .
\end{align}
Note that due to our regularized problem, the computations in Lemma \ref{lem.dens}
are performed for the choice $\beta=\gamma=L$, hence it leads to the restriction
$\zeta \leq (3L-2)L/(3L+2)$.

The terms in $G$ remain to be estimated. It holds that
\begin{align*}
  |G| &\leq C\eps^3\chi\|\bm v\|_{L^6(\Omega)}\|\nabla\vec q\|_{L^4(\Omega)}^3
	\|\nabla \bm u\|_{L^{12}(\Omega)}
  +C\eps^3\|\bm v\|_{L^6(\Omega)}\|\nabla\vec q\|_{L^2(\Omega)}\|\nabla
	\bm u\|_{L^{3}(\Omega)} \\
  &\phantom{xx}{} + C\eps^3\|\bm v\|_{L^6(\Omega)}\|\vec q - q_0\vec 1\|_{L^2(\Omega)}
	\|\bm u\|_{L^3(\Omega)}
+ C\eps^3\|\bm v\|_{L^6(\Omega)}\|\vec q - q_0\vec 1\|_{L^{5/2}(\Omega)}^{3/2}
\|\bm u\|_{L^{30/7}(\Omega)}\\
  &\phantom{xx}{} + C\eps\|\sqrt\rho\|_{L^2(\Omega)}\|\sqrt\rho v\|_{L^2(\Omega)}
	\|\bm u\|_{L^\infty(\Omega)}.
\end{align*}
From \eqref{tmp.est.2}--\eqref{tmp.est.2c}, \eqref{tmp.est.6}, quasi mass conservation
\eqref{mass.fix}, as well as the uniform $H^1$ bound \eqref{2.vq} for $\bm v$ and
Sobolev's embedding $H^1(\Omega)\hookrightarrow L^6(\Omega)$, we infer that
\begin{align*}
|G| &\leq C\eps^{3/4}\chi^{1/4}\|\nabla \bm u\|_{L^{12}(\Omega)} +
C\eps^{3/2}\|\nabla \bm u\|_{L^{3}(\Omega)} +
 C\eps^{3/2}\|\bm u\|_{L^3(\Omega)} + C\eps^{1/2}\|\bm u\|_{L^\infty(\Omega)}\\
&\leq C\eps^{1/2}\|\bm u\|_{W^{1,12}(\Omega)} .
\end{align*}
It follows from Theorem \ref{thm.bog} that
$\|\bm u\|_{W^{1,12}(\Omega)}\leq C\|\rho^\zeta - \langle\rho^\zeta\rangle
\|_{L^{12}(\Omega)}$, such that
$$
  |G|\leq C\eps^{1/2}\big(1 + \|\rho\|_{L^{12\zeta}(\Omega)}^\zeta\big).
$$
Therefore, \eqref{intprho.1} implies that
$$
  \int_\Omega p\rho^\zeta\, dx \leq  C(1+\|\rho\|_{L^{\gamma+\zeta}(\Omega)}^\lambda)
	+ C\eps^{1/2}\|\rho\|_{L^{12\zeta}(\Omega)}^\zeta.
$$
Since $p\geq c\rho^{\gamma} + c\eta\rho^{L}$, while $\lambda<\gamma+\zeta$, we find that
$$
  \int_\Omega\rho^{\gamma+\zeta}dx + \eta\int_\Omega\rho^{L+\zeta}dx\leq
  C + C\eps^{1/2}\left( \int_\Omega\rho^{L+\zeta}dx \right)^{\frac{\zeta}{L+\zeta}},
$$
which leads to \eqref{est.p.rho.nu}.
This finishes the proof of the lemma.
\end{proof}

\begin{remark}\label{r grad.dens}\rm
Due to our assumption $L\geq 2\beta_0$, we still control $\nabla \vec{\rho}$
in $L^2(\Omega;\R^{N\times2})$ independently of $\chi$; we use the argument
between \eqref{tmp.est.9} and \eqref{narho.eps}.
\end{remark}

\subsubsection{Limit passage in the equations}
 At this point, since we have a $(\eps,\chi)$-uniform bound for $\rho$ (given by
\eqref{est.p.rho.nu} and Remark \ref{r grad.dens}), we can take the limit $\chi\to 0$.
Moreover, we still have a uniform $H^1(\Omega)$ bound for $\vec q$ and a uniform
$W^{1,4}(\Omega)$ bound for $\log\theta$, so we deduce via compact Sobolev embedding that for $\chi\to 0$ (up to subsequences),
\begin{align*}
  \vec q^{(\chi)}\to\vec q &\quad\mbox{strongly in }L^{6-\nu}(\Omega;\R^N)
  \mbox{ for all }\nu>0,\\
  \log\theta^{(\chi)}\to\log\theta &\quad\mbox{strongly in }L^\infty(\Omega).
\end{align*}
As before, it is not difficult to justify that $\theta^{(\chi)}\to\theta$ a.e.\ in
$\Omega$. We deduce from the continuity of the free energy $h_\theta$ that
$\vec\rho^{(\chi)}$ is also a.e.~convergent in $\Omega$. Bound \eqref{est.p.rho.nu}
then implies that
$$
  \vec\rho^{(\chi)}\to\vec\rho\quad\mbox{strongly in }L^r(\Omega;\R^N)
	\quad\mbox{for all } r< \frac{12}{11}L.
$$
The limit $\chi\to 0$ in \eqref{app2.rhoi}--\eqref{app2.rhoe} is carried out in a
similar way as the previous limits $\delta\to 0$, $n\to\infty$. This leads
to the following limiting system:
\begin{align}
\label{app3.rhoi}
  0 &= \sum_{i=1}^N\int_\Omega \bigg( -\rho_i \bm v + \sum_{j=1}^N M_{ij}\na q_j
	+ M_i\nabla\bigg(\frac{1}{\theta} \bigg) \bigg)\cdot\nabla\phi_i \,dx
	- \sum_{i=1}^N\int_\Omega r_i\phi_i \,dx \\
\nonumber
  &\phantom{xx}{}+\eps^3 \sum_{i=1}^N\int_\Omega\big(\na q_i \cdot\na\phi_i
  + (1+|q_i-q_0|^{1/2})(q_i-q_0) \phi_i \big) \,dx
  +\eps\sum_{i=1}^N\int_\Omega (\rho_i - \mathfrak{R}_i)\phi_i \,dx; \\
\label{app3.rhov}
  0 &= \int_{\pa\Omega}\alpha_1 \bm u\cdot \bm v\,ds +
  \int_\Omega( -\rho \bm v\otimes \bm v + \St):\nabla\bm u\, dx
	- \int_\Omega (p\diver \bm u  + \rho \bm b \cdot \bm u) \,dx\\
\nonumber
  &\phantom{xx}{}+\eps^3\sum_{i=1}^d\sum_{j=1}^N\int_\Omega v_i\bigg(
  \nabla q_j \cdot\nabla u_i + \frac{1}{2}(1+|q_j-q_0|^{1/2})(q_j-q_0) u_i \bigg)\,dx
  +\eps\int_\Omega\rho \bm v \cdot \bm u \,dx; \\
\label{app3.rhoe}
  0 &= \int_\Omega\bigg( -\rho e \bm v + \kappa(\theta)\nabla\theta
	+ \sum_{i=1}^N M_i\nabla q_i \bigg)\cdot\na\phi_0 \,dx
	+ \int_\Omega(p\diver \bm v - \St : \nabla \bm v)\phi_0 \,dx\\
\nonumber
  &\phantom{xx}{}+ \int_{\pa\Omega} \alpha_2(\theta-\theta_0)\phi_0\,ds
  -\eps\int_\Omega\mathfrak{R}\frac{|\bm v|^2}{2}\,dx\\
  \nonumber
  &\phantom{xx}{}+\xi\int_\Omega\left( (1+\theta)(1+|\na\log\theta|^2)\na(\log\theta )
	\cdot\na\phi_0 + (\log\theta)\phi_0\right) \,dx
\end{align}
for all $\phi_1,\ldots,\phi_N\in H^{1}(\Omega)\cap L^\infty(\Omega)$,
$\bm u\in W^{1,3}(\Omega;\R^3)$, and $\phi_0\in W^{1,4}(\Omega)$.

\subsubsection{Entropy and total energy balance equations}
We will derive the balance equations for the entropy and total energy. It was not
possible to derive the latter equation as long as the $p$-Laplacian regularization
in terms of $\vec q$ was included in the mass densities equation.

Let $\psi\in W^{1,\infty}(\Omega)$ with $\psi\geq 0$ be arbitrary.
By choosing $\phi_i=(q_i-q_0)\psi$ in \eqref{app3.rhoi} and $\phi_0=-\frac{1}{\theta}
\psi$ in \eqref{app3.rhoe} and proceeding as in the derivation of \eqref{ei.fix},
we obtain
\begin{align}\nonumber
  \int_\Omega \bigg( & \rho s \vel + \sum_{i=1}^N q_i \bigg( \sum_{j=1}^N M_{ij}\na q_j
	+ M_i\na\frac{1}{\theta} \bigg) - \frac{1}{\theta}\bigg( \kappa(\theta)\na\theta
	+ \sum_{i=1}^N M_i\na q_i\bigg)\bigg) \cdot\nabla\psi\, dx\\
\label{entr.bal}
  &\phantom{xx}{}+\int_\Omega\bigg(\sum_{i,j=1}^N M_{ij}\nabla q_i\cdot\nabla q_j
  + \kappa^{(\eta)}(\theta)|\nabla\log\theta|^2
  -\sum_{i=1}^N r_i q_i + \frac{1}{\theta} \St : \nabla \bm v\bigg)\psi \,dx\\
\nonumber
  &\phantom{xx}{}+\eps^3\int_\Omega\sum_{i=1}^N\na q_i\cdot (q_i-q_0)\na\psi \,dx
  -\xi\int_\Omega (1+\theta^{-1})(1+|\na\log\theta|^2)\na\log\theta\cdot\na\psi \,dx\\
  \nonumber
  &= \int_{\pa\Omega}\alpha_2\frac{\theta-\theta_0}{\theta}\psi \,ds
  - \xi \int_\Omega\frac{1}{\theta}\log\frac{1}{\theta}\psi \,dx .
\end{align}

Next, let $\varphi\in W^{1,\infty}(\Omega)$. By choosing $\phi_i=-\frac12|\bm v|^2
\varphi$ in \eqref{app3.rhoi}, $\bm u=\bm v\varphi$ in \eqref{app3.rhov},
and $\phi_0=\varphi$ in \eqref{app3.rhoe} and
proceeding as in the derivation of \eqref{energy.fix}, it follows that
\begin{align}
\label{toten.bal}
\int_\Omega( & -\rhotot E \bm{v} - \bm{Q} +\St \bm{v} -p\bm{v})\cdot\na \varphi \, dx
+\int_{\pa\Omega}(\alpha_1|\bm{v}|^2+ \alpha_2 (\theta-\theta_0))\varphi\, ds\\
\nonumber
&{}+\xi\int_\Omega (1+\theta)(1+|\na\log\theta|^2)\na(\log\theta )
\cdot\na\varphi \,dx\\
\nonumber
&{}+\eps\int_\Omega \rho\frac{|\bm v|^2}{2}\varphi \,dx
+\xi\int_\Omega\log\theta \varphi \,dx
= \int_\Omega\rhotot\bm{b}\cdot\bm{v} \varphi\, dx.
\end{align}


\subsection{Limit $\eps\to 0$}

When taking the limit $\eps\to 0$, we can argue as in Subsection
\ref{sec.conv.smooth} and prove that $\bm v^{(\eps)}$ is (up to a subsequence)
strongly convergent in $L^{6-\delta}(\Omega;\R^3)$, $\theta^{(\eps)}$ is strongly
convergent in $L^{3\beta-\delta}(\Omega)$ for all $\delta>0$, and $\rho^{(\eps)}$
is weakly convergent in $L^{L+\nu}(\Omega)$.
On the other hand, $\log\theta^{(\eps)}$ is strongly convergent in
$L^\infty(\Omega)$, thanks to \eqref{tmp.est.4}, which is better than the
convergence property of $\theta$ in Subsection \ref{sec.conv.smooth}.
As in  Subsection \ref{sec.conv.smooth}, we can also show the convergence of the
various terms appearing in equations \eqref{app3.rhoi}, \eqref{app3.rhov},
\eqref{entr.bal}, and \eqref{toten.bal}, which do not involve the regularization.
In this subsection, we only show that the additional $\eps$-dependent regularizing
terms vanish in the limit $\eps\to 0$.

Taking into account estimates \eqref{tmp.est.2b}, \eqref{tmp.est.2c},
\eqref{tmp.est.8} and the uniform $L^1(\Omega)$ bounds for $\rho$, we see that
for every $\phi_1,\ldots,\phi_N\in H^1(\Omega)\cap L^\infty(\Omega)$,
$$
\eps^3 \sum_{i=1}^N\int_\Omega\left(\na q_i \cdot\na\phi_i
+ (1+|q_i-q_0|^{1/2})(q_i-q_0) \phi_i \right) \,dx
+\eps\sum_{i=1}^N\int_\Omega (\rho_i - \mathfrak{R}_i)
\phi_i \,dx\to 0
$$
as $\eps\to 0$. Furthermore, it follows from \eqref{tmp.est.2b}, \eqref{tmp.est.2c}
and the uniform bounds for $\rho$ and $\bm v$ in $L^{6/5}(\Omega)$ and $L^{6}(\Omega)$,
respectively, that, for every $\bm u\in W^{1,3}(\Omega)$,
$$
\eps^3\sum_{i=1}^d\sum_{j=1}^N\int_\Omega v_i\bigg(
\nabla q_j \cdot\nabla u_i + \frac{1}{2}(1+|q_j-q_0|^{1/2})(q_j-q_0) u_i \bigg)\,dx
+\eps\int_\Omega\rho \bm v \cdot \bm u \,dx\to 0
$$
as $\eps\to 0$. Estimates \eqref{tmp.est.2b} and \eqref{tmp.est.2c} imply that
$$
\eps^3\int_\Omega\sum_{i=1}^N\na q_i\cdot (q_i-q_0)\na\psi \,dx\to 0
\quad\mbox{as }\eps\to 0.
$$
Finally, the uniform $L^{L+\nu}(\Omega)$ bound for $\rho$ and the $L^6(\Omega;\R^3)$
bound for $\bm v$ give
\begin{align*}
\eps\int_\Omega
\rho\frac{|\bm v|^2}{2}\varphi \,dx\to 0\quad
\mbox{as }\eps\to 0.
\end{align*}

\subsubsection{Effective viscous flux}
We want to show that Lemma \ref{lem.eff} holds for approximate solutions.
For this, we only need to check that the div-curl Lemma can be applied.
It follows from \eqref{app3.rhov} that, in the distributional sense,
\begin{align*}
  \diver(&\rho^{(\eps)} \bm v^{(\eps)}\otimes \bm v^{(\eps)} - \mathbb{T}^{(\eps)})
	- \rho^{(\eps)}\bm b^{(\eps)} \\
  &= \eps^3\diver\bigg(\bm v^{(\eps)} \otimes \nabla\sum_{j=1}^N q_j\bigg)
  -\frac{1}{2}\eps^3\sum_{j=1}^N (1+|q_i^{(\eps)}-q_0^{(\eps)}|^{1/2})
  (q_i^{(\eps)}-q_0^{(\eps)})v^{(\eps)}-\eps\rho^{(\eps)}\bm v^{(\eps)} .
\end{align*}
By arguing as in the previous subsection, we see that the right-hand side of this
identity tends to zero strongly in $W^{-1,3/2}(\Omega)$. Hence,
$\diver(\rho^{(\eps)} \bm v^{(\eps)}\otimes \bm v^{(\eps)} - \mathbb{T}^{(\eps)})$
is relatively compact in $W^{-1,r}(\Omega)$ for some $r>1$.

We claim that $\diver(\rho^{(\eps)}\bm v^{(\eps)})$ is relatively compact in
$W^{-1,r}(\Omega)$ for some $r>1$. Indeed, we infer from \eqref{app3.rhoi} that
$$
\diver(\rho^{(\eps)}\bm v^{(\eps)}) =
\eps^3 \sum_{i=1}^N\Delta q_i
-\eps^3 \sum_{i=1}^N (1+|q_i-q_0|^{1/2})(q_i-q_0)
-\eps (\rho - \mathfrak{R}) .
$$
Once again, by arguing as in the previous subsection, we find that
$\diver(\rho^{(\eps)}\bm v^{(\eps)})\to 0$ strongly in $W^{-1,3/2-\delta}(\Omega)$
for every $\delta>0$.
We conclude that Lemma \ref{lem.eff} effectively holds for the approximate solutions constructed in the previous subsections.

\subsubsection{Estimates based on the convexity of the free energy}
Lemmata \ref{lem.Wk}--\ref{lem.Qk} also hold for the approximate solutions,
as they only require the properties of the free energy stated in the introduction,
which still hold true for the regularized free energy introduced in \eqref{h.reg},
and the estimates in Lemmata \ref{lem.est1}--\ref{lem.p}, which are satisfied by the
approximate solutions.

\subsubsection{Renormalized continuity equation.}
We prove that the weak limit $\rho$ of $(\rho^{(\eps)})$ satisfies the
renormalized continuity equation \eqref{ren.cont.eq} by mimicking
the proof of Lemma \ref{lem.renorm}. Note that the choice of the test function below can be, due to higher regularity of $\rho^{(\varepsilon)}$, justified by 
approximation.
Choosing $\phi_i = T_k'(\rho^{(\eps)})\psi$ ($i=1,\ldots,N$),
$\psi\in W^{1,\infty}(\Omega)$  in \eqref{app3.rhoi}
(this is possible since $\rho^{(\eps)}\in H^1(\Omega)$ thanks to
\eqref{narho.eps}) yields
\begin{equation}\label{rce.1}
  -\int_\Omega T_k(\rhotot^{(\eps)})\vel^{(\eps)}\cdot\na\psi \,dx
  +\int_\Omega\big(T_k'(\rhotot^{(\eps)})\rhotot^{(\eps)}
  -T_k(\rhotot^{(\eps)})\big)\diver\vel^{(\eps)}\psi \, dx =
  R^{(\eps)}_k(\psi),
\end{equation}
where
\begin{align*}
R^{(\eps)}_k(\psi) &=
-\eps^3 \sum_{i=1}^N\int_\Omega\left(\na q_i^{(\eps)} \cdot\na (T_k'(\rho^{(\eps)})\psi)
+ (1+|q_i^{(\eps)}-q_0^{(\eps)}|^{1/2})(q_i^{(\eps)}-q_0^{(\eps)})T_k'(\rho^{(\eps)})
\psi \right) \,dx\\
&\phantom{xx}{}-\eps\sum_{i=1}^N\int_\Omega (\rho_i^{(\eps)} - \mathfrak{R}_i^{(\eps)})
T_k'(\rho^{(\eps)})\psi \,dx .
\end{align*}
Taking the limit inferior $\eps\to 0$ in both sides of \eqref{rce.1} leads to
\begin{equation}\label{rce.2}
  \bigg|-\int_\Omega \overline{T_k(\rhotot^{(\eps)})}\vel\cdot\na\psi \,dx
  +\int_\Omega\overline{\big(T_k'(\rhotot^{(\eps)})\rhotot^{(\eps)}
  -T_k(\rhotot^{(\eps)})\big)\diver\vel^{(\eps)}}\psi \, dx\bigg|\leq
  \liminf_{\eps\to 0}|R^{(\eps)}_k(\psi)|.
\end{equation}
Thanks to estimates \eqref{tmp.est.2b} and \eqref{tmp.est.2c}, we infer that
\begin{align*}
  \liminf_{\eps\to 0}|R^{(\eps)}_k(\psi)|
  &\leq \liminf_{\eps\to 0}\eps^3\int_\Omega |\na\vec{q}^{(\eps)}||\na\rho^{(\eps)}|
	|T_k''(\rho^{(\eps)})||\psi| \,dx\\
  &\leq \frac{C}{k}\liminf_{\eps\to 0}\eps^3\int_{\{ k\leq\rho^{(\eps)}\leq 3k\}}
	|\na\vec{q}^{(\eps)}||\na\rho^{(\eps)}| |\psi| \,dx .
\end{align*}
We replace $\psi$ by $b'(\overline{T_k(\rho^{(\eps)})})\psi$ in \eqref{rce.2}.
Similarly as above, we have to justify this choice by approximation; note that all integrals below are finite.
We exploit the definition of the truncation operator $T_k$ to
find that
\begin{align}\label{rce.3}
  \bigg| -&\int_{\Omega}
  b\big(\overline{T_k(\rhotot^{(\eps)})}\big) \vel\cdot \nabla \psi\, dx
	+ \int_{\Omega}\diver \vel \Big(\overline{T_k(\rhotot^{(\eps)})} b'\big(
	\overline{T_k(\rhotot^{(\eps)})}\big)-b\big(\overline{T_k(\rhotot^{(\eps)})}\big)
	\Big)\psi\, dx \\
\nonumber
  &\phantom{xx}{} - \int_\Omega b'\big(\overline{T_k(\rhotot^{(\eps)})}\big)
  \overline{\big(T_k(\rhotot^{(\eps)})-\rhotot^{(\eps)} T_k'(\rhotot^{(\eps)})\big)
	\diver\vel^{(\eps)}}\psi \, dx\bigg|\\
\nonumber
  &\leq \frac{C}{k}\liminf_{\eps\to 0}\eps^3\int_{\{ k\leq\rho^{(\eps)}\leq 3k\}}
	|\na\vec{q}^{(\eps)}||\na\rho^{(\eps)}|
  \big|b'(\overline{T_k(\rho^{(\eps)})})\big| |\psi| \,dx .
\end{align}
It is proved in Lemma \ref{lem.renorm} that the left-hand side of
\eqref{rce.3} converges to the left-hand side of the renormalized continuity
equation \eqref{ren.cont.eq} as $k\to\infty$. Therefore, it is sufficient to prove
that the right-hand side of \eqref{rce.3} tends to zero as $k\to\infty$. Indeed,
since $b'$ is bounded and \eqref{tmp.est.2b} and \eqref{narho.eps} hold,
it follows that
\begin{align*}
  \liminf_{\eps\to 0}\,&\eps^3\int_{\{ k\leq\rho^{(\eps)}\leq 3k\}}
	|\na\vec{q}^{(\eps)}||\na\rho^{(\eps)}|
  |b'(\overline{T_k(\rho^{(\eps)})})| |\psi| \,dx\\
  &\leq \liminf_{\eps\to 0}\eps^3\|\na\vec{q}^{(\eps)}\|_{L^2(\Omega)}
  \|\na\rho^{(\eps)}\|_{L^2(\Omega)}\leq C ,
\end{align*}
meaning that the right-hand side of \eqref{rce.3} tends to zero as $k\to\infty$.
We conclude that the weak limit $(\rho,\vel)$ of $(\rho^{(\eps)},\vel^{(\eps)})$
satisfies the renormalized continuity equation \eqref{ren.cont.eq}.

\subsubsection{Limit $\eps\to 0$: conclusion}
Lemmata \ref{thm.conv} and \ref{thm.conv.total} hold as they are not influenced
by the approximation. We conclude that $\rho_i^{(\eps)}\to\rho_i$ a.e.~in $\Omega$.


\subsection{Limit $\xi\to 0$}

Due to the uniform bounds in Lemmata
\ref{lem.est1} and \ref{lem.prhonu} (fulfilled by the approximating sequence),
we deduce that, up to subsequences, $\Pi\vec{q}^{(\xi)}\to\Pi\vec{q}$,
$\log\theta^{(\xi)}\to\log\theta$, $(\theta^{(\xi)})^{\beta/2}\to\theta^{\beta/2}$
weakly in $H^1(\Omega)$ and strongly in $L^{6-\delta}(\Omega)$ for every $\delta>0$,
and $\rho^{(\xi)}\rightharpoonup\rho$ weakly in $L^{L+\nu}(\Omega)$, as
$\xi\to 0$. Moreover,
the a.e.~convergence of $\rho^{(\xi)}$ is proved in an analogous way as in the
compactness part, since no $\xi$-dependent regularizing terms appear in either the
linear momentum or the partial mass densities equations. We only need to show
that the regularizing terms in the entropy balance equation \eqref{entr.bal} and
in the total energy balance equation \eqref{toten.bal} vanish in the limit $\xi\to 0$.

To this end, we note that, for any $\psi\in W^{1,\infty}(\Omega)$ with $\psi\geq 0$,
we have
$$
\limsup_{\xi\to 0}\bigg( - \xi \int_\Omega
\frac{1}{\theta}\log\frac{1}{\theta}\psi \,dx \bigg) \leq
\limsup_{\xi\to 0}\xi \int_{\{\theta\geq 1\}}
\frac{\log\theta}{\theta}\psi \,dx = 0 .
$$
Next, we point out that \eqref{tmp.est.4} and \eqref{2.thetaL2} imply that
$$
\|\theta^{-1/4}\|_{W^{1,4}(\Omega)}\leq C\xi^{-1/4},
$$
which, thanks to Gagliardo--Nirenberg inequality and the $L^1(\Omega)$ bound for
$1/\theta$, leads to
$$
\| \theta^{-1/4} \|_{L^\infty(\Omega)}\leq
C\| \theta^{-1/4} \|_{L^4(\Omega)}^{1/4}
\| \theta^{-1/4} \|_{W^{1,4}(\Omega)}^{3/4}
\leq C\xi^{-3/16} .
$$
We deduce from this estimate and \eqref{tmp.est.4} that
$$
\xi\int_\Omega (1+\theta^{-1})(1+|\na\log\theta|^2)\na\log\theta\cdot\na\psi \,dx\to 0
\quad\mbox{as }\xi\to 0.
$$
It follows from \eqref{tmp.est.9}, Sobolev's embedding, and assumption
$L\geq\gamma+\nu>4/3$ that
$$
\|\theta\|_{L^{4}(\Omega)}^{L/2}\leq
C\|\theta\|_{L^{3L}(\Omega)}^{L/2} =
C\|\theta^{L/2}\|_{L^6(\Omega)}\leq C\|\theta^{L/2}\|_{H^1(\Omega)}\leq C(\eta) .
$$
Together with \eqref{tmp.est.4} and the uniform $L^2(\Omega)$ bound for $\log\theta$,
this imply that
$$
\xi\int_\Omega (1+\theta)(1+|\na\log\theta|^2)\na\log\theta \cdot\na\varphi \,dx
+\xi\int_\Omega\log\theta \varphi \,dx
\to 0\qquad\mbox{as }\xi\to 0.
$$
This finishes the part related to the limit $\xi\to 0$.


\subsection{Limit $\eta\to 0$}

As in the previous limit, it holds (up to subsequences) that
$\Pi\vec{q}^{(\eta)}\to\Pi\vec{q}$, $\log\theta^{(\eta)}\to\log\theta$,
$(\theta^{(\eta)})^{\beta/2}\to\theta^{\beta/2}$
weakly in $H^1(\Omega)$ and strongly in $L^{6-\delta}(\Omega)$ for any $\delta>0$
as $\eta\to 0$.
Lemma \ref{lem.dens} can be shown to hold true as in the compactness part,
since the regularizing terms in the linear momentum equations have been removed in
the limit $\eps\to 0$.
The proof of the a.e.~convergence of $\rho^{(\eta)}$
is identical to that one in the compactness
part; in particular, Lemma \ref{thm.conv.total} holds, as it only employs the
properties of the free energy stated in Hypothesis (H6).

We show that the regularizing term, coming from the heat conductivity
$\kappa^{(\eta)}$ defined in \eqref{reg.kappa}, vanishes in the entropy and
total energy balance equations in the limit $\eta\to 0$. The strongest of these
terms is $\eta (\theta^{(\eta)})^L\na\theta^{(\eta)}$, which appears in the
total energy balance. It follows from
\eqref{tmp.est.9} and \eqref{2.thetaL2}--\eqref{2.thetaH1} via the
Gagliardo--Nirenberg inequality and assumption $L\geq 3\beta-2$ that, as $\eta\to 0$,
\begin{align*}
\eta\int_\Omega (\theta^{(\eta)})^L |\na\theta^{(\eta)}| \,dx
&\leq C\eta\|(\theta^{(\eta)})^{(L+2)/2}\|_{L^2(\Omega)}
\|\na(\theta^{(\eta)})^{L/2}\|_{L^2(\Omega)}\\
&\leq C\eta\|(\theta^{(\eta)})^{\beta/2}\|_{L^2(\Omega)}^{\lambda (L+2)/\beta}\|(\theta^{(\eta)})^{L/2}\|_{H^1(\Omega)}^{1+(1-\lambda)(L+2)/2}\to 0,
\end{align*}
where $\lambda = 2\beta(L-1)/((L-\beta)(L+2))\in [0,1]$
satisfies $1+(1-\lambda)(L+2)/2 < 2$ if and only if $\beta>1$. This is the
condition under which the total energy balance can be obtained.
On the other hand, it is possible to see that the term
$\eta (\theta^{(\eta)})^L\na\log\theta^{(\eta)}$,
which appears in the entropy balance, tends to zero strongly as $\eta\to 0$
for $\beta>2/3$.

Concerning the regularization in the mass densities,
Lemma \ref{lem.prhonu} yields (remember that we have already taken the limit
$\eps\to 0$) for some $\zeta>0$,
\begin{equation*}
\int_\Omega\rho^{L+\zeta}\,dx\leq C\eta^{-1}.
\end{equation*}
As a consequence, the regularizing terms in the pressure $p$ and internal energy
density $\rho e$ tend to zero strongly in $L^1(\Omega)$ as $\eta\to 0$.
The entropy density $\rho s$ does not contain regularizing terms.
Finally, let us turn our attention to the chemical potentials
$$
\mu_i^{(\eta)} = \pa_{\rho_i}h_{\theta^{(\eta)}}(\vec\rho^{(\eta)})
+\eta L \frac{(\overline{n}^{(\eta)})^{L-1}}{m_i}
+\eta f_{\alpha_0}'(\rho_i^{(\eta)}) + 2\eta \rho_i^{(\eta)},
\quad i=1,\ldots,N.
$$
Clearly, $\eta L (\overline{n}^{(\eta)})^{L-1}/m_i + 2\eta \rho_i^{(\eta)}\to 0$
strongly in $L^1(\Omega)$. The only problematic term is
$\eta f_{\alpha_0}'(\rho_i^{(\eta)})$, as it might be singular for $\rho_i\to 0$.
Since $\alpha_0<1$, it follows that $s\mapsto sf_{\alpha_0}'(s)$, $\R^+\to\R$ has a finite limit
for $s\to 0^+$, so $\rho_i^{(\eta)}\mu_i^{(\eta)}$ is a.e.~convergent in $\Omega$.
We know from \eqref{aeposit} that, up to sets of measure zero,
$\cup_{i=1}^N\{\rho_i=0\} = \{\rho=0\}$, which implies (together with
the strong convergence of $\vec\rho^{(\eta)}$, $\theta^{(\eta)}$) that
the a.e.~limit $\vec\mu$ of $\vec\mu^{(\eta)}$ is equal to
$\na h_\theta(\vec\rho)$ on $\{\rho>0\}$.
This finishes the proof of Theorem \ref{thm.ex}.

\begin{remark}[Integral of source term in partial mass balance]\label{last_remark}\rm
By integrating the partial mass balance equation \eqref{1.massbal} and taking
into account the boundary conditions \eqref{bc.J}, \eqref{bc.vel},
we find that $\int_\Omega r_i dx = 0$ for $i=1,\ldots,N$. This seemingly additional
constraint is, in fact, satisfied by the solution constructed in Theorem \ref{thm.ex}.
Indeed, for $j=1,\ldots,N$, choose $\phi_i = \delta_{ij}$ (being the
Kronecker delta) in \eqref{app.rhoi} and use estimates \eqref{mass.fix},
\eqref{tmp.est.2b}, \eqref{tmp.est.2c}, and \eqref{tmp.est.8}.
A straightforward computation yields
$$
\bigg| \int_\Omega r_j(\Pi(\vec\mu^{(\eps)}/\theta^{(\eps)}),\theta^{(\eps)}) \,dx
\bigg| \leq C(\eps + \eps^{6/5}) \leq C\eps .
$$
The strong convergence of $\Pi(\vec\mu^{(\eps)}/\theta^{(\eps)})$,
$\theta^{(\eps)}$ and Hypothesis (H5) imply that $\int_\Omega r_j(\Pi(\vec\mu/\theta)$,
$\theta) \,dx = 0$ in the limit $\eps\to 0$.
\qed
\end{remark}


\begin{appendix}

\section{Auxiliary results}\label{app.aux}

For the convenience of the reader, we collect some results needed in this paper.
For the first lemma, which is proved in \cite[Theorem 10.11]{NoSt04},
we introduce the space
$$
  L_{0}^{p}(\Omega) = \bigg\{u\in L^p(\Omega):\int_\Omega u \, dx=0\bigg\},
	\quad 1<p<\infty.
$$

\begin{theorem}[Bogovskii operator]\label{thm.bog}
Let $\Omega\subset\R^d$ ($d\ge 2$) be a bounded domain with Lipschitz boundary
and let $1<p<\infty$. Then there exists a bounded linear operator
$\B:L_0^p(\Omega)\to W_0^{1,p}(\Omega)$ such that
$\diver\B(u)=u$ in the sense of distributions  for all $u\in L_0^p(\Omega)$, and
there exists $C>0$ such that for all $u\in L_0^p(\Omega)$,
$$
  \|\B(u)\|_{W^{1,p}(\Omega)} \le C\|u\|_{L^p(\Omega)}.
$$
\end{theorem}

Next, we recall the div-curl lemma, developed by Murat and Tartar and proved in,
e.g., \cite[Theorem 10.21]{FeNo09}.
We introduce the curl of a vector-valued function
$\bm{w}=(w_j)$ by setting $(\operatorname{curl}\bm{w})_{ij}
=\pa w_j/\pa x_i-\pa w_i/\pa x_j$.

\begin{lemma}[Div-curl lemma]\label{lem.divcurl}
Let $\Omega\subset\R^n$ be an open set, $1/p+1/q=1/r<1$, and $s>1$.
Let $\bm{u}_\delta\rightharpoonup \bm{u}$ weakly in $L^p(\Omega;\R^n)$
and $\bm{w}_\delta\rightharpoonup \bm{w}$ weakly in $L^q(\Omega;\R^n)$ as $\delta\to 0$.
If $(\diver\bm{u}_\delta)$ is precompact in $W^{-1,s}(\Omega)$ and
$(\operatorname{curl}\bm{w}_\delta)$ is precompact in
$W^{-1,s}(\Omega;\R^{n\times n})$, then
$$
  \bm{u}_\delta\cdot\bm{w}_\delta\rightharpoonup\bm{u}\cdot\bm{w}
	\quad\mbox{weakly in }L^r(\Omega)\quad\mbox{as }\delta\to 0.
$$
\end{lemma}

We need the following version of the Korn inequality, proved e.g. in
\cite[Lemma 2.1]{JNP14}.

\begin{lemma}[Korn's inequality]\label{lem.korn}
Let $\Omega\subset\R^d$ ($d\ge 2$) be a bounded, not axially symmetric
domain with Lipschitz boundary. Let the stress tensor fulfil Hypothesis (H3). Then there exists $C>0$, only depending on $d$
and $\Omega$, such that for all $\bm{u}\in W^{1,2}(\Omega;\R^d)$ with
$\bm{u}\cdot \vcg{\nu}=0$ on $\pa \Omega$,
$$
  \|\bm{u}\|^2_{W^{1,2}(\Omega)} \le C \int_{\Omega} \frac{\St(\theta,\na\bm{u})
	:\na\bm{u}}{\theta}\, dx, \quad
  \|\bm{u}\|^2_{W^{1,2}(\Omega)} \le C \int_{\Omega} \St(\theta,\na\bm{u})
	:\na\bm{u}\, dx.
$$
\end{lemma}

The following result is a slight generalization of Fatou's lemma.

\begin{lemma}\label{lem.fatou}
Let $\Omega\subset\R^d$ ($d\ge 1$) be a bounded domain, and let $({\mathbb M}_k)\subset
L^\infty(\Omega;\R^{N\times N})$ be a sequence of symmetric, positive semi-definite
matrices and $(\vec {v_k})\subset L^2(\Omega;\R^N)$ be such that
$$
  \vec{v}_k\rightharpoonup \vec{v}\quad\mbox{weakly in }L^2(\Omega;\R^N), \quad
	{\mathbb M}_k\to {\mathbb M}\quad\mbox{a.e. in }\Omega
$$
as $k\to\infty$. Then
$$
  \liminf_{k\to\infty}\int_\Omega \vec{v}_k\cdot {\mathbb M}_k \vec{v}_k \, dx\ge \int_\Omega \vec{v}\cdot {\mathbb M} \vec{v} \, dx.
$$
\end{lemma}

\begin{proof}
Since $({\mathbb M}_k)$ converges a.e., Egorov's theorem implies that
there exists $\Omega^{(1)}\subset\Omega$ such that
$|\Omega\setminus\Omega^{(1)}|<1/2$ and ${\mathbb M}_k\to {\mathbb M}$ uniformly in
$\Omega^{(1)}$. We consider
\begin{align*}
  \int_\Omega \vec{v}_k\cdot {\mathbb M}_k\vec{v}_k \, dx
	&\ge \int_{\Omega^{(1)}}\vec{v}_k\cdot {\mathbb M}_k\vec{v}_k \, dx
	= \int_{\Omega^{(1)}}\vec{v}_k\cdot({\mathbb M}_k-{\mathbb M})\vec{v}_k \, dx
	+ \int_{\Omega^{(1)}}\vec{v}_k\cdot {\mathbb M}\vec{v}_k \, dx.
\end{align*}
As ${\mathbb M}_k\to {\mathbb M}$ uniformly in $\Omega^{(1)}$ and
$(\vec{v}_k)$ is bounded in $L^2(\Omega;\R^N)$, it follows that
\begin{equation}\label{a.Omega}
  \liminf_{k\to\infty}\int_\Omega \vec{v}_k\cdot {\mathbb M}_k\vec{v}_k \, dx
	\ge \liminf_{k\to\infty}\int_{\Omega^{(1)}}\vec{v}_k\cdot {\mathbb M}\vec{v}_k \, dx
	\ge \int_{\Omega^{(1)}} \vec{v}\cdot {\mathbb M}\vec{v} \, dx.
\end{equation}
The last inequality is a consequence of the weak lower semicontinuity
of strongly continuous convex functionals. Applying Egorov's theorem
again, there exists $\Omega^{(1)}\subset\Omega^{(2)}\subset\Omega$
such that $|\Omega\setminus\Omega^{(2)}|<1/4$ and
${\mathbb M}_k\to {\mathbb M}$ uniformly in $\Omega^{(2)}$. Then \eqref{a.Omega} also
holds with $\Omega^{(1)}$ replaced by $\Omega^{(2)}$.
We obtain an increasing sequence $\Omega^{(m)}$ of subsets of $\Omega$
such that $|\Omega\setminus\Omega^{(m)}|< 2^{-m}$ and
${\mathbb M}_k\to {\mathbb M}$ uniformly in $\Omega^{(m)}$. Replacing $\Omega^{(1)}$ by
$\Omega^{(m)}$ in \eqref{a.Omega} and taking the supremum for $m\ge 1$ leads to
$$
  \liminf_{k\to\infty}\int_\Omega \vec{v}_k\cdot {\mathbb M}_k\vec{v}_k \, dx
	\ge \sup_{m\ge 1}\int_{\Omega^{(m)}} \vec{v}\cdot {\mathbb M}\vec{v} \, dx.
$$
Since $\vec{v}\cdot {\mathbb M}\vec{v}\ge 0$ a.e.\ in $\Omega$ and $(\Omega^{(m)})$
is an increasing sequence of subsets of $\Omega$ such that
$\lim_{m\to\infty}|\Omega\setminus\Omega^{(m)}| = 0$, we conclude the proof.
\end{proof}


\section{Example for the free energy}\label{app.energy}

We show that the free energy density
\begin{equation*}
  h_\theta(\vec\rho) = \theta\sum_{i=1}^N\frac{\rho_i}{m_i}\log\frac{\rho_i}{m_i}
	+ \bigg(\sum_{i=1}^n\frac{\rho_i}{m_i}\bigg)^\gamma - c_W\theta\log\theta,
\end{equation*}
where $\gamma>1$, satisfies Hypothesis (H6). We compute the partial derivatives
\begin{align*}
  \mu_i &= \pa_i h_\theta(\vec\rho)
	= \frac{\theta}{m_i}\bigg(1+\log\frac{\rho_i}{m_i}\bigg)
	+ \frac{\gamma}{m_i}\bigg(\sum_{j=1}^N\frac{\rho_j}{m_j}\bigg)^{\gamma-1}, \\
  \pa_{ij}^2 h_\theta(\vec\rho) &= \frac{\gamma(\gamma-1)}{m_im_j}
	\bigg(\sum_{k=1}^N\frac{\rho_k}{m_k}\bigg)^{\gamma-2} + \frac{\theta}{\rho_im_i}
	\delta_{ij},\\
	\pa_\theta \pa_i h_\theta(\vec{\rho}) &= \frac{1}{m_i} \bigg(1+\log\frac{\rho_i}{m_i}\bigg),
\end{align*}
where $\pa_i=\pa/\pa\rho_i$, $\pa_{ij}^2=\pa^2/(\pa\rho_i\pa\rho_j)$, and
$\pa_\theta=\pa/\pa\theta$. Clearly, the Hessian 
$(\pa_{ij}^2 h_\theta(\vec\rho))_{i,j}$ is positive definite for 
$\vec{\rho}\in \R^N_+$ and $\theta \in \R_+$. It follows directly from the 
above relations that \eqref{LMB1}--\eqref{LMB21} hold. Moreover, 
$h_\theta(\vec{\rho})$ is smooth with respect to $\theta$ in this region and 
satisfies growth conditions \eqref{H6.b}--\eqref{grow_h_1} as well as 
\eqref{H7C}. Hence, it remains to verify \eqref{AUXc}.

We write the Hessian $D^2h_\theta(\vec\rho)$ as the sum of the two matrices
$\mathbb{A}$ and $\mathbb{B}$ with coefficients
$$
  \mathbb{A}_{ij} = \frac{\gamma(\gamma-1)}{m_im_j}
	\bigg(\sum_{j=1}^N\frac{\rho_j}{m_j}\bigg)^{\gamma-2}, \quad
	\mathbb{B}_{ij} = \frac{\theta\delta_{ij}}{\rho_im_i}.
$$
We can write 
\begin{equation*}
  D^2h_\theta(\vec\rho)^{-1}= (\mathbb{A}+\mathbb{B})^{-1}
	= \mathbb{B}^{-1/2}\big(\mathbb{I} + \mathbb{B}^{-1/2}\mathbb{A}\mathbb{B}^{-1/2}
	\big)^{-1}\mathbb{B}^{-1/2}.
\end{equation*}
It follows that
$$
\sum_{i=1}^N\mathbb{B}^{-1/2}_{ji}\pa_\theta \pa_i h_\theta(\vec{\rho})=\sum_{i=1}^N\delta_{ij}\sqrt{\frac{\rho_im_i}{\theta}}\frac{1}{m_i} \bigg(1+\log\frac{\rho_i}{m_i}\bigg)=\sqrt{\frac{\rho_j}{\theta m_j}} \bigg(1+\log\frac{\rho_j}{m_j}\bigg).
$$
Thus, since $\theta\in (\kappa, \kappa^{-1})$ and $\rho\le \kappa^{-1}$, we see that $\sum_{i=1}^N\mathbb{B}^{-1/2}_{ji}\pa_\theta \pa_i h_\theta(\vec{\rho})$ is bounded as well as the matrices $\mathbb{B}^{-1/2}$ and $(\mathbb{I} + \mathbb{B}^{-1/2}\mathbb{A}\mathbb{B}^{-1/2})^{-1}$. We infer that \eqref{AUXc} holds, concluding the proof..
\end{appendix}


\section*{Data Availability Statement}

No data has been used in this work.


\end{document}